\documentclass[10pt]{article}

\usepackage{graphics}
\usepackage{graphicx}

\usepackage[a4paper, left=35mm,right=35mm,top=34mm,bottom=34mm]{geometry}
\usepackage[utf8]{inputenc}
\usepackage[T1]{fontenc}
\usepackage[english]{babel}

\usepackage{enumerate}
\usepackage{graphicx}
\usepackage{hyperref}
\hypersetup{
    colorlinks=true,
    linkcolor=blue,
    filecolor=magenta,      
    urlcolor=cyan,
}
\usepackage{lipsum}
\usepackage{amsfonts}
\usepackage{graphicx}
\usepackage{epstopdf}
\usepackage{algorithmic}

 \usepackage{dsfont}
 \usepackage{psfrag}
 \usepackage{color}
 \usepackage{tikz}
 \usepackage{subfigure}
\usepackage{mathtools,amsthm,amssymb,amsmath}
\usepackage{bbm}
\usepackage{algorithm}
\usepackage{algorithmic}
\makeatother
\theoremstyle{plain}
\def\del  {\partial}
\def\eps{\varepsilon}

\def\R{\mathbb{R}}

\def\N{\mathbb{N}}

 \def\dx{{\rm d}x}
 \def\dy{{\rm d}y}

\newtheorem{remark}{\textbf{Remark}}
\newtheorem{theorem}{\textbf{Theorem}}
\newtheorem{lemma}{\textbf{Lemma}}
\newtheorem{definition}{\textbf{Definition}}
\newtheorem{proposition}{\textbf{Proposition}}
\newtheorem{corollary}{\textbf{Corollary}}

\usepackage{caption} 
\usepackage{fancyhdr}

\author{
  {\normalsize Adrien Dekkers}\thanks{CentraleSup\'elec, Universit\'e Paris-Saclay, France.}
  \and
  {\normalsize Anna Rozanova-Pierrat}\thanks{CentraleSup\'elec, Universit\'e Paris-Saclay, France
    (correspondence, anna.rozanova-pierrat@centralesupelec.fr).}
  		}
\title{Dirichlet boundary valued problems for linear and nonlinear wave equations on arbitrary and fractal domains}
\date{}

\begin{document}
\maketitle
\thispagestyle{fancy}

\begin{abstract}
\noindent   The weak well-posedness results of the strongly damped linear wave equation and of the non linear Westervelt equation with homogeneous Dirichlet boundary conditions  are proved on arbitrary three dimensional domains or any two dimensional domains which can be obtained by a limit of NTA domains caractarized by the same geometrical constants. The two dimensional result is obtained thanks to the Mosco convergence of the functionals corresponding to the weak formulations for the Westervelt equation with the homogeneous Dirichlet boundary condition.  The non homogeneous Dirichlet condition is also treated in the class of admissible domains composed on Sobolev extension domains of $\R^n$ with a $d$-set boundary $n-1\le d<n$ preserving Markov's local inequality.
The obtained Mosco convergence also alows  to approximate the solution of the Westervelt equation on an arbitrary domain by  solutions on a converging sequence of domains without additional conditions on their boundary regularity in $\R^3$, or on a converging sequence of NTA domains in $\R^2$. 
\end{abstract}


%

\section{Introduction}
The influence of the boundary regularity on a wave propagation is an important question, which is open for the wave interaction with an irregular and even fractal boundary. In this article, starting by discussing the linear models, as in particular the Poisson stationary equation
$$-\Delta u=f$$
and the strong damping wave equation
$$\partial^2_t u-c^2\Delta u-\eps\nu  \Delta\partial_t u=f, \quad \eps,c,\nu>0,$$
we are mainly interesting on the nonlinear model of Westervelt equation
$$\partial^2_t u-c^2\Delta u-\eps\nu  \Delta\partial_t u=\eps\alpha u \partial^2_t u+\eps\alpha (\partial_t u)^2+f,$$
 known to be able to describe the ultrasound propagation~\cite{DEKKERS-2019,DEKKERS-2020-1,DEKKERS-2020,Westervelt}. The propagation of the ultrasounds in a human body in the help of the imagery  techniques give us a necessity to pose boundary valued problems with boundaries of different regularity, even fractal, as soon it is known that the cancer tumors have irregular forms in the opposite of the health tissues. In this framework of bounded domains, we ask the question about the most general class of domains where it is possible to have the weak solution of the linear and nonlinear models with firstly Dirichlet homogeneous boundary condition and secondly with the non homogeneous one. 

Let us notice that the Westervelt equation classically~\cite{AANONSEN-1984,DEKKERS-2020-1,Westervelt} is presented for the perturbation of the pressure of the wave by
\begin{equation}\label{EWest}
	\del_t^2 p-c^2\Delta p - \eps\nu  \Delta  \del_t p=\eps\frac{\gamma+1}{c^2} \del_t p \del_t^2 p \quad \gamma>0,
\end{equation}
but for our study of its weak well-posedness in the most possible large class of domains we have modified a little bit the form of the nonlinear term. To do it, it sufficient to derive Eq.~(\ref{EWest}) once on time and  pose $u=\del_t p$. 
The coefficients $c$, $\nu$ and $\gamma$ are the physical constants describing the speed of the sound in the homogeneous non perturbed medium, the viscosity coefficient and the relation of the the ratio of the heat capacities at constant pressure and at constant volume respectively. The constant $\eps$ describe the size of  perturbations and of the viscosity effects~\cite{ROZANOVA-PIERRAT-2015,DEKKERS-2020-1,DEKKERS-2019} and physically is very small  to compare to $1$ (of order $10^{-5}$).  We  notice that the Westervelt equation is the nonlinear wave equation with the strong damping term $\eps\nu\Delta \del_t u$, which actually changes~\cite{Shibata} the finite speed propagation of the linear wave equation to the infinite as for the heat type equations. Indeed, the linear part of the Westervelt equation can be viewed as two compositions of the heat operator $\del_t-\Delta$ in the following way:
$$\del_t^2 u-c^2\Delta u- \eps\nu  \Delta \del_t u=\del_t(\del_t u - \eps\nu \Delta u)-c^2 \Delta u.$$

Using the spectral properties of the Dirichlet Laplacian known for an arbitrary domain~\cite{EVANS-1994}, by the usual Galerkin method it follows the weak well-posedness of the Dirichlet homogeneous problem for the linear wave equation in an arbitrary bounded domain~\cite{EVANS-1994}. We develop this method to obtain the analogous result for the strongly damped wave equation (see Section~\ref{secwpdampwavdirhom} and Ref.~\cite{DEKKERS-2020} for the analogous result for the mixed boundary conditions). 

The well-posedness results  and the regularity of the solutions of  the linear strongly damped wave equation and the non-linear Westervelt equation on regular domains, typically with at least a $C^2$ boundary, is well known~\cite{Kalt3,Kalt2,Kalt1,Kaltwer,Meyer}. The advantage to work with a such regular boundaries is to enjoy the fact that more the initial data are regular more the solution is regular up to the boundary. In particular, the $C^1$ boundary regularity allows to work with the Sobolev space $H^2$ in the domain of Laplacian~\cite{EVANS-1994}, but it is no more possible in a non convex case of a Lipschitz boundary~\cite{Grisvard}, where we only have access to $H^1$.
 
In the framework of weak solutions on an arbitrary bounded domain $\Omega\subset \R^n$, the homogeneous Dirichlet boundary valued problem for the Poisson equation is understood in the following variational form
\begin{equation}\label{EqWeekDirPois}
\forall v\in H^1_0(\Omega)\quad \int_{\Omega} \nabla u\nabla v\dx=\int_\Omega fv \dx,	
\end{equation}
in which there is no  any boundary influence and to obtain the unique weak solution $u\in H^1_0(\Omega)$ it suffices to apply the Riesz representation theorem.
Moreover, thanks to Evans~\cite{EVANS-2010} Theorem~2 p.~304 and Theorem~3 p.~316 we have (even for solutions in $H^1(\Omega)$ and thus for different boundary conditions) the  interior regularity of the weak solution, \textit{i.e.}, the fact that for a subset $V$ compactly included in $\Omega$, $V\subset\subset\Omega$, the solution on $\Omega$ has on $V$ the same regularity as for a domain with regular boundaries.  For instance, if $f\in C^\infty(\Omega)$ then $u\in C^\infty(\Omega)\cap H^1_0(\Omega)$. 
So, for any domain $\Omega$ with a boundary as ``bad'' as we want, a fractal, a fractal tree, a domain with cusps and ctr., the weak solution of~(\ref{EqWeekDirPois}) is in $C^\infty(\Omega)$ for the same regularity of $f$. The key point here that $\Omega$ is open.

The question is whether on less regular domains we can have a weak solution which is continuous or $C^1$ up to the boundary. The examples of Arendt and Elst~\cite{ArendtElst} show that problems appear for the definition of the trace as soon as the boundary is not $C^1$. 
The property to be in $C(\overline{\Omega})$ is much more restrictive than just to be $C(\Omega)$, since  the continuity on a compact requires from $u$ to be  bounded and equicontinuous, and does not hold for arbitrary shapes of $\del \Omega$~\cite{EDMUNDS-1987}. By  results of Nystr\"om~\cite{Nystrom} the necessary condition for $\Omega$ is to be a non-tangentially  accessible domain (a NTA domain, see Definition~\ref{defNTA} and also Ref.~\cite{Jerison}).
In addition Nystr\"om~\cite{Nystrom} showed that for von Koch's snowflake $\Omega$ (which is actually an example of a NTA domain) the weak solution $u\in H^1_0(\Omega)\cap C^{\infty}(\Omega)$ of~\eqref{EqWeekDirPois} for all $f\in \mathcal{D}(\Omega)$  non negative and non identically zero is continuous up to the boundary $u\in C(\overline{\Omega})$, but $u\notin H^2(\Omega).$  
However it holds for the convex polygonal domains~\cite{Grisvard}. The convexity condition does not allow the incoming angles, which creates the singularities.

An other  important question is whether the solutions of the Poisson problem with the homogeneous Dirichlet boundary condition belong to $C(\Omega)\cap L^{\infty}(\Omega) $ (a  weaker condition than to be continuous up to the boundary) with an estimate of the form:
\begin{equation}\label{EqCREch2}
	\Vert u\Vert_{L^{\infty}(\Omega)}\leq C \Vert f\Vert_{L^2(\Omega)}.
\end{equation}
By Nystr\"om~\cite{Nystrom}   the  answer is positive in dimension $n=2$ in the class of the NTA domains. By Xie~\cite{Xie} it is also positive for the three dimensional case considering the solutions of~(\ref{EqWeekDirPois}) in arbitrary domains. In Section~\ref{secpoissdir} we precise the constant dependence in estimates obtained by Nystr\"om~\cite{Nystrom}, which is important for the uniform boundness of a sequence of solutions $(u_m)_{m\in \N^*}$ independently on the shape of $\Omega_m$. This kind of uniform estimates are crucial to obtain the Mosco convergence for functionals coming from the definition of the weak solutions of the Westervelt equation (see Section~\ref{secasymptan}). 
In Section~\ref{SecWPW} we study the weak well posedness of the Westervelt equation.
As the method used to prove the weak well posedness of the Westervelt equation with the homogeneous Dirichlet boundary condition is based on the properties of the linear problem and in particular uses estimate~\eqref{EqCREch2}, we also obtain it for an arbitrary domain in $\R^3$ and for a NTA domain in $\R^2$ by applying the abstract theorem of nonlinear functional analysis of Sukhinin~\cite{Sukhinin} (see Theorem~\ref{thSuh}).  This theorem was previously successively applied in different frameworks for the Westervelt equation~\cite{DEKKERS-2019,DEKKERS-2020,DEKKERS-2020-1}, for the heat~\cite{ROZANOVA-2004} and the abstract~\cite{ROZANOVA-2004-1,ROSANOVA-2005} nonlinear equations.  The interest of its application is to be able to give a sharp estimate of the smallness of the initial data and of the source term, and in the same time to estimate the bound of the corresponding solution of the Westervelt equation in the space of solutions for the linear problem (see Theorem~\ref{ThWPWestGlob}).  Hence, we prove the existence and uniqueness of a global in time weak solution of the Westervelt equation on an arbitrary domain in $\R^3$ and on a NTA domain in $\R^2$ with the homogeneous Dirichlet condition. 
Thanks to the Mosco convergence of functionals defining the weak solutions of the Westervelt equation, in Section~\ref{secasymptan} we improve this well posedness result in $\R^2$ showing that  it holds  on all arbitrary domains of $\R^2$ which can be considered as a limit of a sequence of NTA domains with uniform geometrical constants (see Definition~\ref{defNTA} and Theorem~\ref{thmconvR3}).

But if we consider the nonhomogeneous Dirichlet boundary condition, for example for the Poisson equation
\begin{equation}\label{PoissonDir1}
 \left\lbrace
 \begin{array}{l}
 -\Delta u=f \hbox{ in }\Omega\\
 u\vert_{\Omega}=g \hbox{ on }\partial\Omega,
 \end{array}
 \right.
 \end{equation}
the notions as the trace and extension operators become very important. By the standard schema, assuming that there exists $g^*\in H^1(\Omega)$ such that the trace of $g^*$ to $\partial\Omega$ is $g$ (attention must be paid to the definition of the trace), we may prove with the Riesz representation theorem that given $f\in L^2(\Omega)$, $g^*\in H^1(\Omega)$, there exits a unique $u\in H^1(\Omega)$ such that $-\Delta u=f$ in the sense of distributions and $u-g^*\in H^1_0(\Omega)$. In other words we need to define the trace of an element of $H^1(\Omega)$ on the boundary and  to ensure the existence of such an extension $g^*$ satisfying $u=g$ on $\del \Omega$.
This turns in two necessary assumptions:
\begin{enumerate}
	\item $\Omega$ is the $H^1$-Sobolev extension domain~\cite{HAJLASZ-2008-1}, $i.e.$ there exists a bounded linear extension operator $E: H^1(\Omega) \to H^1(\R^n)$. This means that for all $u\in H^1(\Omega)$ there exists a $v=Eu\in  H^1(\R^n)$ with $v|_\Omega=u$ and it holds
 $$\|v\|_{H^1(\R^n)}\le C\|u\|_{H^1(\Omega)}\quad \hbox{with a constant } C>0.$$
 \item there exists a bounded linear trace operator $\operatorname{Tr}: H^1(\R^n)\to \operatorname{Im(Tr}(H^1(\R^n)))\subset L^2(\del \Omega)$ with a linear bounded right inverse operator $\operatorname{Ext}:  \operatorname{Im(Tr}(H^1(\R^n)))\to H^1(\R^n)$ such that $\operatorname{Tr}(\operatorname{Ext} g)=g$ on $\del \Omega$.
\end{enumerate}
By the composition of these two assumptions we obtain  that the trace operator $\operatorname{Tr}: H^1(\Omega)\to \operatorname{Im(Tr}(H^1(\Omega)))\subset L^2(\del \Omega)$ is linear and bounded with a linear bounded right inverse. A domain satisfying these two assumptions gives an example of  %
an admissible domain~\cite{ARFI-2017} or  of a Sobolev admissible domain~\cite{ROZANOVA-PIERRAT-2020} (see Definition~\ref{defadmisdomain}).

For the well-posedness of the nonhomogeneous Dirichlet boundary valued problem for the Poisson equation  if $\partial\Omega$ is regular enough, see for example  Raviart-Thomas~\cite{Raviart}), for the Lipschitz boundary see  Marschall~\cite{Marschall}, for the case of bounded  $(\eps,\infty)$-domains with a $d$-set boundary ($n-2<d<n$) see Jonsson-Wallin~\cite{JONSSON-1997} and for more general admissible domains also requiring a $d$-set boundary see Arfi-Rozanova-Pierrat~\cite{ARFI-2017} (for the definitions see Section~\ref{spacesolexttrace}). But, if we modify the definition of the image of the trace operator following Ref.~\cite{JONSSON-1994}, the problem can be also solved for more general boundaries, not necessary of a fixed dimension in the generalized framework of Sobolev admissible domains of Rozanova-Pierrat~\cite{ROZANOVA-PIERRAT-2020} and which was used for the mixed boundary valued problem of the Westervelt equation in Ref.~\cite{DEKKERS-2020-1} (see Theorem~\ref{WeakSolPoiss}).   
We discuss these results in Section~\ref{secpoissdir} introducing the necessary functional framework as the generalization of the trace operator in the sense of special Besov spaces in Section~\ref{secfirstresult}.

Having this trace  generalization, we have the Green formula and integration by parts in the usual way  in the sense of a linear continuous forms, but this time not on $H^\frac{1}{2}(\del \Omega)$, but on the Besov space~\cite{Lancia1,BARDOS-2016,ARFI-2017,ROZANOVA-PIERRAT-2020}. In Section~\ref{secfirstresult} we also define the general framework of the admissible domains and the useful results from Ref.~\cite{ARFI-2017}. In $\R^2$ the admissible domains are the NTA domains with a $d$-set boundary $1\le d<2$ satisfying local Markov's inequality, but in $\R^3$ the class becomes more general. In Section~\ref{secpoissdir}, as was mentioned, we improve the results of Nystr\"{o}m~\cite{Nystrom} showing explicitly the constant dependence in the corresponding estimates (see Theorem~\ref{thmNystrom} and Corollary~\ref{linfPoissonR2}). %

In addition in Section~\ref{SecGenWrem} we notice that the analogous results as in Refs.~\cite{Kalt3,Kaltwer} 
on the well-posedness of the Westervelt equation developed for bounded domains with a regular $C^2$ boundary can be obtained on admissible domains with a $d$-set boundary satisfying the local Markov inequality. The key point is the proof of the necessary estimates given in Proposition~\ref{propestadmdom} which updates the analogous estimates of Refs.~\cite{Kalt3,Kaltwer} making possible their approach.  In the same way, thanks to Ref.~\cite{Grisvard} the results of well-posedness in Refs.~\cite{Kalt3,Kalt2,Kalt1,Kaltwer} found initially for a regular $C^2$ boundary can be extended without modifications   for convex polygonal domains in $\R^2$. %

Using results of Sections~\ref{secfirstresult} and~\ref{secpoissdir} we firstly prove the global in time weak well-posedness result for the strongly damped wave equation %
on arbitrary domains for the homogeneous Dirichlet boundary condition in Subsection~\ref{subsecweakwpdampwavedir} (using the Galerkin method and thanks to the Poincar\'e inequality) and then establish the maximal regularity result using the domain of the weak Dirichlet Laplacian in Subsection~\ref{subsecMaxDirH}. The analogous maximal regularity result on the admissible domains for the nonhomogeneous Dirichlet conditions is given in Subsection~\ref{subssecWPnHLin}. %
Basing on the properties for the linear model, we prove the global in time well-posedness of the Westervelt equation on arbitrary domains in $\R^3$ and on admissible domains in $\R^2$ for the homogeneous Dirichlet condition in Subsection~\ref{secwpWesdirhom}. The nonhomogeneous Dirichlet condition is treated in Subsection~\ref{secwpWesdirinhom} on admissible domains.
In Section~\ref{secasymptan} for the homogeneous Dirichlet boundary condition we improve the well posedness result for the Westervelt equation obtained on NTA two dimensional domains to the arbitrary domains which can be approximated by NTA domains with uniform geometrical constants $M$ and $r_0$ (see Definition~\ref{defNTA}). This result is related with the  
possibility of the approximation of the solution of the Westervelt equation on a fixed arbitrary domain $\Omega$ (in $\R^3$) or NTA domain in $\R^2$ 
by a sequence of solutions on a sequence of domains of the same type (arbitrary or NTA domains with uniform geometrical constants $M$ and $r_0$ (see Theorem~\ref{equivNTAepsdelt})).  For a sequence of such domains $(\Omega_m)_{m_\in \N^*}$ converging to $\Omega$ in sense of Definition~\ref{convomegmR3} we prove the Mosco convergence
of functionals coming from the weak formulations for the Westervelt problem (see Theorem~\ref{Mconv}). 
 The notion of the Mosco convergence ($M$-convergence) for functionals was initially introduced in Ref.~\cite{Mosco}. It implies the weak convergence of the uniformly  bounded (on $m$) sequence of the weak solutions of the Westervelt problems on $\Omega_m$ to the weak solution on $\Omega$ (see Theorem~\ref{thmconvR3}). This kind of arbitrary approximation or the approximation in the same class of domains is common to the shape optimization techniques~\cite{FEIREISL-2002-1,MAGOULES-2020}. We also notice, see for instance p.~113 in Ref.~\cite{HENROT-2005}, that $M$-convergence is related with $\gamma$-convergence.

\section{Main definitions and the functional framework.}%
\label{secfirstresult}
The question to be able to solve the Poisson equation on the most general domain is related with the developing of the extension theory and the definition of the trace on a subset of $\R^n$.

Firstly Calderon-Stein~\cite{CALDERON-1961,STEIN-1970}  states that every Lipschitz domain $\Omega$ is a Sobolev extension domain. Next result is due to Jones~\cite{JONES-1981}, which established that every  $(\eps,\infty)$-domain is a Sobolev extension domain and that this class of domains is optimal in $\R^2$: a simply connected plane domain is a Sobolev extension domain if and only if it is an $(\eps,\infty)$-domain. For  reader's convenience we give the definition of the $(\eps,\delta)$-domain:
\begin{definition}[$(\eps,\delta)$-domain~\cite{JONES-1981}]\label{DefEDD}
An open connected subset $\Omega$ of $\R^n$ is an $(\eps,\delta)$-domain, $\eps > 0$, $0 < \delta \leq \infty$, if whenever $(x, y) \in \Omega^2$ and $|x - y| < \delta$, there is a rectifiable arc $\gamma\subset \Omega$ with length $\ell(\gamma)$ joining $x$ to $y$ and satisfying
\begin{enumerate}
 \item $\ell(\gamma)\le \frac{|x-y|}{\eps}$ and
 \item $d(z,\del \Omega)\ge \eps |x-z|\frac{|y-z|}{|x-y|}$ for $z\in \gamma$. 
\end{enumerate}
\end{definition}
Next important definition for the Sobolev extension domains is the definition of a $d$-set:
\begin{definition}[Ahlfors $d$-regular set or $d$-set~\cite{JONSSON-1984,JONSSON-1995,WALLIN-1991,TRIEBEL-1997}]\label{defdset}
Let $F$ be a Borel non-empty subset of $\R^n$. The set $F$ is is called a $d$-set ($0<d\le n$) if there exists a $d$-measure  $\mu$ on $F$, $i.e.$ a positive Borel measure with support $F$ ($\operatorname{supp} \mu=F$) such that there exist constants 
$c_1$, $c_2>0$,
\begin{equation*}
 c_1r^d\le \mu(F\cap\overline{B_r(x)})\le c_2 r^d, \quad \hbox{ for  } ~ \forall~x\in F,\; 0<r\le 1,
 \end{equation*}
where $B_r(x)\subset \R^n$ denotes the Euclidean ball centered at $x$ and of radius~$r$.
\end{definition}
Thanks to Proposition~1, p.~30 from Ref.~\cite{JONSSON-1984} all $d$-measures on a fixed $d$-set $F$ are equivalent. Hence it is also possible to define a $d$-set by the $d$-dimensional Hausdorff measure $m_d$, which in particular implies that $F$ has Hausdorff dimension $d$ in the neighborhood of each point of $F$ (see p.~33 of Ref.~\cite{JONSSON-1984}).
Let us make attention on the following two examples
\begin{itemize}
\item In $\mathbb{R}^n$, Lipschitz domains and regular domains are $n-$sets with $(n-1)-$sets as boundaries.
\item In $\mathbb{R}^n$, the $(\varepsilon,\delta)$ domains are $n-$sets with a possibly fractal $d-$set boundary~\cite{JONSSON-1984}.
\end{itemize}

There is also a notion of the Non Tangentially Accessible (NTA) domains introduced in Ref.~\cite{Jerison}, which use the definition of Harnack chain:
\begin{definition}[Harnack chain~\cite{Jerison}]
An $M$ non-tangential ball in a domain $\Omega$ is a ball $B(A,r)$ in $\Omega$ whose distance from $\partial\Omega$ is comparable to its radius: 
$$Mr>d(B(A,r),\partial\Omega)>M^{-1}r.$$
For $P_1$, $P_2$ in $\Omega$, a Harnack chain from $P_1$ to $P_2$ in $\Omega$ is a sequence of $M$ non-tangential balls such that the first ball contains $P_1$, the last contains $P_2$, and such that consecutive balls have non empty intersections. 
\end{definition}

\begin{definition}[NTA domain~\cite{Jerison}]\label{defNTA}
A bounded domain $\Omega\subset\mathbb{R}^n$ is called NTA when there exists constants $M$ and $r_0$ such that:
\begin{enumerate}
\item Corkscrew condition: For any $Q \in \partial\Omega$, $r<r_0$, there exists $A=A_r(Q)\in \Omega$ such that $M^{-1}r<\vert A-Q\vert <r$ and $d(A,\partial\Omega)>M^{-1}r$.
\item $\mathbb{R}^n \setminus \overline{\Omega}$ satisfies the Corkscrew condition.
\item Harnack chain condition: If $\epsilon>0$ and $P_1$ and $P_2$ belongs to $\Omega$, $d(P_j;\partial\Omega)>\epsilon$ and $\vert P_1-P_2\vert<C \epsilon$, then there exists a Harnack chain from $P_1$ to $P_2$ whose length depends on $C$ and not on $\epsilon$.
\end{enumerate}
\end{definition}
The relation between the NTA and $(\eps,\delta)$-domains are given by the following theorem:
\begin{theorem}\label{equivNTAepsdelt}[\cite{Nystromal}]
If $\Omega$ is a bounded NTA domain characterized by $M$ and $r_0$, then $\Omega$ is an $(\varepsilon,\delta)$-domain with $\varepsilon$ and $\delta$ characterized by $M$ and $r_0$ only.
\end{theorem}
Let us characterize the geometry of the NTA domains in the plane. There is a close connection between NTA domains and the theory of quasi-conformal mappings. By a quasicircle is understood the image of a circle by a quasi conformal mapping. A domain bounded by a quasicircle is called a quasidisc. For the theory on quasi-conformal mappings we can refer to Refs.~\cite{Gehring} and \cite{Vaisala} for example.
\begin{definition}
A simple closed curve in the plane is said to satisfy \textit{Ahlfors’ three point condition} if for any points $z_1$, $z_2$ of the curve and any $z_3$ on the arc between $z_1$ and $z_2$ of smaller diameter the distance between $z_1$ and $z_3$ is bounded by a constant times the distance between $z_1$ and $z_2$.
\end{definition}
We thus deduce the following theorem:
\begin{theorem}\label{ThEquivNTAS}
Let $\Omega$ be a bounded simply connected subset of the plane. Then the following statements are equivalent:
\begin{enumerate}
\item $\Omega$ is a quasidisc.
\item $\partial\Omega$ satisfies the Ahlfors’ three point condition.
\item $\Omega$ is an NTA domain.
\item $\Omega$ is a Sobolev extension domain.
\end{enumerate}
\end{theorem}
\begin{proof}
$(1)\Leftrightarrow(2)$ is due to Ref.~\cite{Ahlfors} and $(1)\Leftrightarrow(3)$ is due to Ref.~\cite{Jones2}. $(1)\Leftrightarrow(4)$ follows from Ref.~\cite{JONES-1981}.
\end{proof}

Going back to the Sobolev extension results, Jones also mentioned that there are no equivalence between the Sobolev extension domain and the $(\eps,\infty)$-domains in $\R^3$, $i.e.$ there are Sobolev extension domains which are not $(\eps,\infty)$-domains. This question was solved by Haj\l{}asz and al.~\cite{HAJLASZ-2008-1}  for Sobolev spaces $W^k_p(\Omega)$ with $k\in \N^*$ and $1<p<\infty$ giving the optimal class of the Sobolev extension domains in $\R^n$.
It consists of all $n$-sets ($i.e.$  $d$-sets with $d=n$)  on which the Sobolev space is equivalent to the space $C_p^k(\Omega)$ of the sharp functions of the same regularity: %
\begin{multline*}
                                                              C_p^k(\Omega)=\{f\in L^p(\Omega)|\\
                                                              f_{k,\Omega}^\sharp(x)=\sup_{r>0} r^{-k}\inf_{P\in \mathcal{P}^{k-1}}\frac{1}{\lambda(B_r(x))}\int_{B_r(x)\cap \Omega}|f-P|\dy\in L^p(\Omega)\}
                                                             \end{multline*}
with the norm $\|f\|_{C_p^k(\Omega)}=\|f\|_{L^p(\Omega)}+\|f_{k,\Omega}^\sharp\|_{L^p(\Omega)}$ and with the  notation $\mathcal{P}^{k-1}$ for the space of polynomials on $\mathbb{R}^n$ of degree less or equal $k-1$. More precisely, it holds
\begin{theorem}[Sobolev extension~\cite{HAJLASZ-2008}]\label{ThSExHaj}
 For $1<p <\infty$, $k=1,2,...$ a domain $\Omega\subset \R^n$ is a $W^k_p$-extension domain if and only if  $\Omega$ is an $n$-set and $W^{k,p}(\Omega)=C_p^k(\Omega)$ (in the sense of equivalent norms).
\end{theorem}

Thanks, for instance, to Ref.~\cite{JONSSON-1984} it is possible to define the trace of a regular distribution point wise. More precisely, for an arbitrary open set $\Omega$ of $\mathbb{R}^n$ the trace operator $\hbox{Tr}$ is defined for $u\in L^1_{loc}(\Omega)$ by
\begin{equation}\label{deftrace}
	\hbox{Tr} u(x)=\lim_{r\rightarrow 0} \frac{1}{\lambda(\Omega\cap B_r(x))}\int_{\Omega\cap B_r(x)} u(y)\;d\lambda,
\end{equation}
where $\lambda$ is the $n$-dimensional Lebesgue measure and $B_r(x)$ is the Euclidean ball centered at $x$ of  radius $r$.
The trace operator $\hbox{Tr}$ is considered for all $x\in \overline{\Omega}$ for which the limit exists.
By~\cite{WALLIN-1991,JONSSON-1984} it is known that, if $\del \Omega$ is a $d$-set with a positive Borel $d$-measure $\mu$ with  $\operatorname{supp}\mu=\del \Omega$, the limit in Definition~\eqref{deftrace} exists $\mu$-a.e. for $x\in \del \Omega$.

In addition it is possible to define the trace operator as a linear continuous operator from a Sobolev space on $\Omega$ to a Besov space on $\del \Omega$ which is its image, $i.e.$ there exists the right inverse extension $E_{\del \Omega\to \Omega}$ operator and $\hbox{Tr}(E_{\del \Omega\to \Omega} u)=u\in \hbox{Im}(\hbox{Tr}).$
The image of $\hbox{Tr}(H^1(\Omega))$ in this case is the Besov space $B^{2,2}_\alpha(\del \Omega)$ with $\alpha=1-\frac{n-d}{2}>0$~\cite{WALLIN-1991,JONSSON-1984}. From where we obtain the restriction on the dimension of the boundary: $n-2<d<n$. By the way, for a connected boundary of a bounded domain the case $n-2<d<n-1$ is impossible, so it is more realistic to impose $n-1\le d<n$.

Let us notice that if the image of the trace is a Besov space with $\alpha<1$ then we don't need to have any additional geometrical restrictions on the boundary to have the continuity and the surjective property of the trace. But if $\alpha\ge1$ we need to ensure~\cite[2.1]{WALLIN-1982} that there exists a bounded linear extension operator  of the Hölder space $C^{k-1,\alpha-k+1}(\del \Omega)$ to the Hölder space $C^{k-1,\alpha-k+1} (\R^n)$, where for $k\in \N^*$ $k-1<\alpha\le k$ (see also p.~2 Ref.~\cite{JONSSON-1984}). This extension of  Hölder spaces  allows~\cite{JONSSON-1997}  to show the existence of a linear continuous extension  of the Besov space
$B^{p,p}_{\alpha}(\del \Omega)$ on $ \del \Omega$ to the Sobolev space $W^k_p(\R^n)$ with $\alpha=k - \frac{(n - d)}{p}\ge 1$ and $k\ge 2$. To be able to ensure it, we need additionally to assume that the boundary $\del \Omega$ preserves the Markov local inequality~\cite{JONSSON-1984} 
(see Ref.~\cite{ROZANOVA-PIERRAT-2020} for a detailed discussion). 
\begin{definition}[Markov's local inequality]\label{defMark}
A closed subset $V$ in $\mathbb{R}^n$ preserves Markov's local inequality if for every fixed $k\in \mathbb{N}^*$, there exists a constant $c=c(V,n,k)>0$, such that 
$$\max_{V\cap \overline{B_r(x)}}\vert \nabla P\vert\leq \frac{c}{r} \max_{V\cap \overline{B_r(x)}}\vert  P\vert$$
for all polynomials $P\in\mathcal{P}_k$ and all closed balls $\overline{B_r(x)}$, $x\in V$ and $0<r\leq 1$.
\end{definition}
The geometrical characterization of sets preserving Markov's local inequality was initially given in Ref.~\cite{JONSSON-1984-1} (see Theorem 1.3) and can be simply interpreted as sets which are not too flat anywhere. Smooth manifolds in $\R^n$ of dimension less than $n$, as for instance a sphere, are examples of ``flat'' sets not preserving Markov's local inequality, but any $d$-set with $d>n-1$ preserves it, as all $\R^n$.
In the case $\alpha<1$ (hence $k=1$) the local Markov inequality (see Definition~\ref{defMark}) is trivially satisfied  on all closed sets of $\R^n$, and hence we do not need to impose it~\cite[p.~198]{JONSSON-1997}.
  Moreover, we able to consider more general boundaries if we modify the definition of the image of the trace~\cite{ROZANOVA-PIERRAT-2020} thanks to Ref.~\cite{JONSSON-1994}. 

  As detailed in Refs.~\cite{JONSSON-1994,JONSSON-2009} (see also Refs.~\cite{ROZANOVA-PIERRAT-2020,DEKKERS-2020})  we can consider Borel positive measures $\mu$  with a support $\operatorname{supp} \mu= \del \Omega$, which satisfy for some constants $c>0$ and $c'>0$
\begin{equation}\label{EqMeas}
	c\:r^s\leq \mu(B_r(x))\leq c'\:r^d, \quad x\in \del \Omega, \quad 0<r\leq 1.
\end{equation}
  We see that for $d=s$, the measure $\mu$ is a $d$-measure (see Definition~\ref{defdset}). 
  For this general measure $\mu$ supported on a closed subset $\del \Omega\subset \mathbb{R}^n$, which is actually a boundary of a domain $\Omega$ and hence at least $n-1$-dimensional, it is possible, thanks to Ref.~\cite{JONSSON-1994},  to define the corresponding Lebesgue spaces $L^p(\del \Omega,\mu)$ and  Besov spaces $\hat{B}_1^{p,p}(\del \Omega)$ (the norm is different to the norm of $B^{p,p}_{1-\frac{n-d}{p}}(\del \Omega)$ constructed on $d$-sets)  in a such way that we have our second assumption on the operator of the trace mapping $W^1_p(\Omega)$ with $1<p<\infty$ onto $\hat{B}_1^{p,p}(\del \Omega)$ (to compare with Theorem~6 Ref.~\cite{ROZANOVA-PIERRAT-2020}). This is is a particular case of Theorem 1 from Ref.~\cite{JONSSON-1994}.  

The spaces $\hat{B}_1^{p,p}(\del \Omega)$, $B^{p,p}_{1-\frac{n-d}{p}}(\del \Omega)$  are Banach spaces, while $\hat{B}_1^{2,2}(\del \Omega)$, $B^{2,2}_{1-\frac{n-d}{2}}(\del \Omega)$ are Hilbert spaces. %
It is important to notice that for a $d$-set boundary $\del \Omega$ the space $\hat{B}^{p,p}_1(\del \Omega)$ is equivalent to the Besov space $B^{p,p}_\alpha(\del \Omega)$ with $0<\alpha= 1-\frac{n-d}{p}<1$ (see Ref.~\cite{JONSSON-1994}). In addition if $d=s=n-1$, the trace space of $H^1(\Omega)$,  as it also mentioned in Ref.~\cite{BARDOS-2016}, is given by the Besov space with $\alpha=\frac{1}{2}$  which coincide with  $H^\frac{1}{2}(\del \Omega)$:
$$\hat{B}_1^{2,2}(\del \Omega)=B_\frac{1}{2}^{2,2}(\del \Omega)=H^\frac{1}{2}(\del \Omega)$$ as usual in the case of the classical results~\cite{LIONS-1972,MARSCHALL-1987} for Lipschitz boundaries.

Here in this article, to be able to use the regular boundary conditions with $g$ as the trace of an element of $H^2(\Omega)$, we are more interested to stay in the framework of admissible domains introduced in Ref.~\cite{ARFI-2017} and thus to work with the $d$-set boundaries preserving Markov's local inequality, what is the case for  $n-1<d<n$.  Since  Theorem 1 from Ref.~\cite{JONSSON-1994} can be applied only for $H^\beta(\Omega)$ with $\frac{n-d}{2}<\beta \le 1+\frac{n-s}{2}$ with $s\ge d\ge n-1$, the case $H^2(\Omega)$ never occurs. Nevertheless, for weak solutions only in $H^1(\Omega)$ there is no problem.

\begin{definition}[Admissible domain]\label{defadmisdomain}
Let $n-1\le d<n$, $1<p<\infty$ and $k\in \mathbb{N}^*$. A domain $\Omega\subset \mathbb{R}^n$ is called admissible if 
it is a bounded Sobolev extension domain, $i.e.$ an $n-$set  for which $W^{k,p}(\Omega)=C^k_p(\Omega)$ as set with equivalent norms, with a $d-$set boundary $\partial\Omega$ preserving local Markov's inequality.
\end{definition}
Thanks to Theorem~\ref{equivNTAepsdelt} for the plane bounded simply connected domains the admissible domains are equivalent to the NTA domains with a $d$-set boundary preserving local Markov's inequality for $1\le d<2$.

Now, let us recall the trace theorem from Ref.~\cite{ARFI-2017}.

\begin{theorem}\label{thmtradmissdom}[\cite{JonssonBes}]
Let $1<p<+\infty$, $k\in \mathbb{N}^*$ be fixed. Let $\Omega$ be an admissible domain in $\mathbb{R}^n$. Then for $\beta=k-\frac{n-d}{p}>0$, the following trace operators
\begin{enumerate}
\item $\text{Tr}:W^{k,p}(\mathbb{R}^n)\rightarrow B^{p,p}_{\beta}(\partial \Omega)$,
\item $\text{Tr}_{\Omega}:W^{k,p}(\mathbb{R}^n)\rightarrow W^{k,p}(\Omega)$,
\item $\text{Tr}_{\partial\Omega}:W^{k,p}(\Omega)\rightarrow  B^{p,p}_{\beta}(\partial \Omega)$
\end{enumerate} 
are linear continuous and surjective with linear bounded right inverse, \textit{i.e.} extension,
operators $E:B^{p,p}_{\beta}(\partial \Omega)\rightarrow W^{k,p}(\mathbb{R}^n)$, $E_{\Omega}: W^{k,p}(\Omega)\rightarrow W^{k,p}(\mathbb{R}^n)$, $E_{\partial\Omega}: B^{p,p}_{\beta}(\partial \Omega)\rightarrow W^{k,p}(\Omega)$.
\end{theorem}
The definition of the Besov space $B^{p,p}_{\beta}(\partial \Omega) $ on a close $d$-set $\partial\Omega$ can be found, for instance, in Ref.~\cite{JONSSON-1984} p.135.  
The next proposition was shown in Ref.~\cite{ARFI-2017} with the help of Ref.~\cite{Lancia1}.
\begin{proposition}[Green formula]\label{Green}
Let $\Omega$ be an admissible domain in $\mathbb{R}^n$ ($n\geq2$). Then for all $u$, $v\in H^1(\Omega)$ with $\Delta u\in L_2(\Omega)$ it holds the Green formula
$$\langle \frac{\partial u}{\partial n }, Tr v\rangle_{((B^{2,2}_{\beta}(\partial\Omega))',B^{2,2}_{\beta}(\partial\Omega))}:= \int_{\Omega}v\Delta u \;dx+\int_{\Omega}\nabla u \nabla v\;dx,$$
where $\beta=1-\frac{n-d}{2}>0$ and the Besov space $B^{2,2}_{\beta}(\partial\Omega)$ and dual Besov space $(B^{2,2}_{\beta}(\partial\Omega))'= B^{2,2}_{-\beta}(\partial\Omega)$.
\end{proposition}
From~\cite{ARFI-2017} it is also known the generalization of the Rellich-Kondrachov theorem in the class of the Sobolev extension domains.
\begin{theorem}[Sobolev's embeddings]\label{thmsobolembadm}
Let $\Omega\subset\mathbb{R}^n$ be a bounded $n-$set with $W^k_p(\Omega)=C^k_p(\Omega)$, $1<p<+\infty$, $k$, $l\in\mathbb{N}^*$. Then there hold the following compact embeddings
\begin{enumerate}
\item $W^{k+l,p}(\Omega)\subset\subset W^{l,p}(\Omega)$,
\item $W^{k,p}(\Omega)\subset\subset L^q(\Omega)$,
\end{enumerate}
with $q\in [1,+\infty[$ if $kp=n$, $q\in [1,+\infty]$ if $kp>n$, and with $q\in \left[1,\frac{pn}{n-kp}\right[$ if $kp<n$.
Moreover, if $kp<n$ the  embedding
$$W^{k,p}(\Omega) \hookrightarrow L^{\frac{pn}{n-kp}}(\Omega)$$
is continuous.
\end{theorem}
The Poincar\'e inequality stays true on bounded arbitrary domain:
\begin{theorem}[Poincar\'e's inequality]\label{inegPoinc}
Let $\Omega\subset\mathbb{R}^n$ with $n\geq 2$ be a bounded  domain. For all $u\in W^{1,p}_0(\Omega)$ with $1\leq p<+\infty$, there exists $C>0$ depending only on $\Omega$, $p$ and $n$ such that
$$\Vert u\Vert_{L^p(\Omega)}\leq C \Vert \nabla u\Vert_{L^p(\Omega)} .$$
Therefore the semi-norm $\Vert .\Vert_{W^{1,p}_0(\Omega)}$, defined by $\Vert u\Vert_{W^{1,p}_0(\Omega)}:=\Vert \nabla u\Vert_{L^p(\Omega)}$, is a norm which is equivalent to $\Vert .\Vert_{W^{1,p}(\Omega)}$ on $W^{1,p}_0(\Omega)$.

Moreover, if $\Omega$ a bounded $n-$set with $W^1_p(\Omega)=C^1_p(\Omega)$, $1<p<+\infty$, for all $u\in W^{1,p}(\Omega)$ there exists $C>0$ depending only on $\Omega$, $p$ and $n$ such that
$$\left\Vert u-\frac{1}{\lambda(\Omega)}\int_{\Omega} u\;d\lambda\right\Vert_{L^p(\Omega)}\leq C \Vert \nabla u\Vert_{L^p(\Omega)} .$$
\end{theorem}
\begin{proof}
The result for $u\in W^{1,p}_0(\Omega)$ comes from the boundness of $\Omega$. The result for $u\in W^{1,p}(\Omega)$ comes from the compactness of the embedding $W^{1,p}(\Omega)\subset\subset L^p(\Omega)$ from Theorem~\ref{thmsobolembadm} and following for instance the proof in Ref.~\cite{EVANS-1994} (see section 5.8.1 Theorem 1).
\end{proof}
\section{Regularity results on arbitrary domains for the Poisson equation.}\label{secpoissdir}
In a way, the theorems coming from Ref.~\cite{ARFI-2017} show that the functions space on admissible domain share a lot of property with the same function spaces considered on regular domains or  domains with Lipschitz boundary. Nevertheless difference on the regularity of solutions occur when we consider partial differential equation even as simple as the Poisson equation on admissible domains. Using the results from Ref.~\cite{ARFI-2017} extending the results from Ref.~\cite{JONSSON-1997} on $(\varepsilon,\delta)-$domains,  we have the following well-posedness result for the Laplace equation.
\begin{theorem}\label{WeakSolPoiss}
Let $\Omega$ be a bounded Sobolev admissible domain in $\mathbb{R}^n$ ($n\geq2$) in the sense of~Refs.\cite{DEKKERS-2020,ROZANOVA-PIERRAT-2020}:
it is a Sobolev extension domain with a boundary defined by the support of a Borel measure satisfying~(\ref{EqMeas}) with $n-1\le d\le s<n$. %
Then for all  given $f\in L^2(\Omega)$ and $g\in \hat{B}^{2,2}_{1}(\partial \Omega)$ the Poisson problem~(\ref{PoissonDir1})
has a unique weak solution $u\in H^1(\Omega)$ in the following sense 
$$\forall v\in H^1_0(\Omega),\hbox{ } \int_{\Omega}\nabla u\nabla v=\int_{\Omega} f v \quad \hbox{and}\quad %
Tr_{\partial\Omega}u =g.$$
Furthermore, the mapping $\lbrace f,g\rbrace\mapsto u$ is a bounded linear operator from $L^2(\Omega)\times \hat{B}^{2,2}_{1}(\partial \Omega)$ to $H^1(\Omega)$.
\end{theorem}
\begin{proof}
First we use the right inverse of the trace  operator $E_{\partial\Omega}: \hat{B}^{2,2}_{1}(\del \Omega) \to H^1(\Omega)$,  which is bounded~\cite{DEKKERS-2020,ROZANOVA-PIERRAT-2020}, to obtain $\overline{g}\in H^1(\Omega)$ such that $Tr_{\partial\Omega}\overline{g} =g$ and $\Vert \overline{g}\Vert_{H^1(\Omega)}\leq C \Vert g\Vert_{B^{2,2}_{\beta}(\partial \Omega)}$ with $C>0$.
Then the desired solution $u$ of the Poisson problem~(\ref{PoissonDir1}) is defined as $u=w+ \overline{g}$, where $w\in H^1_0(\Omega)$ satisfies
$$\forall v\in H^1_0(\Omega)\hbox{ } \int_{\Omega}\nabla w\nabla v=\int_{\Omega} f v-\int_{\Omega} \nabla \overline{g}\nabla v.$$
The application of  the Lax-Milgram theorem with the Poincar\'e inequality implies the existence and uniqueness of a such $w$  and we can easily deduce the boundness of the mapping $\lbrace f,g\rbrace\mapsto u$ thanks to the estimate
$$ \Vert w\Vert_{H^1(\Omega)} \leq C_1\Vert \nabla w\Vert_{L^2(\Omega)}\leq C_2 \Vert f \Vert_{L^2(\Omega)}+\Vert\nabla \overline{g}  \Vert_{L^2(\Omega)},$$ 
with  constants $C_1>0$ and $C_2>0$.
\end{proof}

Let us also notice that as $-\Delta$ with homogeneous Dirichlet data is a symmetric elliptic operator.  Since on arbitrary domain  the embedding of $H^1_0(\Omega)$ into $L^2(\Omega)$ is  compact,  thanks to Ref.~\cite{EVANS-1994} (Section 6.5.1 p. 334) the spectral problem for $-\Delta$  can be posed on an arbitrary bounded domain to the usual properties:
to have at most countable number of eigenvalues, of finite multiplicity and strictly positive, associated with eigenfunctions forming an orthonormal basis of $L^2(\Omega)$.
As it was mentioned in the Introduction, the regularity of the source $f$ implies~\cite{EVANS-1994} the interior regularity of the weak solution of the Poisson problem $u$ independently of the shape of $\del \Omega$. The question of global regularity up to the boundary is solved for regular boundaries~\cite{EVANS-1994} $\partial\Omega$  of class  $C^{m+2}$: if $f\in H^m(\Omega)$ for $m\in \mathbb{N}$, then the solution of the homogeneous Dirichlet Poisson problem $u\in H^{m+2}(\Omega)\cap H^1_0(\Omega)$ and satisfies
 the estimate
$$
\Vert u\Vert_{H^{m+2}(\Omega)}\leq C \Vert f\Vert_{H^m(\Omega)}
$$
with the constant $C>0$ depending only on $m,$ $\Omega$.
The work of Grisvard ~\cite{Grisvard} tells us that in dimension $n=2$ this result is also true for convex polygonal domains.
But~\cite{Nystrom} this is no longer true for domains with a fractal boundary even if the data $f$ is very regular. In the general case we never have the weak solutions from $H^2(\Omega)$, but only from $H^2_{loc}(\Omega)$. 

However a weaker regularity property  $u\in C(\Omega)\cap L^{\infty}(\Omega)$ can be still considered. %
In dimension $n=3$ we have the next theorem:
\begin{theorem}\label{thmXie}[\cite{Xie}]
Let $\Omega$ be an arbitrary open set in $\mathbb{R}^3$. If $u\in H^1_0(\Omega)$ and $\Delta u\in L^2(\Omega)$, then
$$\Vert u\Vert_{L^{\infty}(\Omega)}\leq\frac{1}{\sqrt{2\pi}}\Vert \nabla u\Vert_{L^2(\Omega)}^{1/2} \Vert \Delta u\Vert_{L^2(\Omega)}^{1/2}.$$
The constant $\frac{1}{\sqrt{2\pi}}$ is the best possible for all $\Omega$.
\end{theorem}
Using Theorems~\ref{WeakSolPoiss}  and~\ref{thmXie} with the interior regularity results of Evans we deduce the following result:
\begin{corollary}\label{PoissonR3}
Let $\Omega$ be a bounded arbitrary domain in $\mathbb{R}^3$. Assume $f\in L^2(\Omega)$. Suppose furthermore that $u\in H^1_0(\Omega)$ is the weak solution of the Poisson problem (\ref{PoissonDir1}) with $g=0$.
Then 
$$u\in C(\Omega)\cap L^{\infty}(\Omega) $$
and it holds  estimate~(\ref{EqCREch2})
with the constant $C>0$ depending only on the Lebesgue measure of $\Omega$ denoted by $\lambda(\Omega)$.
\end{corollary}
\begin{proof}
On  a bounded arbitrary domain $\Omega$ for $f\in L^2(\Omega)$ we have~\cite{EVANS-1994} 
$u\in H^2_{loc}(\Omega)\cap H^1_0(\Omega)$. By the standard stability estimate coming from the Lax-Milgram theorem with the help of the Poincar\'e inequality,  we find
$$\Vert \nabla u\Vert_{L^2(\Omega)}\leq C \Vert f\Vert_{L^2(\Omega)}=C\Vert \Delta u\Vert_{L^2(\Omega)}$$
with a constant $C>0$ depending only on $\lambda(\Omega)$.
Then we use the estimate of Theorem~\ref{thmXie} to conclude.

\end{proof}
 For a similar result in dimension $n=2$ we use  Ref.~\cite{Nystromthese}.
For $f\in L^2(\Omega)$ we denote by $Gf\in H^1_0(\Omega)$ the Green potential solution of the Poisson problem~(\ref{PoissonDir1}) with $g=0$. Thanks to  Ref.~\cite{Nystrom} we have  the following theorem in which we precise the dependence on $\Omega$ of the constants in the estimate of Nystr\"om (see~\ref{append1} for the proof).
\begin{theorem}\label{preqthmNys}
Let $\Omega\subset \mathbb{R}^n$ with $n\geq 2$ a bounded and simply connected NTA domain. Let $q_0=1+\frac{1}{1-\beta(M)}>2$. Then there exist constants $\epsilon=\epsilon(\Omega)$ and 
$$C=C(M,n,q,r_0,\delta,\operatorname{dimloc}(\partial\Omega),diam(\Omega))$$
such that if $\frac{n}{n-1}<q<q_0+\varepsilon$, $\frac{1}{q}=\frac{1}{p}-\frac{1}{n}$, then the following inequality is valid for all $f\in L^p(\Omega)$,
$$\left(\int_{\Omega}\left\vert\frac{Gf(x)}{d(x,\partial\Omega)}\right\vert^q dx\right)^{\frac{1}{q}}\leq C \Vert f\Vert_{L^p(\Omega)}=C \Vert\Delta  Gf\Vert_{L^p(\Omega)}.
$$
\end{theorem}
\begin{remark}
Here $M$ and $r_0$ come from  Definition~\ref{defNTA} of the NTA domain and $\delta$ is a fixed constant such that $\delta<<r_0$. We have $\operatorname{dimloc}(\partial\Omega)\leq n$ characteristic constant of $\Omega$ induced by a Whitney cubes decomposition of $\Omega$ see page~16 in Ref.~\cite{STEIN-1970} and Definition~2.5 in Ref.~\cite{Nystromthese}. The key point in the proof in Ref.~\cite{Nystrom} is that for an NTA domain $\operatorname{dimloc}(\partial\Omega)<n$ and more precisely $\operatorname{dimloc}(\partial\Omega)\leq \operatorname{dim}_M(\partial\Omega)$ the Minkowski dimension of $\Omega$.
\end{remark}
The constant $\beta=\beta(M)>0$ in the statement of Theorem \ref{preqthmNys} is the $\beta$ described in the following lemma.  
\begin{lemma}\label{lemNystrom}
[\cite{Jerison,Nystrom}] Let $Q_0\in \partial\Omega$ and  $\Gamma(Q_0,r):=B(Q_0,r)\cap \partial\Omega$ for all $r>0$. Let $w(x,\Gamma(Q_0,r),\Omega)$ be the harmonic measure on $\Omega$ and $d(y,\partial\Omega)$ be the Euclidean distance from $y$ to $\partial\Omega$.
Let in addition $G$ be the Green potential associated to $\Omega$ for the Poisson problem (\ref{PoissonDir1}) with $g=0$. Then there exists a constant $C=C(n)$ such that if $Cr<r_0$ and $x\in \Omega \setminus B(Q_0,Cr)$, then 
there exists $C(M,n)>0$ with $\beta=\beta(M)>0$ for all $y\in B(Q_0,r)\cap\Omega$ for which it holds
$$G(x,y)\leq C(M,n)\frac{d(y,\partial\Omega)^{\beta}}{r^{n-2+\beta}} w(x,\Gamma(Q_0,r),\Omega).$$
\end{lemma}

In the same way as for Theorem~\ref{preqthmNys} (see~\ref{append1} for the constant dependence precisions) we also precise the main theorem from Ref.~\cite{Nystrom}.
\begin{theorem}\label{thmNystrom}
Let $\Omega\subset\mathbb{R}^n$ with $n \geq 2$ be a bounded simply connected NTA domain. Let $q_0=1+\frac{1}{1-\beta(M)}>2$, where $\beta(M)>0$ is a constant describing the behavior of the Green function near the boundary (see Lemma~\ref{lemNystrom}) and $M$ is a constant which appears in the Definition~\ref{defNTA} of a NTA domain. Then there exists constants $\epsilon=\epsilon(\Omega)$ and 
$$C=C(M,n,q,r_0,\delta,\operatorname{dimloc}(\partial\Omega),diam(\Omega))$$
 such that for $\frac{n}{n-1}<q<q_0+\varepsilon$ and $\frac{1}{q}=\frac{1}{p}-\frac{1}{n}$ the following inequality is valid for all $f\in L^p(\Omega)$,
$$\Vert \nabla Gf\Vert_{L^q(\Omega)}\leq C \Vert f\Vert_{L^p(\Omega)}=C \Vert\Delta  Gf\Vert_{L^p(\Omega)}.
$$
\end{theorem}
\begin{remark}
The same results hold true for Lipschitz domains but thanks to Dahlberg~\cite{Dahlberg} it is possible to replace $q_0(M)$ by $q_2=4$ in dimension $n=2$ and $q_n= 3$ in dimension $n\geq 3$.
\end{remark}
From Theorems~\ref{preqthmNys} and~\ref{thmNystrom} it follows the continuity up to boundary result for solutions of the Poisson equation with homogeneous Dirichlet boundary conditions on the plane NTA domains:
\begin{corollary}\label{linfPoissonR2}
Let $\Omega\subset\mathbb{R}^2$ be a bounded NTA domain characterized by its constant $M$ and $r_0$. Assume $f\in L^2(\Omega)$. Suppose furthermore that $u\in H^1_0(\Omega)$ is the weak solution of the Poisson problem (\ref{PoissonDir1}) with $g=0$.
Then $u\in C(\overline{\Omega}) $ satisfying~\eqref{EqCREch2} with 
the constant $C$ depending only on $M$, $r_0$ and $diam(\Omega)$.
\end{corollary}
\begin{proof}
According to Theorem~\ref{thmNystrom} we have for $q_0=1+\frac{1}{1-\beta(M)}>2$ the solution of the Poisson problem (\ref{PoissonDir1}) with $g=0$ $u\in W^{1,q_0}_0(\Omega)$ such that 
$$\Vert \nabla u\Vert_{L^{q_0}(\Omega)}\leq C \Vert f\Vert_{L^2(\Omega)},
$$
with $C=C(M,n,q_0,r_0,\delta,\operatorname{dimloc}(\partial\Omega),diam(\Omega)).$
Here we have fixed $n=2$, we have $q_0=q_0(M)$ and $\delta<<r_0$ fixed. Moreover, according to Ref.~\cite{Nystromthese} we can suppress on an NTA domain, characterized by $M$ and $r_0$, the dependence on $\operatorname{dimloc}(\partial\Omega)$. Then
$$\Vert \nabla u\Vert_{L^{q_0}(\Omega)}\leq C(M,r_0,diam(\Omega)) \Vert f\Vert_{L^2(\Omega)},
$$
but by the Poincar\'e inequality and extending by $0$ we have
$$\Vert  u\Vert_{W^{1,q_0}(\mathbb{R}^2)}\leq C(M,r_0,diam(\Omega)) \Vert f\Vert_{L^2(\Omega)}.
$$
From the Sobolev embedding, it follows that $u\in C(\mathbb{R}^2)$ with
$$\Vert u\Vert_{L^{\infty}(\mathbb{R}^2)}\leq C(q_0) \Vert  u\Vert_{W^{1,q_0}(\mathbb{R}^2)}.$$
Taking its restriction to $\Omega$ we obtain $u\in C(\overline{\Omega})$ with
$$\Vert u\Vert_{L^{\infty}(\Omega)}\leq C(M,r_0,diam(\Omega)) \Vert f\Vert_{L^2(\Omega)}.$$
\end{proof}

\section{Generalization of known well-posedness results for the Westervelt equation}\label{SecGenWrem}
 Refs.~\cite{Kalt3} and \cite{Kaltwer} on the well-posedness of the Westervelt equation use estimates that are true for bounded domains with a regular $C^2$ boundary. In this section we present the analogous estimates necessary in order to have similar results of well-posedness of the Westervelt equation on the admissible domains.
\begin{proposition}\label{propestadmdom}
Let $\Omega$ be a bounded connected admissible domain in $\mathbb{R}^n$ for $n=2$ or $3$ with a $d-$set boundary $\partial\Omega$ such that $n-1\le d<n$. Set
$$\beta_1=1-\frac{n-d}{2}>0\hbox{ and }\beta_2=2-\frac{n-d}{2}>0.$$
For $w\in H^1(\Omega)$ with $\Delta w\in L^2(\Omega)$ and $Tr_{\partial\Omega} w\in B^{2,2}_{\beta_2}(\partial\Omega)$ there holds
 \begin{align}
H^1(\Omega)\subset L^2(\Omega)&\hbox{ with }\Vert w\Vert_{L^2(\Omega)}\leq \tilde{C}_0 (\Vert \nabla w\Vert_{L^2(\Omega)}+\Vert Tr_{\partial\Omega} w\Vert_{B^{2,2}_{\beta_1}(\partial\Omega)}),\label{injL2}\\
&\Vert\nabla w\Vert_{L^2(\Omega)}\leq \widehat{C}_0 (\Vert \Delta w\Vert_{L^2(\Omega)}+\Vert Tr_{\partial\Omega} w\Vert_{B^{2,2}_{\beta_2}(\partial\Omega)}),\label{injL2nab}\\
H^1(\Omega)\subset L^6(\Omega)&\hbox{ with } \Vert w\Vert_{L^6(\Omega)}\leq \tilde{C}_1 (\Vert \nabla w\Vert_{L^2(\Omega)}+\Vert Tr_{\partial\Omega} w\Vert_{B^{2,2}_{\beta_1}(\partial\Omega)}),\label{injL6}\\
&\Vert w\Vert_{L^{\infty}(\Omega)} \leq \tilde{C}_2 (\Vert \Delta w\Vert_{L^2(\Omega)}+\Vert Tr_{\partial\Omega} w\Vert_{B^{2,2}_{\beta_2}(\partial\Omega)}),\label{injLin}\\
 L^{\frac{6}{5}}(\Omega)\subset H^{-1}(\Omega),&\hbox{ with }\Vert w\Vert_{H^{-1}(\Omega)}\leq \tilde{C}_3 \Vert w\Vert_{L^{\frac{6}{5}}(\Omega)}.\label{injdual}
\end{align}
Moreover, for $n=2$ and for $p_1>2$ and $p'_1>2$ fixed in a such way  that $2<p_1<q_0+\epsilon$ (see Theorem \ref{thmNystrom}) and $\frac{1}{p_1}+\frac{1}{p_1'}=\frac{1}{2}$ there exist constants $C_{p_1}$, $C_{p_1'}>0$ such that
\begin{align}
\Vert \nabla w\Vert_{L^{p_1}(\Omega)}\leq C_{p_1} (\Vert \Delta w\Vert_{L^2(\Omega)}+\Vert Tr_{\partial\Omega} w\Vert_{B^{2,2}_{\beta_2}(\partial\Omega)}),\label{injLp1}\\
\Vert w\Vert_{L^{p_1'}(\Omega)}\leq C_{p_1'} (\Vert \nabla w\Vert_{L^2(\Omega)}+\Vert Tr_{\partial\Omega} w\Vert_{B^{2,2}_{\beta_1}(\partial\Omega)}).\label{injLp1'}
\end{align}
\end{proposition}
\begin{proof}
Estimates~(\ref{injL2}) and~(\ref{injL6}) are a direct consequence of Proposition~3 in Ref.~\cite{ARFI-2017} as the norm $\sqrt{\Vert \nabla . \Vert_{L^2(\Omega)}^2+\Vert Tr_{\partial\Omega}.\Vert_{L^2(\partial\Omega)}^2}$ is equivalent to the $H^1$-norm and by Ref.~\cite{ARFI-2017} for instance $B^{2,2}_{\beta_1}(\partial\Omega)\subset\subset L^2(\partial\Omega)$. Estimate~(\ref{injdual}) comes from Theorem~\ref{thmsobolembadm} and  the duality.  In dimension $n=2$ we have by Theorem~\ref{thmtradmissdom} 
$$E_{\partial\Omega}(Tr_{\partial\Omega}w)\in H^2(\Omega)\hbox{ with }\Vert E_{\partial\Omega}(Tr_{\partial\Omega}w)\Vert_{H^2(\Omega)}\leq C \Vert Tr_{\partial\Omega}w\Vert_{B^{2,2}_{\beta_2}(\partial\Omega)} $$ 
and $Tr_{\partial\Omega}[w-E_{\partial\Omega}(Tr_{\partial\Omega}w)]=0$. 

Then it implies $w-E_{\partial\Omega}(Tr_{\partial\Omega}w)\in H^1_0(\Omega)$ and $\Delta[w-E_{\partial\Omega}(Tr_{\partial\Omega}w)]\in L^2(\Omega)$. So by Theorem~\ref{thmNystrom} we take $p_1>2$ such that $2<p_1<q_0+\epsilon$ to obtain 
$$\Vert\nabla [w-E_{\partial\Omega}(Tr_{\partial\Omega}w)]\Vert_{L^{p_1}(\Omega)}\leq C(p_1,\Omega) \Vert \Delta[w-E_{\partial\Omega}(Tr_{\partial\Omega}w)]\Vert_{ L^{\frac{2p_1}{2+p_1}}(\Omega)}.$$
But $1<\frac{2p_1}{2+p_1}<2$ and $\Omega$ is bounded. Consequently $L^2(\Omega)\hookrightarrow L^{\frac{2p_1}{2+p_1}}(\Omega)$ and thus  we  obtain estimate~(\ref{injLp1}) by the fact that
\begin{align*}
\Vert \nabla w\Vert_{L^{p_1}(\Omega)}\leq & \Vert \nabla[ w - E_{\partial\Omega}(Tr_{\partial\Omega}w) ]\Vert_{L^{p_1}(\Omega)}+\Vert \nabla E_{\partial\Omega}(Tr_{\partial\Omega}w)  \Vert_{L^{p_1}(\Omega)}\\
\leq & C \Vert \Delta[w-E_{\partial\Omega}(Tr_{\partial\Omega}w)]\Vert_{ L^{\frac{2p_1}{2+p_1}}(\Omega)} +C \Vert E_{\partial\Omega}(Tr_{\partial\Omega}w)  \Vert_{H^2(\Omega)}\\
\leq & C \Vert \Delta w\Vert_{L^2(\Omega)}+ C \Vert \Delta E_{\partial\Omega}(Tr_{\partial\Omega}w)\Vert_{ L^2(\Omega)} +C \Vert E_{\partial\Omega}(Tr_{\partial\Omega}w)  \Vert_{H^2(\Omega)}\\
\leq & C \Vert \Delta w\Vert_{L^2(\Omega)} +C \Vert E_{\partial\Omega}(Tr_{\partial\Omega}w)  \Vert_{H^2(\Omega)}\\
\leq & C \Vert \Delta w\Vert_{L^2(\Omega)} + C \Vert Tr_{\partial\Omega}w\Vert_{B^{2,2}_{\beta_2}(\partial\Omega)} .
\end{align*}
We  deduce estimate~(\ref{injL2nab})  in the same way but also estimate~(\ref{injLin}) as $W^{1,p_1}_0(\Omega)\subset L^{\infty}(\Omega)$ by Theorem \ref{thmsobolembadm}.
In dimension $n=3$ we use again $E_{\partial\Omega}(Tr_{\partial\Omega}w) $ with Corollary~\ref{PoissonR3} and Theorem~\ref{thmtradmissdom} to obtain estimate~(\ref{injLin}).
The proof of estimate~(\ref{injLp1'}) is not different from the proof of estimate~(\ref{injL6}) as $p_1'>2$ and we are in the dimension $n=2$.
\end{proof}
\begin{remark}\label{remarkNTA}
Estimates~(\ref{injL2})--(\ref{injdual}) are very similar to those used in Ref.~\cite{Kalt3} and \cite{Kaltwer} for a regular domain, with the Besov spaces replacing $H^{3/2}(\partial\Omega)$ and $H^{1/2}(\Omega)$. Nevertheless  Theorem~\ref{thmNystrom} tell us that we do not have on a general NTA domain or Lipschitz domain the estimate
$$\Vert \nabla w\Vert_{L^6(\Omega)}\leq C (\Vert \Delta w\Vert_{L^2(\Omega)}+\Vert Tr_{\partial\Omega} w\Vert_{B^{2,2}_{\beta_2}(\partial\Omega)}).$$
This requires  a sly modification in the proof of Ref.~\cite{Kaltwer}. In the dimension $n=2$ this estimate stays true for convex polygonal domains thanks to Ref.~\cite{Grisvard} which allows to extend the results of well-posedness in Refs.~\cite{Kalt3,Kalt2,Kalt1,Kaltwer} found initially for a regular $C^2$ boundary.
\end{remark}

\section{Well posedness of the damped linear wave equation.}\label{secwpdampwavdirhom}
\subsection{Existence and unicity of a weak solution with homogeneous Dirichlet boundary condition.}\label{subsecweakwpdampwavedir}
In this subsection we suppose that $\Omega$ is an arbitrary bounded domain in $\mathbb{R}^n$, on which we consider the following linear strongly damped wave equation with homogeneous Dirichlet boundary condition:
\begin{equation}\label{dampedwaveeqdirhom}
\left\lbrace
\begin{array}{c}
\del_t^2 u-c^2 \Delta u- \eps \nu\Delta \del_t u=f\;\;\;on\;\;]0,+\infty[\times\Omega,\\
u\vert_{\partial\Omega}=0\;\;\;on\;\;[0;+\infty[\times\partial\Omega,\\
u(0)=u_0,\;\;\del_t u(0)=u_1.
\end{array}
\right.
\end{equation}
We are looking for weak solutions of system~(\ref{dampedwaveeqdirhom}) in the following sense:
\begin{definition}\label{defweakdampwav}
For $f\in L^2([0,+\infty[;L^2(\Omega))$, $u_0\in H^1_0(\Omega)$, and $u_1\in L^2(\Omega)$,
we say that a function $u\in L^2([0,+\infty[;H^1_0(\Omega))$ with $\partial_t u \in L^2([0,+\infty[;H^1_0(\Omega))$ and $\partial^2_t u\in L^2([0,+\infty[;H^{-1}(\Omega)$ is a weak solution of  problem~(\ref{dampedwaveeqdirhom}) if  it satisfies
$$ u(0)=u_0,\;\;\del_t u(0)=u_1$$
and for all $v\in L^2([0,+\infty[;H^1_0(\Omega))$
\begin{multline}
\int_0^{+\infty}\langle  \del_t^2 u,v\rangle_{(H^{-1}(\Omega),H^1_0(\Omega))} +c^2 (\nabla u,\nabla v)_{L^2(\Omega)}+\eps \nu(\nabla \del_t u,\nabla v)_{L^2(\Omega)} ds\\
=\int_0^{+\infty} (f,v)_{L^2(\Omega)}ds.\label{varformdampedwaveeqdirhom}
\end{multline}
\end{definition}
To prove the existence and uniqueness of a such weak solution we use the Galerkin method  and follow ~\cite[p.~379--387]{EVANS-2010} using   the Poincar\'e inequality. 
 To perform the Galerkin method we start by   solving a finite dimensional approximation. We thus  select smooth functions $w_k=w_k(x)$ $(k\geq 1)$, such that
\begin{equation}\label{ortgbasH10}
\lbrace w_k\rbrace_{k=1}^{\infty}\text{ is an orthogonal basis of }H^1_0(\Omega)
\end{equation}
and in addition
\begin{equation}\label{ortnbasL2}
\lbrace w_k\rbrace_{k=1}^{\infty}\text{ is an orthonormal basis of }L^2(\Omega).
\end{equation}
Typically $w_k$ are the normalized eigenfunctions of the operator $-\Delta$ on $\Omega$ with homogeneous Dirichlet boundary conditions, solving
$$-\Delta w_k=\lambda_k w_k$$
 in the following sense
\begin{equation}\label{eigenfuncDelta}
\forall w\in H^1_0(\Omega) \quad (\nabla w_k,\nabla w)_{L^2(\Omega)}=\lambda_k (w_k,w)_{L^2(\Omega)}.
\end{equation}
For  a positive integer $m$ we define the finite approximation of $u$ by
\begin{equation}\label{Galerksol}
u_m(t):=\sum_{i=1}^m d^k_m(t) w_k.
\end{equation}
Then we firstly determine the coefficients $d^k_m(t)$: 
\begin{proposition}\label{WpGalerksol}
For each integer $m=1,2,...,$ there exists a unique function $u_m$ of form~(\ref{Galerksol}) with the coefficients $d^k_m(t)\in H^2([0,+\infty[)$ $(t \geq 0,\;k=1,...,m)$ satisfying
\begin{align}
d^{k}_m(0)&=(u_0,w_k)_{L^2(\Omega)}\in\mathbb{R}\;(k=1,...,m),\label{Galerkinit1}\\
\partial_t d^{k}_m(0)&=(u_1,w_k)_{L^2(\Omega)}\in\mathbb{R}\;(k=1,...,m)\label{Galrkinit2}
\end{align}
 and for $t \geq 0$ and $k=1,...,m$
\begin{equation}\label{Galerkdampedeq}
(\partial^2_t u_m,w_k)_{L^2(\Omega)}+c^2 (\nabla u_m,\nabla w_k)_{L^2(\Omega)}+\eps \nu(\nabla \partial_t u_m,\nabla w_k)_{L^2(\Omega)}=(f,w_k)_{L^2(\Omega)}.
\end{equation}
\end{proposition}
\begin{proof}
Let $u_m$ be given by Eq.~(\ref{Galerksol}).
Furthermore, we have 
$$e^{kl}:=(\nabla w_l, \nabla w_k)_{L^2(\Omega)}\in\mathbb{R}.$$
We also write $f^k=(f,w_k)_{L^2(\Omega)}\in L^2([0,+\infty[)$.
Consequently, relation~(\ref{Galerkdampedeq}) becomes the linear system of ODE
\begin{equation}\label{Galerkdampedeq2}
\partial^2_t d^k_m(t)+\sum_{l=1}^m c^2 e^{kl}d^l_m(t)+\sum_{l=1}^m \eps \nu e^{kl} \partial_t d^l_m(t) =f^k(t)\;\;(t \geq 0,\;k=1,...,m),
\end{equation}
with the initial conditions~(\ref{Galerkinit1}) and~(\ref{Galrkinit2}). According to Cauchy-Lipschitz theory for ordinary differential equations~\cite{Roubicek}, there exists a unique function 
$$d_m(t)=(d^1_m(t),..,d^m_m(t))\in H^2([0,+\infty[),$$ satisfying~(\ref{Galerkinit1}),~(\ref{Galrkinit2}) and solving~(\ref{Galerkdampedeq2}) for $t \geq 0$.
\end{proof}
Now, we need to obtain uniform energy estimates on $m$ to be able to pass the limit in Eq.~(\ref{Galerkdampedeq2}) for $m\rightarrow \infty$.
\begin{proposition}
Let $u_m$ be of form~(\ref{Galerksol}) satisfying~(\ref{Galerkinit1}),~(\ref{Galrkinit2}) and~(\ref{Galerkdampedeq}), defined in Proposition~\ref{WpGalerksol}.
Then there exists a constant $C>0$, depending only on $\Omega$, such that for $m\in\mathbb{N}^*$
\begin{align}
\max_{t \geq 0} (\Vert u_m(t)\Vert_{H^1_0(\Omega)}^2 &+ \Vert\partial_t u_m(t)\Vert_{L^2(\Omega)}^2)+\Vert \nabla \partial_t u_m\Vert_{L^2([0,+\infty[;L^2(\Omega)) }^2 \nonumber\\
&+\Vert \nabla  u_m\Vert_{L^2([0,+\infty[;L^2(\Omega)) }^2 +\Vert\partial^2_t u_m\Vert_{L^2((0;+\infty[;H^{-1}(\Omega))}^2\nonumber\\
&\leq C (\Vert f\Vert_{L^2([0,+\infty[;L^2(\Omega))}^2+\Vert u_0\Vert_{H^1_0(\Omega)}^2+\Vert u_1\Vert_{L^2(\Omega)}^2).\label{aprioriGalerkdamped}
\end{align}
\end{proposition}
\begin{proof}
We multiply equality~(\ref{Galerkdampedeq}) by $\partial_t d^k_m(t)$, sum over $k=1,..,m$ and recall~(\ref{Galerksol}) to obtain
$$(\partial^2_t u_m,\partial_t u_m)_{L^2(\Omega)}+c^2 (\nabla u_m,\nabla \partial_t u_m)_{L^2(\Omega)}+\eps \nu(\nabla \partial_t u_m,\nabla \partial_t u_m)_{L^2(\Omega)}=(f,\partial_t u_m)_{L^2(\Omega)}$$
for almost all $t \geq 0$. 
Using successively Cauchy-Schwarz, Poincar\'e and Young inequalities we find
\begin{align}
(f,\partial_t u_m)_{L^2(\Omega)}\leq & \Vert f\Vert_{L^2(\Omega)} \Vert \partial_t u_m\Vert_{L^2(\Omega)}\nonumber\\
\leq & K\Vert f\Vert_{L^2(\Omega)}  \Vert\nabla \partial_t u_m\Vert_{L^2(\Omega)}\label{proofGalerkest1}\\
\leq & \frac{K^2}{2\eps \nu} \Vert f\Vert_{L^2(\Omega)}^2+ \frac{\eps \nu}{2} \Vert\nabla \partial_t u_m\Vert_{L^2(\Omega)}^2\nonumber.
\end{align}
Moreover, for $t \geq 0$ we integrate in time and find
\begin{align*}
\Vert\partial_t u_m(t)\Vert_{L^2(\Omega)}^2+c^2 \Vert u_m(t)\Vert_{H^1_0(\Omega)}^2+&\eps \nu\int_0^t \Vert \nabla \partial_t u_m(s)\Vert_{L^2(\Omega)}^2\;ds\\
&\leq C( \Vert f\Vert_{L^2([0,+\infty[;L^2(\Omega))}^2+c^2\Vert u_0\Vert_{H^1_0(\Omega)}^2+\Vert u_1\Vert_{L^2(\Omega)}^2)
\end{align*}
with $C>0$ independent of$t$.
Since $t \geq 0$ was arbitrary, from the lust estimate  it follows
\begin{align}
\max_{0\leq t}(\Vert\partial_t u_m(t)\Vert_{L^2(\Omega)}^2+& \Vert u_m(t)\Vert_{H^1_0(\Omega)}^2)+\int_0^{+\infty} \Vert \nabla \partial_t u_m(s)\Vert_{L^2(\Omega)}^2\;ds\nonumber\\
&\leq C( \Vert f\Vert_{L^2([0,+\infty[;L^2(\Omega))}^2+\Vert u_0\Vert_{H^1_0(\Omega)}^2+\Vert u_1\Vert_{L^2(\Omega)}^2).\label{partaprioridampedGalerk}
\end{align}
We multiply equations~(\ref{Galerkdampedeq}) by $  d^k_m(t) $ and sum over $k=1,...,m$. By definition~(\ref{Galerksol}) of $u_m$  we have
$$(\partial^2_t u_m, u_m)_{L^2(\Omega)}+c^2 (\nabla u_m,\nabla  u_m)_{L^2(\Omega)}+\eps \nu(\nabla \partial_t u_m,\nabla  u_m)_{L^2(\Omega)}=(f, u_m)_{L^2(\Omega)}$$
for a.e. $t \geq 0$.

As for the term $(f,\partial_t u_m)_{L^2(\Omega)}$ (see estimate~(\ref{proofGalerkest1})) we have
\begin{align*}
(f, u_m)_{L^2(\Omega)}
\leq  \frac{K^2}{2c^2 } \Vert f\Vert_{L^2(\Omega)}^2+ \frac{c^2}{2} \Vert\nabla  u_m\Vert_{L^2(\Omega)}^2
\end{align*}
with a constant $K>0$ independent oftime.
Then we find
$$\frac{d}{dt}\Big(\frac{\eps \nu}{2}\Vert \nabla u_m(t)\Vert_{L^2(\Omega)}^2\Big)+\frac{c^2}{2} \Vert \nabla u_m(s)\Vert_{L^2(\Omega)}^2\;ds\leq \frac{K^2}{2c^2}\Vert f\Vert_{L^2(\Omega)}^2-\int_{\Omega} \partial^2_t u_m  u_m\;dx.$$
Integrating over $[0,t]$ it becomes
\begin{align*}
\frac{\eps \nu}{2}\Vert \nabla u_m(t)\Vert_{L^2(\Omega)}^2+ & \frac{c^2}{2}\int_0^t \Vert \nabla u_m(s)\Vert_{L^2(\Omega)}^2\;ds\\
\leq & \frac{\eps \nu}{2}\Vert \nabla u_0\Vert_{L^2(\Omega)}^2+\frac{K^2}{2c^2} \Vert f\Vert_{L^2([0,+\infty[;L^2(\Omega))}^2-\int_0^t \int_{\Omega} \partial^2_t u_m  u_m\;dx\;ds.
\end{align*}
We obtain 
\begin{align*}
-\int_0^t \int_{\Omega} \partial^2_t u_m  u_m\;dx = &  \int_0^t\left[-\frac{d}{dt}\left( \int_{\Omega} \partial_t u_m  u_m\;dx\right)+\int_{\Omega} ( \partial_t u_m)^2\;dx\right]\;ds\\
= & -\int_{\Omega} \partial_t u_m  u_m\;dx+\int_{\Omega} \partial_t u_m(0)  u_m(0)\;dx\\
& +\Vert \partial_t u_m\Vert_{L^2([0,t];L^2(\Omega))}^2,\\
\intertext{so by the Young inequality}
-\int_0^t \int_{\Omega} \partial^2_t u_m u_m\;dx \leq &\frac{1}{2} \Vert \partial_t u_m\Vert_{L^{\infty}([0,t];L^2(\Omega))}^2 +\frac{1}{2} \Vert u_m\Vert_{L^{\infty}([0,t];L^2(\Omega))}^2\\
& +\frac{1}{2} \Vert  \partial_t u_m(0)\Vert_{L^2(\Omega)} +\frac{1}{2}\Vert  u_m(0)\Vert_{L^2(\Omega)}\\
& +\Vert \partial_t u_m\Vert_{L^2([0,t];L^2(\Omega))}^2,\\
\intertext{then by the Poincar\'e inequality and relations for initial conditions (\ref{Galerkinit1})--(\ref{Galrkinit2})}
-\int_0^t \int_{\Omega} \partial^2_t u_m  u_m\;dx \leq &\frac{1}{2} \Vert \partial_t u_m\Vert_{L^{\infty}([0,+\infty[;L^2(\Omega))}^2 +\frac{K^2}{2} \Vert u_m\Vert_{L^{\infty}([0,+\infty[;H^1_0(\Omega))}^2\\
& +\frac{1}{2} \Vert  u_1\Vert_{L^2(\Omega)}^2 +\frac{K^2}{2}\Vert\nabla u_0\Vert_{L^2(\Omega)}^2\\
& +K^2 \Vert \partial_t u_m\Vert_{L^2([0,+\infty[;H^1_0(\Omega))}^2
\intertext{with a constant $K>0$ depending only on $\Omega$. Thus by estimate (\ref{partaprioridampedGalerk}) we find}
-\int_0^t \int_{\Omega} \partial^2_t u_m  u_m\;dx\leq & C ( \Vert f\Vert_{L^2([0,+\infty[;L^2(\Omega))}^2+\Vert\nabla u_0\Vert_{L^2(\Omega)}^2+\Vert  u_1\Vert_{L^2(\Omega)}^2)
\end{align*}
with a constant $C>0$ independent of$t$.
Consequently, we  deduce the existence of $C>0$ (all generic constants are always denoted by $C$) such that
\begin{equation}\label{partaprioridampedGalerk2}
\int_0^{+\infty} \Vert \nabla u_m(s)\Vert_{L^2(\Omega)}^2\;ds\leq C ( \Vert f\Vert_{L^2([0,+\infty[;L^2(\Omega))}^2+\Vert\nabla u_0\Vert_{L^2(\Omega)}^2+\Vert  u_1\Vert_{L^2(\Omega)}^2)
\end{equation}

Fixing any $v\in H^1_0(\Omega)$ with $\Vert v\Vert_{H^1_0(\Omega)}\leq 1$, we decompose $v=v^{1}+v^{\perp}$ in a part $v^{1}\in span\lbrace w_k\rbrace_{k=1}^m$ and its orthogonal component $(v^{\perp},w_k)_{L^2(\Omega)}=0$ $(k=1,..,m)$. Let us notice that $\Vert v^1\Vert_{H^1_0(\Omega)}\leq 1$. Then Eqs.~(\ref{Galerksol}) and~(\ref{Galerkdampedeq}) imply 
\begin{align*}
\langle\partial^2_t u_m,v\rangle_{(H^{-1}(\Omega),H^1_0(\Omega))}&=(\partial^2_t u_m,v)_{L^2(\Omega)}=(\partial^2_t u_m,v^{1})_{L^2(\Omega)}\\
&=(f,v^{1})_{L^2(\Omega)}-c^2 (\nabla u_m,\nabla v^{1})_{L^2(\Omega)}-\eps \nu(\nabla\partial_t u_m,\nabla v^{1})_{L^2(\Omega)}.
\end{align*}
Thus, by the Cauchy-Schwartz inequality and the Poincar\'e inequality
\begin{multline*}
\vert \langle\partial^2_t u_m,v\rangle_{(H^{-1}(\Omega),H^1_0(\Omega))}\vert\leq  \Vert f\Vert_{L^2(\Omega)} \Vert v^{1}\Vert_{L^2(\Omega)}+c^2 \int_{\Omega} \vert \nabla u_m \nabla v^{1}\vert\;dx\\+\eps \nu\int_{\Omega} \vert \nabla\partial_t u_m \nabla v^{1}\vert\;dx 
\leq  C \Vert f\Vert_{L^2(\Omega)} \Vert v^{1}\Vert_{H^1_0(\Omega)}+ \Vert u_m\Vert_{H^1_0(\Omega)} \Vert v^{1}\Vert_{H^1_0(\Omega)}\\+\Vert\partial_t u_m\Vert_{H^1_0(\Omega)} \Vert v^{1}\Vert_{H^1_0(\Omega)},
\end{multline*}
then since $\Vert v^{1}\Vert_{H^1_0(\Omega)}\leq 1$,
\begin{multline*}
\vert \langle\partial^2_t u_m,v\rangle_{(H^{-1}(\Omega),H^1_0(\Omega))}\vert\leq C  (\Vert f\Vert_{L^2(\Omega)}+ \Vert u_m\Vert_{H^1_0(\Omega)}+\Vert\partial_t u_m\Vert_{H^1_0(\Omega)} ).
\end{multline*}
Consequently, by estimates~(\ref{partaprioridampedGalerk}) and (\ref{partaprioridampedGalerk2}) we deduce
\begin{align*}
\int_0^{+\infty} \Vert \partial^2_t u_m\Vert_{H^{-1}(\Omega)}^2dt\leq & C \int_0^{+\infty} (\Vert f\Vert_{L^2(\Omega)}^2+ \Vert u_m\Vert_{H^1_0(\Omega)}^2+\Vert\partial_t u_m\Vert_{H^1_0(\Omega)}) \;dt\\
\leq & C ( \Vert f\Vert_{L^2([0,+\infty[;L^2(\Omega))}^2+\Vert u_0\Vert_{H^1_0(\Omega)}^2+\Vert u_1\Vert_{L^2(\Omega)}^2),
\end{align*}
and, combining with estimates~(\ref{partaprioridampedGalerk}) and (\ref{partaprioridampedGalerk2}), we obtain~(\ref{aprioriGalerkdamped}). 
\end{proof}
Thus we establish the following well-posedness result: 
\begin{theorem}\label{ThExLinW}
For all arbitrary bounded domain $\Omega$, there exists a unique weak solution of the damped wave equation problem~(\ref{dampedwaveeqdirhom}) with the homogeneous Dirichlet boundary condition in the sense of Definition~\ref{defweakdampwav}.
\end{theorem}
\begin{proof}
Let us start to prove the existence.
According to the energy estimates~(\ref{aprioriGalerkdamped}), we see that the Galerkin sequences $\lbrace u_m\rbrace_{m=1}^{\infty}$ and $\lbrace \partial_t u_m\rbrace_{m=1}^{\infty}$ are bounded in $L^2([0,+\infty[;H^1_0(\Omega))$ and the sequence $\lbrace \partial^2_t u_m\rbrace_{m=1}^{\infty}$ is bounded in $L^2([0,+\infty[;H^{-1}(\Omega))$.
Then there exit a subsequence $\lbrace u_{m_l}\rbrace_{l=1}^{\infty}\subset \lbrace u_m\rbrace_{m=1}^{\infty}$ and $u\in L^2([0,+\infty[;H^1_0(\Omega))$ with $\partial_t u\in L^2([0,+\infty[;H^1_0(\Omega))$, $\partial^2_t u\in L^2([0,+\infty[;H^{-1}(\Omega))$ and such that
\begin{equation}\label{convuGalerkdamped}
\left\lbrace
\begin{array}{l}
u_{m_l}\rightharpoonup u\;\;\;\text{weakly in }L^2([0,+\infty[;H^1_0(\Omega)),\\
\partial_t u_{m_l}\rightharpoonup \partial_t u\;\;\;\text{weakly in }
L^2([0,+\infty[;H^1_0(\Omega)),\\
\partial^2_t u_{m_l}\rightharpoonup \partial^2_t u\;\;\;\text{weakly in }
L^2([0,+\infty[;H^{-1}(\Omega)).
\end{array}
\right.
\end{equation}
In the same way that for the wave equation problem (see Ref.~\cite{EVANS-1994} p. 384) we show that $u$ is a weak solution of the damped wave equation problem (\ref{dampedwaveeqdirhom}).

Let us prove the unicity.
By linearity, we need to show that the only weak solution of~(\ref{dampedwaveeqdirhom}) with $f\equiv 0$ and $ u_0\equiv u_1 \equiv 0$ is 
$ u\equiv 0.$
To prove it we take $\partial_t u\in L^2([0,+\infty[;H^1_0(\Omega))$. By definition of $u$ as a weak solution of the damped wave equation problem (\ref{dampedwaveeqdirhom}), for $T \geq 0$
$$\int_0^{T}\langle \partial^2_t u,\partial_t u\rangle_{(H^{-1}(\Omega),H^1_0(\Omega))}+c^2(\nabla u,\nabla \partial_t u)_{L^2(\Omega)}+\eps \nu(\nabla\partial_t u,\nabla \partial_t u)_{L^2(\Omega)}\;dt=0.$$
Then 
$$\frac{1}{2}\Vert \partial_t u(T)\Vert_{L^2(\Omega)}^2+\frac{c^2}{2} \Vert \nabla u(T)\Vert_{L^2(\Omega)}^2 +\int_0^T\Vert \nabla \del_t u(s)\Vert_{L^2(\Omega)}ds=0,$$
which implies $\nabla \del_t u\equiv 0$ and thus $\nabla u\equiv 0$, as $\nabla u_0=0$. Consequently, $u\equiv 0$, since $u\in H^1_0(\Omega)$ as a weak solution.
\end{proof}
\subsection{Maximal regularity results.}
\subsubsection{Homogeneous Dirichlet boundary condition.}\label{subsecMaxDirH}
We start by showing the regularity of the weak solution obtained in Theorem~\ref{ThExLinW}.
\begin{theorem}\label{dampedwaveeqregint1}
Let $\Omega$ be an arbitrary bounded domain in $\mathbb{R}^n$ ($n\geq2$) and $u$ be the weak solution of  problem (\ref{dampedwaveeqdirhom}) in the sense of~(\ref{varformdampedwaveeqdirhom}).

Then
\\(i) it has in addition the following regularity
$$u\in L^{\infty}([0,+\infty[;H^1_0(\Omega)),\;\;\partial_t u\in L^{\infty}([0,+\infty[;L^2(\Omega))$$
and satisfies the  estimate
\begin{eqnarray}
\operatorname{ess} \sup_{t \geq 0}(\Vert u(t)\Vert_{H^1_0(\Omega)} && +\Vert\partial_t u(t)\Vert_{L^2(\Omega)})+\int_0^{+\infty} \Vert \nabla\partial_t u(s)\Vert_{L^2(\Omega)}\;ds+ \Vert\partial^2_t u\Vert_{ L^2([0,+\infty[;H^{-1}(\Omega))}\nonumber\\
&&\leq C ( \Vert f\Vert_{L^2([0,+\infty[;L^2(\Omega))}+\Vert u_0\Vert_{H^1_0(\Omega)}+\Vert u_1\Vert_{L^2(\Omega)}).\nonumber
\end{eqnarray}
(ii) If the initial data are taken more regular
$$ \Delta u_0\in L^2(\Omega),\;\;u_1\in H^1_0(\Omega),$$ 
then the solution satisfies
\begin{eqnarray}
&&u\in L^{\infty}([0,+\infty[;H^1_0(\Omega))\cap L^2([0,+\infty[;H^1_0(\Omega))\cap L^{2}([0,+\infty[;H^2_{loc}(\Omega)) ,\nonumber\\
&&\partial_t u\in L^{\infty}([0,+\infty[;H^1_0(\Omega))\cap L^2([0,+\infty[;H^1_0(\Omega))\cap L^{2}([0,+\infty[;H^2_{loc}(\Omega)), \nonumber\\
&&\partial^2_t u\in L^{2}([0,+\infty[;L^2(\Omega)),\nonumber\\
&&\Delta u \in L^{\infty}([0,+\infty[;L^2(\Omega))\cap L^2( [0,+\infty[;L^2(\Omega)),\nonumber\\
&&\Delta \partial_t u \in L^2( [0,+\infty[;L^2(\Omega))\nonumber
\end{eqnarray}
with the following a priori estimates
\begin{eqnarray}
\operatorname{ess}\sup_{0 \leq t}(\Vert\Delta u(t)\Vert_{L^2(\Omega)}^2&&+ \Vert \nabla\partial_t u(t)\Vert_{L^2(\Omega)}^2 )+\int_0^{\infty}\Vert\Delta\partial_t u(s)\Vert_{L^2(\Omega)}^2 \;ds \nonumber\\
\leq && C ( \Vert f\Vert_{L^2([0,+\infty[;L^2(\Omega))}^2+\Vert\Delta u_0\Vert_{L^2(\Omega)}^2+\Vert \nabla u_1\Vert_{L^2(\Omega)}^2)\label{aprioriglobdampedwave}
\end{eqnarray}
and 
\begin{equation}\label{aprioriglobdampedwavebis}
\int_0^{+\infty} \Vert \Delta u(s)\Vert_{L^2(\Omega)}^2\;ds\leq C ( \Vert f\Vert_{L^2([0,+\infty[;L^2(\Omega))}^2+\Vert\Delta u_0\Vert_{L^2(\Omega)}^2+\Vert \nabla u_1\Vert_{L^2(\Omega)}^2) 
\end{equation}
with a constant $C>0$ depending only on $\Omega$.
\end{theorem}
\begin{proof}
Passing to limits in~(\ref{aprioriGalerkdamped}) as $m=m_l\rightarrow \infty$, we deduce (\textit{i}).

Assume now the hypothesis of assertion (\textit{ii}). We consider again $u_m$ found in Proposition~\ref{WpGalerksol}. We multiply Eq.~(\ref{Galerkdampedeq}) by $-\lambda_k \partial_t d^k_m(t) $ and sum over $k=1,...,m$. By definitions of $u_m$  and of $\lambda_k$ (see Eqs.~(\ref{Galerksol}) and~(\ref{eigenfuncDelta}) respectively) we have
$$\int_{\Omega}- \partial^2_t u_m \Delta \partial_t u_m -c^2 \nabla u_m\nabla \Delta \partial_t u_m -\eps \nu\nabla \partial_t u_m\nabla \Delta \partial_t u_m\;dx=-(f, \Delta \partial_t u_m)_{L^2(\Omega)}.$$
Thanks to the definition of $u_m$ (see Proposition~\ref{WpGalerksol})
as a linear combination of  Laplacian's eigenfunctions $w_k$,  satisfying~(\ref{ortgbasH10})--(\ref{eigenfuncDelta}), we have $\partial^2_t u_m\in H^1_0(\Omega)$ and $\Delta \partial_t u_m\in H^1_0(\Omega)$. Therefore, by the Green formula we have
$$\int_{\Omega}\nabla \partial^2_t u_m \nabla \partial_t u_m + c^2 \Delta u_m\Delta \partial_t u_m +\eps \nu\Delta \partial_t u_m\Delta \partial_t u_m\;dx=-(f, \Delta \partial_t u_m)_{L^2(\Omega)}.$$
Applying Young's inequality, we estimate the right-hand term
$$\vert (f,\Delta\partial_t u_m)_{L^2(\Omega)}\vert \leq \frac{2}{\eps \nu}\Vert f\Vert_{L^2(\Omega)}^2+\frac{\eps \nu}{2}\Vert \Delta \partial_t u_m\Vert_{L^2(\Omega)}^2.$$
Then we find
$$\frac{d}{dt}\Big(\frac{1}{2}\Vert \nabla\partial_t u_m(t)\Vert_{L^2(\Omega)}^2+\frac{c^2}{2}\Vert \Delta u_m(t)\Vert_{L^2(\Omega)}+\frac{\eps \nu}{2}\int_0^t \Vert \Delta \partial_t u_m(s)\Vert_{L^2(\Omega)}\;ds\Big)\leq \frac{2}{\eps \nu}\Vert f\Vert_{L^2(\Omega)}^2.$$
Integrating over $[0,t]$ we obtain
\begin{align}
\frac{1}{2}\Vert \nabla\partial_t u_m(t)\Vert_{L^2(\Omega)}^2+&\frac{c^2}{2}\Vert \Delta u_m(t)\Vert_{L^2(\Omega)}+\frac{\eps \nu}{2}\int_0^t \Vert \Delta \partial_t u_m(s)\Vert_{L^2(\Omega)}\;ds\label{aprioriglobdampedum}\\
\leq & \frac{2}{\eps \nu} \int_0^t\Vert f(s)\Vert_{L^2(\Omega)}^2\;ds+\frac{1}{2}\Vert \nabla \partial_t u_m(0)\Vert_{L^2(\Omega)}^2+\frac{1}{2}\Vert \Delta u_m(0)\Vert_{L^2(\Omega)}\nonumber\\
\leq & \frac{2}{\eps \nu} \Vert f\Vert_{L^2([0,+\infty[; L^2(\Omega))}^2+\frac{1}{2}\Vert \nabla u_1\Vert_{L^2(\Omega)}^2+\frac{1}{2}\Vert \Delta u_0\Vert_{L^2(\Omega)}\nonumber,
\end{align}
from where  we deduce~(\ref{aprioriglobdampedwave}) taking the weak limit of a subsequence.
Hence
\begin{align*}
&\partial_t u\in L^{\infty}([0,+\infty[;H^1_0(\Omega)) ,\quad \Delta u \in L^{\infty}([0,+\infty[;L^2(\Omega)),\quad \Delta \partial_t u \in L^2( [0,+\infty[;L^2(\Omega)).
\end{align*}
The linearity of the equation implies for all $T\geq 0$ 
$$\partial^2_t u\in L^2([0,T];L^2(\Omega)),$$
as $\Delta u \in L^{\infty}([0,+\infty[;L^2(\Omega))\subset L^2([0,T];L^2(\Omega)) $ for all $T\geq 0$.

Moreover, as $\partial_t u_m\in H^1_0(\Omega)$ and $\Delta \partial_t u_m\in L^2(\Omega)$ by Proposition \ref{propestadmdom} we have
$$\Vert \nabla \partial_t u_m\Vert_{L^2(\Omega)}\leq C \Vert \Delta \partial_t u_m\Vert_{L^2(\Omega)},$$
which implies by estimate (\ref{aprioriglobdampedum})
\begin{align*}
\Vert \nabla \partial_t u_m\Vert_{L^2([0,+\infty[; L^2(\Omega))}\leq & C \Vert \Delta \partial_t u_m\Vert_{L^2([0,+\infty[; L^2(\Omega))}\\
\leq & C ( \Vert f\Vert_{L^2([0,+\infty[;L^2(\Omega))}^2+\Vert\Delta u_0\Vert_{L^2(\Omega)}^2+\Vert \nabla u_1\Vert_{L^2(\Omega)}^2),
\end{align*}
and, taking the weak limit of a subsequence, we result in
\begin{equation}\label{aprioriglobdampedwave1}
\Vert  \partial_t u\Vert_{L^2([0,+\infty[; H^1_0(\Omega))}\leq  C ( \Vert f\Vert_{L^2([0,+\infty[;L^2(\Omega))}^2+\Vert\Delta u_0\Vert_{L^2(\Omega)}^2+\Vert \nabla u_1\Vert_{L^2(\Omega)}^2).
\end{equation}
By Proposition~\ref{propestadmdom} we also have
$$\Vert \nabla  u_m\Vert_{L^2(\Omega)}\leq C \Vert \Delta  u_m\Vert_{L^2(\Omega)},$$
which implies in the same way as for estimate~(\ref{aprioriglobdampedwave1}) 
\begin{equation}\label{aprioriglobdampedwave2}
\Vert  u\Vert_{L^{\infty}([0,+\infty[; H^1_0(\Omega))}\leq  C ( \Vert f\Vert_{L^2([0,+\infty[;L^2(\Omega))}^2+\Vert\Delta u_0\Vert_{L^2(\Omega)}^2+\Vert \nabla u_1\Vert_{L^2(\Omega)}^2).
\end{equation}
This time instead of multiplication of equations~(\ref{Galerkdampedeq}) by $-\lambda_k \del_t d^k_m(t) $, we multiply them by $-\lambda_k  d^k_m(t) $ and sum over $k=1,...,m$. Following the same steps as for the proof of estimate~(\ref{aprioriglobdampedwave1}) we result in 
$$\frac{d}{dt}\Big(\frac{\eps \nu}{2}\Vert \Delta u_m(t)\Vert_{L^2(\Omega)}^2+\frac{c^2}{2}\int_0^t \Vert \Delta u_m(s)\Vert_{L^2(\Omega)}^2\;ds\Big)\leq \frac{2}{c^2}\Vert f\Vert_{L^2(\Omega)}^2-\int_{\Omega}\nabla \partial^2_t u_m \nabla u_m\;dx.$$
Integrating over $[0,t]$ we obtain
\begin{align*}
\frac{\eps \nu}{2}\Vert \Delta u_m(t)\Vert_{L^2(\Omega)}^2+ & \frac{c^2}{2}\int_0^t \Vert \Delta u_m(s)\Vert_{L^2(\Omega)}^2\;ds\\
\leq & \frac{\eps \nu}{2}\Vert \Delta u_0\Vert_{L^2(\Omega)}^2+\frac{2}{c^2} \Vert f\Vert_{L^2([0,+\infty[;L^2(\Omega))}^2-\int_0^t \int_{\Omega}\nabla \partial^2_t u_m \nabla u_m\;dx\;ds.
\end{align*}
We express the last term in the previous inequality as
\begin{align*}
-\int_0^t \int_{\Omega}\nabla \partial^2_t u_m \nabla u_m\;dx = &  \int_0^t\left[-\frac{d}{dt}\left( \int_{\Omega}\nabla \partial_t u_m \nabla u_m\;dx\right)+\int_{\Omega} (\nabla \partial_t u_m)^2\;dx\right]\;ds\\
= & -\int_{\Omega}\nabla \partial_t u_m \nabla u_m\;dx+\int_{\Omega}\nabla \partial_t u_m(0) \nabla u_m(0)\;dx\\
& +\Vert \partial_t u_m\Vert_{L^2([0,t];H^1_0(\Omega))}^2,\\
\intertext{and obtain by Young's inequality}
-\int_0^t \int_{\Omega}\nabla \partial^2_t u_m \nabla u_m\;dx \leq &\frac{1}{2} \Vert \partial_t u_m\Vert_{L^{\infty}([0,t];H^1_0(\Omega))}^2 +\frac{1}{2} \Vert u_m\Vert_{L^{\infty}([0,t];H^1_0(\Omega))}^2\\
& +\frac{1}{2} \Vert \nabla \partial_t u_m(0)\Vert_{L^2(\Omega)} +\frac{1}{2}\Vert \nabla u_m(0)\Vert_{L^2(\Omega)}\\
& +\Vert \partial_t u_m\Vert_{L^2([0,t];H^1_0(\Omega))}^2\\
\leq &\frac{1}{2} \Vert \partial_t u_m\Vert_{L^{\infty}([0,+\infty[;H^1_0(\Omega))}^2 +\frac{1}{2} \Vert u_m\Vert_{L^{\infty}([0,+\infty[;H^1_0(\Omega))}^2\\
& +\frac{1}{2} \Vert \nabla u_1\Vert_{L^2(\Omega)} +\frac{1}{2}\Vert \nabla u_0\Vert_{L^2(\Omega)} +\Vert \partial_t u_m\Vert_{L^2([0,+\infty[;H^1_0(\Omega))}^2.
\intertext{Using now estimates (\ref{aprioriglobdampedum}),  (\ref{aprioriglobdampedwave1}) and (\ref{aprioriglobdampedwave2}) we find}
-\int_0^t \int_{\Omega}\nabla \partial^2_t u_m \nabla u_m\;dx\leq & C ( \Vert f\Vert_{L^2([0,+\infty[;L^2(\Omega))}^2+\Vert\Delta u_0\Vert_{L^2(\Omega)}^2+\Vert \nabla u_1\Vert_{L^2(\Omega)}^2)
\end{align*}
with $C>0$ independent of$t$.
Thus, we can deduce 
$$\int_0^{+\infty} \Vert \Delta u_m(s)\Vert_{L^2(\Omega)}^2\;ds\leq C ( \Vert f\Vert_{L^2([0,+\infty[;L^2(\Omega))}^2+\Vert\Delta u_0\Vert_{L^2(\Omega)}^2+\Vert \nabla u_1\Vert_{L^2(\Omega)}^2)  $$
and taking a convergent subsequence we find in the limit
$$\int_0^{+\infty} \Vert \Delta u(s)\Vert_{L^2(\Omega)}^2\;ds\leq C ( \Vert f\Vert_{L^2([0,+\infty[;L^2(\Omega))}^2+\Vert\Delta u_0\Vert_{L^2(\Omega)}^2+\Vert \nabla u_1\Vert_{L^2(\Omega)}^2) . $$
The linearity of the equation gives us $\partial^2_tu\in L^2([0,+\infty[;L^2(\Omega))$, since
$$\Vert \partial^2_tu\Vert_{L^2([0,+\infty[;L^2(\Omega))}\leq C (\Vert f \Vert_{L^2([0,+\infty[;L^2(\Omega))}+\Vert \Delta u \Vert_{L^2([0,+\infty[;L^2(\Omega))}+\Vert \Delta \partial u \Vert_{L^2([0,+\infty[;L^2(\Omega))}).$$
By Ref.~\cite{EVANS-1994}  Theorem~1 p.~309, we can also deduce
$$\partial_t u\in L^2([0,+\infty[;H^2_{loc}(\Omega))$$
with the estimate on $V\subset\subset\Omega$
$$\Vert \partial_t u\Vert_{L^2([0;+\infty[;H^2(V))}\leq C \Vert\Delta \partial_t u\Vert_{L^2([0;+\infty[;L^2(\Omega))},$$
which implies
$$\Vert \partial_t u\Vert_{L^2((0;+\infty);H^2(V))}^2\leq C ( \Vert f\Vert_{L^2([0,+\infty[;L^2(\Omega))}^2+\Vert\Delta u_0\Vert_{L^2(\Omega)}^2+\Vert \nabla u_1\Vert_{L^2(\Omega)}^2).$$
Repeating the previous argument it is sufficient to obtain
$$u\in L^{\infty}([0,+\infty[;H^2_{loc}(\Omega))\cap L^{2} ([0,+\infty[;H^2_{loc}(\Omega)),$$
as  $\Delta u \in L^{\infty}([0,+\infty[;L^2(\Omega))\cap L^{2}([0,+\infty[;L^2(\Omega))$, which finishes the proof.
\end{proof}
Let us now define the domain of the weak Dirichlet Laplacian operator, which we use to improve the regularity of solutions for which as it was mentioned we don't have $H^2$-regularity. 
\begin{definition}[Weak Laplacian]\label{DefOpA}
The Laplacian operator $-\Delta$ 
on $\Omega$ is considered  with homogeneous Dirichlet boundary conditions in the weak sense: 
\begin{align*}
-\Delta:\mathcal{D}(-\Delta)\subset H^1_0(\Omega)&\rightarrow L^2(\Omega) \\
u&\mapsto -\Delta u.
\end{align*}
Here $\mathcal{D}(-\Delta)$ is the domain of $-\Delta$, defined in the following way: $u\in \mathcal{D}(-\Delta)$ if and only if $u\in H^1_0(\Omega)$ and $-\Delta u\in L^2(\Omega)$ in the sense that there exists $f\in L^2(\Omega)$ such that
$$\forall v \in H^1_0(\Omega)\quad (\nabla u,\nabla v)_{L^2(\Omega)}=(f,v)_{L^2(\Omega)} .$$
The operator $-\Delta$ is linear self-adjoint and coercive in the sense that for $u\in \mathcal{D}(-\Delta)$
$$(-\Delta u, u)_{L^2(\Omega)}=(\nabla u,\nabla u)_{L^2(\Omega)} $$
 and we use the  notation 
$$\Vert u\Vert_{ \mathcal{D}(-\Delta)}=\Vert \Delta u\Vert_{L^2(\Omega)} \quad \hbox{for } u\in \mathcal{D}(-\Delta).$$
\end{definition}
In addition we need to introduce the following space:
\begin{definition}\label{DefSpaceX}
For $U$ an open set we define the Hilbert space
\begin{equation}\label{Hspace}
X(U)=H^1([0,+\infty[;\mathcal{D}(-\Delta))\cap H^2([0,+\infty[;L^2(U)).
\end{equation}
with the norm
$$\Vert u\Vert_{X(U)}^2=\Vert \Delta u\Vert_{L^2([0,+\infty[;L^2(U))}^2 +\Vert \Delta \partial_t u\Vert_{L^2([0,+\infty[;L^2(U))}^2+\Vert \partial_t^2 u\Vert_{L^2([0,+\infty[;L^2(U))}^2$$ associated to the scalar product
$$(u,v)_{X(U)}=\int_0^{+\infty}(\Delta u,\Delta v)_{L^2(U)}+(\Delta \partial_t u,\Delta \partial_t v)_{L^2(U)}+(\partial^2_t u,\partial^2_t v)_{L^2(U)}ds.$$
\end{definition}

\begin{theorem}\label{thmeqivcaudampwav}
 For $\Omega$ an arbitrary bounded domain in $\mathbb{R}^n$ 
 and
 \begin{equation}\label{HspaceX0}
 X_0(\Omega):= \lbrace u\in X(\Omega)\vert u(0)=0,\;\partial_t u(0)=0\rbrace 
 \end{equation}
there exists a unique weak solution of the Dirichlet homogeneous problem for the strongly damped wave equation with also homogeneous initial conditions, in the sense of the variational formulation~(\ref{varformdampedwaveeqdirhom}),
$u\in X_0(\Omega)$ 
if and only if
$f\in L^2([0,+\infty[;L^2(\Omega))$. 

Moreover we have the estimate
\begin{equation}\label{EqAElinhom}
	\Vert u\Vert_{X(\Omega)}\leq C \Vert f\Vert_{L^2([0,+\infty[;L^2(\Omega))}.
\end{equation}
\end{theorem}
\begin{proof}
  If $f\in L^2([0,+\infty[;L^2(\Omega))$, $u_0\in \mathcal{D}(-\Delta)$ and $u_1\in H^1_0(\Omega)$, by Theorems~\ref{ThExLinW} and~\ref{dampedwaveeqregint1} there exists a unique $u$ weak solution of problem (\ref{dampedwaveeqdirhom}) in the sense of~(\ref{varformdampedwaveeqdirhom}) with $\partial^2_t u$, $\Delta u$ and $\Delta\partial_t u$ in $L^2([0,+\infty[;L^2(\Omega))$. The estimates~(\ref{aprioriglobdampedwave}) and (\ref{aprioriglobdampedwavebis}) are also satisfied which implies $u\in X(\Omega)$ with the desired estimate.
Conversely, let us consider a weak solution $u\in H^1([0,+\infty[;\mathcal{D}(-\Delta))\cap H^2([0,+\infty[;L^2(\Omega))$ with $u(0)=0$, $\partial_t u(0)=0$ of the homogeneous problem~(\ref{dampedwaveeqdirhom}) for the strongly damped wave equation. By linearity $u$ is unique and, by regularity of $u$, we have
$f\in L^2([0,+\infty[;L^2(\Omega))$. 
\end{proof}
\subsubsection{Non homogeneous Dirichlet boundary condition.}\label{subssecWPnHLin}
In this Subsection $\Omega$ is an admissible domain in $\mathbb{R}^n$ ($n=2$ or $3$) with a $d-$set boundary $\partial\Omega$ for $n-2<d<n$ and
$-\Delta$ is the weak Dirichlet Laplacian operator from Definition~\ref{DefOpA}. 

The fact to have non homogeneous Dirichlet boundary conditions implies the use of traces and extensions operators which leads us to leave the field of arbitrary domains.
\begin{theorem}\label{strondampwavedirnonhom}
Let $\beta_2=2-\frac{n-d}{2}>0$ and
\begin{equation}\label{regtrspacedir}
F:=H^{1}([0,+\infty[;B^{2,2}_{\beta_2}(\partial\Omega))\cap H^{\frac{7}{4}}([0,+\infty[;L^2(\partial\Omega)).
\end{equation}
For $u_0\in H^2(\Omega)$, $u_1\in H^1(\Omega)$, $g\in F$ and $f\in L^2([0,+\infty[;L^2(\Omega))$
with the compatibility conditions
$$g(0)=Tr_{\partial\Omega}u_0,\quad  \partial_tg(0)=Tr_{\partial\Omega} u_1$$
 there exists a unique weak solution $\tilde{u}$ of the problem
\begin{equation}\label{cauchydampwavdirnonhom}
\left\lbrace
\begin{array}{l}
\partial_t^2 \tilde{u}-c^2 \Delta \tilde{u}-\nu \varepsilon \Delta \partial_t \tilde{u}=f\hbox{ in }[0,+\infty[\times\Omega,\\
\tilde{u}=g \hbox{ on }\partial\Omega,\\
\tilde{u}(0)=u_0,\;\;\partial_t \tilde{u}(0)=u_1.
\end{array}
\right.
\end{equation}
It is a weak solution in such a way that $\tilde{u}=u^*+\overline{g}$ with
\begin{equation}\label{spacesolexttrace}
\overline{g}\in X_1(\Omega):= H^2([0,+\infty[;L^2(\Omega))\cap H^1([0,+\infty[;H^2(\Omega)),
\end{equation}
such that
$Tr_{\partial\Omega}\overline{g}=g $
 and 
with $u^*\in X(\Omega)$ (see~(\ref{Hspace}) for the definition of the space $X(\Omega)$), which  is the unique weak solution of the system
\begin{equation}\label{syststaru}
\left\lbrace
\begin{array}{l}
\partial^2_t u+c^2(-\Delta) u+\nu \varepsilon\partial_t (-\Delta) u =f-\partial_t^2 \overline{g}+c^2 \Delta \overline{g}+\nu \varepsilon \Delta \partial_t \overline{g}, \\
u(0)=u_0-\overline{g}(0), \hbox{  } \partial_t u(0)=u_1-\partial_t\overline{g}(0)
\end{array}
\right.
\end{equation}
in the sense of the variational formulation~(\ref{varformdampedwaveeqdirhom}).
Moreover, it holds the a priori estimate
\begin{equation}\label{EqAPLinNh}
\Vert u^*\Vert_{X(\Omega)}\leq C (\Vert f\Vert_{L^2([0,+\infty[;L^2(\Omega))}+\Vert  u_0\Vert_{H^2(\Omega)}+\Vert u_1\Vert_{H^1(\Omega)}+\Vert g\Vert_F).	
\end{equation}
\end{theorem}
\begin{proof}
By Lemma 3.5 from Ref.~\cite{DenkHiebPruss} we have the following continuous embedding:
\begin{eqnarray*}
&&H^2([0,+\infty[;L^2(\mathbb{R}^{n-1}\times\R^+))\cap H^1([0,+\infty[;H^2(\mathbb{R}^{n-1}\times\R^+))\\
&&\hookrightarrow H^{\frac{7}{4}}([0,+\infty[;L^2(\mathbb{R}^{n-1}))\cap  H^{1}([0,+\infty[;H^{\frac{3}{2}}(\mathbb{R}^{n-1})).
\end{eqnarray*}
Thus, as $g\in F$, defined in (\ref{regtrspacedir}), the existence of $\overline{g}\in X_1(\Omega)$ 
 (for the definition of the space $X_1(\Omega)$ see (\ref{spacesolexttrace}))
with $Tr_{\partial\Omega}\overline{w}=g$ comes from the properties of the trace operator, which thanks to Theorem~\ref{thmtradmissdom} has a bounded linear right inverse, $i.e.$ the  extension operator $E_{\partial\Omega}$. %
Moreover, the boundness of $E_{\partial\Omega}$ implies
\begin{equation}\label{estgba}
\Vert \overline{g}\Vert_{X_1(\Omega)}\leq C \Vert g\Vert_{X(\Omega)}.
\end{equation}
Let us define $u^*$ as a solution of system (\ref{syststaru}).

The regularity of $\overline{g}$ ensures
$$-\partial_t^2 \overline{g} +c^2 \Delta \overline{g} +\nu \varepsilon \Delta \partial_t \overline{g}\in L^2([0,T];L^2(\Omega)), \;\; 
u_0-\overline{g}(0)\in H^2(\Omega),\;\;u_1-\partial_t \overline{g}(0)\in H^1(\Omega).$$
The compatibility conditions also allows to have
\begin{align*}
Tr_{\partial\Omega}( u_0-\overline{g}(0))=Tr_{\partial\Omega}( u_0)-g(0)=0,\\ 
Tr_{\partial\Omega}( u_1-\partial_t \overline{g}(0))=Tr_{\partial\Omega}( u_1)-\partial_tg(0)=0.
\end{align*}
Then we  apply Theorem \ref{thmeqivcaudampwav} to obtain the existence of a unique solution $u^*$ of system (\ref{syststaru}) with the desired regularity. 
The regularity of $u_0$, $u_1$ and $\overline{g}$ with the help of estimate~(\ref{EqAElinhom}) in Theorem~\ref{thmeqivcaudampwav} allows to deduce estimate~(\ref{EqAPLinNh}).
\end{proof}

\section{Well-posedness of the Westervelt equation with Dirichlet boundary conditions.}\label{SecWPW}
\subsection{Well-posedness with homogeneous boundary condition.}\label{secwpWesdirhom}
In this section $\Omega$ is an arbitrary bounded domain in $\mathbb{R}^3$ or a bounded NTA domain in $\mathbb{R}^2$ and
$-\Delta$ is the weak Laplacian defined in Definition~\ref{DefOpA}. 

To be able to give a sharp estimate of the smallness of the initial data and in the same time to estimate the bound of the corresponding solution of the Westervelt equation, we use the following theorem from Ref.~\cite{Sukhinin}:
\begin{theorem}[Sukhinin]\label{thSuh}
 Let $X$ be a Banach space, let $Y$ be a separable
topological vector space, let $L : X \rightarrow Y$ be a linear
continuous operator, let $U$ be the open unit ball in $X$, let ${\rm
P}_{LU}:LX \to [0,\infty [$ be the Minkowski functional of the set
$LU$, and let $\Phi :X \to LX$ be a mapping satisfying the condition
\begin{equation*}
 {\rm P}_{LU} \bigl(\Phi (x) -\Phi (\bar{x})\bigr) \leq
\Theta (r) \left\|x -\bar{x} \right\|\quad \text{for} \quad \left\|x
-x_0 \right\| \leqslant r,\quad \left\|\bar{x} -x_0 \right\| \leq r
\end{equation*} for some $x_0 \in X,$ where $\Theta :[0,\infty [ \to [0,\infty [$ is a monotone
non-decreasing function. Set $b(r) =\max \bigl(1 -\Theta (r),0
\bigr)$ for $r \geq 0$.

 Suppose that $$w =\int\limits_0^\infty b(r)\,dr \in ]0,\infty ], \quad r_* =\sup \{ r
\geq 0|\;b(r) >0 \},$$

$$w(r) =\int\limits_0^r b(t)dt \quad (r \geq 0) \quad\hbox{and} \quad g(x) =Lx
+\Phi(x) \quad \hbox{for} \quad x \in X.$$
Then for any $r \in
[0,r_*[$ and $ y \in g(x_0) +w(r)LU$, there exists an
 $ x \in x_0 +rU$ such that $g(x) =y$.
\end{theorem}
\begin{remark} \label{remch22.1.} If either $L$ is injective or $KerL$ has a topological
complement $E$ in $X$ such that $L(E \cap U) =LU$, then the
assertion of Theorem~\ref{thSuh} follows from the contraction
mapping principle~\cite{Sukhinin}. In particular, if $L$ is injective,
then the solution is unique.
\end{remark}
We use Theorem~\ref{thSuh} to prove the following global well-posedness result.
\begin{theorem}\label{ThWPWestGlob}
We take $\Omega$ as an arbitrary bounded domain in $\mathbb{R}^3$ or a bounded NTA domain in $\mathbb{R}^2$.
 Let $\nu>0$, and $\R^+=[0,+\infty[$. Let $X(\Omega)$ be the Hilbert space defined in~(\ref{Hspace}),
$$u_0\in\mathcal{D}(-\Delta)  ,\quad u_1\in H^1_0(\Omega)\quad \hbox{and} \quad f\in L^2(\R^+;L^2(\Omega)) $$
and in addition $C_1=O(1)$ is the minimal constant such that the weak
  solution, in the sense of~(\ref{varformdampedwaveeqdirhom}), $u^*\in X(\Omega)$ of the corresponding non homogeneous linear boundary-valued problem (\ref{dampedwaveeqdirhom}) satisfies
  $$\|u^*\|_{X(\Omega)}\le \frac{C_1}{\nu \eps}(\Vert f\Vert_{L^2(\R^+;L^2(\Omega))} +\Vert u_0\Vert_{\mathcal{D}(-\Delta)}+\Vert u_1\Vert_{H^1_0(\Omega)}).$$
  Then there exists $r_{*}>0$ with $r_*=O(1)$ such that for all  $r\in[0,r_{*}[$
  and all 
  data  satisfying  
$$
 \Vert f\Vert_{L^2(\R^+;L^2(\Omega))} +\Vert u_0\Vert_{\mathcal{D}(-\Delta)}+\Vert u_1\Vert_{H^1_0(\Omega)}\le \frac{\nu \eps}{C_1}r,
$$
there exists the unique weak solution $u\in X(\Omega)$ of the following  problem for the Westervelt equation
\begin{equation}\label{CauchypbWesdirhom}
\left\lbrace
\begin{array}{l}
\partial^2_t u-c^2\Delta u-\nu \varepsilon \partial_t \Delta u=\alpha\varepsilon u \partial^2_t u+\alpha\varepsilon (\partial_t u)^2+f\quad on\quad [0,+\infty[\times\Omega,\\
u=0\quad on \quad [0,+\infty[\times\partial\Omega,\\
u(0)=u_0, \hbox{ }\partial_t u(0)=u_1.
\end{array}\right.
\end{equation} 
Here $u$ is a weak in the sense that $u=u^*+v$ where $u^*\in X(\Omega)$ is the weak solution of the variational formulation~(\ref{varformdampedwaveeqdirhom}) and $v\in X(\Omega)$ is the solution of an homogeneous non linear initial-boundary valued problem depending on $u^*$ and determined by Theorem~\ref{thSuh}. More precisely, $u\in X(\Omega)$ is the weak solution of the following variational formulation for all $\phi\in L^2([0,+\infty[;H^1_0(\Omega))$
\begin{align}
\int_0^{+\infty}&(\partial^2_t u,\phi)_{L^2(\Omega)}+c^2(\nabla u,\nabla \phi)_{L^2(\Omega)}+\nu \varepsilon (\nabla \partial_t u,\nabla \phi)_{L^2(\Omega)}ds\nonumber\\
&=\int_0^{+\infty} \alpha \varepsilon (u \partial^2_tu+ (\partial_t u)^2+f,\phi)_{L^2(\Omega)}ds\label{weakformWes}
\end{align}
with $u(0)=u_0$ and $\partial_t u(0)=u_1$.
Moreover 
$$\Vert u\Vert_{X(\Omega)}\leq 2 r.$$
\end{theorem}
\begin{proof}
For $u_0\in \mathcal{D}(-\Delta)$ and $u_1\in H^{1}_0(\Omega)$ and $f\in L^2(\R^+;L^2(\Omega))$  let us denote, by Theorem \ref{thmeqivcaudampwav}, $u^*\in X(\Omega)$ is the unique weak solution of the linear problem
$$
\begin{cases}
\del_t^2 u^*-c^2\Delta u^*-\nu \varepsilon \Delta \del_t u^*=f \quad on\quad [0,+\infty[\times\Omega,\\
u=0\quad on \quad [0,+\infty[\times\partial\Omega,\\
u^*(0)=u_0\in \mathcal{D}(-\Delta),\;\;\del_t u^*(0)=u_1\in H^{1}_0(\Omega),
\end{cases}
$$
in the sense of the variational formulation~(\ref{varformdampedwaveeqdirhom}).
In addition, according to Theorem~\ref{thmeqivcaudampwav}, we consider $X(\Omega)$ and $X_0(\Omega)$, defined in Definition~\ref{DefSpaceX} and in~(\ref{HspaceX0}) respectively, and introduce the Banach space 
 $Y=L^2(\R^+;L^2(\Omega))$. Then by Theorem~\ref{thmeqivcaudampwav}, the linear operator 
$$L:X_0(\Omega)\rightarrow Y,\quad  u\in X_0(\Omega)\mapsto\;L(u):=\del^2_tu+c^2(-\Delta) u+\nu \varepsilon \del_t (-\Delta) u\in Y$$
 is a bi-continuous isomorphism. 
 
 Let us now notice that if $v$ is the unique solution of the non-linear boundary valued problem 
 \begin{equation}\label{SystwesnV}
\begin{cases}
\del_t^2 v+c^2(-\Delta) v+\nu\varepsilon (-\Delta) \del_t v-\alpha \varepsilon (v+u^*)\del_t^2(v+u^*)-\alpha \varepsilon [\del_t(v+u^*)]^2=0, \\
v(0)=0,\quad \del_t v(0)=0,
\end{cases}
\end{equation}
 then $u=v+u^*$ is the unique weak solution of the boundary valued problem for the Westervelt equation~(\ref{CauchypbWesdirhom}).
 Let us prove the existence of a such $v$, using Theorem~\ref{thSuh}. Note that $v$ will be a weak solution with $v\in X_0(\Omega)$, in the sense that $\forall \phi\in L^2([0,+\infty[;H^1_0(\Omega))$
 \begin{align*}
\int_0^{+\infty}&(\partial^2_t v,\phi)_{L^2(\Omega)}+c^2(\nabla v,\nabla \phi)_{L^2(\Omega)}+\nu \varepsilon (\nabla \partial_t v,\nabla \phi)_{L^2(\Omega)}ds\nonumber\\
&=\int_0^{+\infty} \alpha \varepsilon ((v+u^*)\del_t^2 (v+u^*)+[\del_t(v+u^*)]^2,\phi)_{L^2(\Omega)}ds
 \end{align*}
 with $v(0)=0$ and $\partial_t v(0)=0$.
 
We suppose that $\Vert u^*\Vert_{X(\Omega)}\leq r$
and define for $v\in X_0(\Omega)$
$$\Phi(v):=\alpha \varepsilon (v+u^*)\del_t^2 (v+u^*)+\alpha \varepsilon [\del_t(v+u^*)]^2.$$

For $w$ and $z$ in $X_0(\Omega)$ such that 
$\Vert w\Vert_{X(\Omega)}\leq r$ and $\Vert z\Vert_{X(\Omega)}\leq r$,
 we estimate
$$
 \Vert \Phi(w)-\Phi(z)\Vert_Y,
$$
by applying the triangular inequality.
The key point is that it appears terms of the form $\Vert a \del_t^2 b\Vert_Y$ and $\Vert \del_ta \del_t b\Vert_Y$
with $a$ and $b$ in $X(\Omega)$ and we have the estimate
\begin{align*}
\Vert a \del_t^2 b\Vert_Y\leq &\Vert a\Vert_{L^{\infty}(\mathbb{R}^+\times\Omega)}\Vert \del_t^2 b\Vert_Y,\\
\intertext{by Corollary~\ref{PoissonR3} in $\mathbb{R}^3$ and Corollary~\ref{linfPoissonR2} for NTA domains in $\mathbb{R}^2$ we have}
\Vert a \del_t^2 b\Vert_Y\leq & C \Vert a\Vert_{L^{\infty}(\mathbb{R}^+;\mathcal{D}(-\Delta))}\Vert b\Vert_{X(\Omega)}\\
\intertext{and the Sobolev embedding implies}
\Vert a \del_t^2 b\Vert_Y\leq & C \Vert a\Vert_{H^1(\mathbb{R}^+;\mathcal{D}(-\Delta))}\Vert b\Vert_{X(\Omega)}
\\
\leq & B_1 \Vert a\Vert_{X(\Omega)} \Vert b\Vert_{X(\Omega)},
\end{align*}
with a constant $B_1>0$ depending only on $\Omega$.
Moreover we have 
\begin{align*}
\Vert \del_t a \del_t b\Vert_Y\leq & \sqrt{\int_{0}^{+\infty} \Vert \del_t a\Vert_{L^{\infty}(\Omega)} \Vert \del_t b\Vert_{L^2(\Omega)}ds}.
\intertext{Therefore, by Corollary~\ref{PoissonR3} in $\mathbb{R}^3$ and Corollary~\ref{linfPoissonR2} for NTA domains in $\mathbb{R}^2$ we have}
\Vert \del_t a \del_t b\Vert_Y\leq & C \sqrt{\int_{0}^{+\infty} \Vert \del_t a\Vert_{\mathcal{D}(-\Delta)} \Vert \del_t b\Vert_{L^2(\Omega)}ds}\\
\leq & C \Vert \del_t a\Vert_{L^{2}(\R^+; \mathcal{D}(-\Delta))}\Vert \del_t b\Vert_{L^{\infty}(\R^+; L^2(\Omega))}\\
\leq & C \Vert a\Vert_{X(\Omega)} \Vert \del_t b\Vert_{H^1(\R^+; L^2(\Omega))}
\end{align*}
also using Sobolev's embeddings. As a result we have
$$\Vert \del_t a \del_t b\Vert_Y\leq  B_2 \Vert a\Vert_{X(\Omega)} \Vert b\Vert_{X(\Omega)},$$
with a constant $B_2>0$ depending only on $\Omega$.
Taking $a$ and $b$ equal to $u^*$, $w$, $z$ or $w-z$, as $\Vert u^*\Vert_{X(\Omega)}\leq r$, $\Vert w\Vert_{X(\Omega)}\leq r$ and $\Vert z\Vert_{X(\Omega)}\leq r$, we obtain
\begin{align*}
\Vert \Phi(w)-\Phi(z)\Vert_Y\leq 8
\alpha B \varepsilon r \Vert w-z\Vert_{X(\Omega)}
\end{align*}
with $B=\max(B_1,B_2)>0$, depending only on $\Omega$.

By the fact that $L$ is a bi-continuous isomorphism, there exists a minimal constant $C_\eps=O\left(\frac{1}{\eps \nu} \right)>0$ (coming from the inequality $C_0 \eps \nu\|u\|_{X(\Omega)}^2\le \|f\|_Y\|u\|_{X(\Omega)}$ for $u$, a solution of the linear problem~(\ref{dampedwaveeqdirhom}) with homogeneous initial data [for a maximal possible constant $C_0=O(1)>0$])
such that
$$\forall u\in X_0(\Omega) \quad \Vert u\Vert_{X(\Omega)}\leq C_\eps \Vert Lu\Vert_Y.$$
Hence, for all $g\in Y$
$$P_{LU_{X_0(\Omega)}}(g)\leq C_\eps P_{U_Y}(g)=C_\eps\Vert g\Vert_Y.$$
Then we find for $w$ and $z$ in $X_0(\Omega)$, such that $\|w\|_{X(\Omega)}\le r$, $\|z\|_{X(\Omega)}\le r$, and also with $\|u^*\|_{X(\Omega)}\le r$, that
$$P_{LU_{X_0(\Omega)}}(\Phi(w)-\Phi(z))\leq \Theta(r) \Vert w-z\Vert_{X(\Omega)},$$
where $\Theta(r):=8 B C_\eps \alpha\varepsilon r$.
Thus we apply Theorem~\ref{thSuh} for 
$g(x)=L(x)-\Phi(x)$ and $x_0=0$. Therefore, knowing that $C_\eps=\frac{C_0}{\eps \nu}$, we have, that for all  $r\in[0,r_{*}[$ with
\begin{equation}\label{Eqret}
 r_{*}=\frac{\nu}{8 B C_0 \alpha}=O(1),
\end{equation}
  for all $y\in \Phi(0)+w(r) L U_{X_0(\Omega)}\subset Y$
with $$w(r)= r-4 B \frac{C_0}{\nu}  \alpha r^2,$$
there exists a unique $v\in 0+r U_{X_0(\Omega)}$ such that $L(v)-\Phi(v)=y$.
But, if we want that $v$ be the solution of the non-linear Cauchy problem~(\ref{SystwesnV}), then we need to impose $y=0$, and thus to ensure that $0\in \Phi(0)+w(r) L U_{X_0(\Omega)}$.
Since $-\frac{1}{w(r)}\Phi(0)$ is an element of $Y$ and $LX_0(\Omega)=Y$, there exists a unique $z\in X_0(\Omega)$ such that
\begin{equation}\label{Eqz}
 L z=-\frac{1}{w(r)}\Phi(0).
\end{equation}
Let us show that $\|z\|_{X(\Omega)}\le 1$, what will implies that $0\in \Phi(0)+w(r) L U_{X_0(\Omega)}$.
Noticing that
\begin{align*}
\Vert \Phi(0)\Vert_Y & \leq \alpha \varepsilon \Vert \del_t u^* \del_t^2 u^*\Vert_Y +\alpha \varepsilon \Vert \del_t u^* \del_t u^*\Vert_Y\\
& \leq 2 \alpha \varepsilon B \Vert u^*\Vert_{X(\Omega)}^2 \\
& \leq 2 \alpha \varepsilon B r^2
\end{align*}
and using~(\ref{Eqz}), we find
\begin{align*}
  \Vert z\Vert_{X(\Omega)} \leq C_\eps\Vert L z\Vert_Y=C_\eps\frac{\Vert \Phi(0)\Vert_Y}{w(r)}\leq \frac{C_\eps 2 B \alpha \varepsilon r}{(1-4 C_\eps B \alpha \varepsilon r)}<\frac{1}{2},
\end{align*}
as soon as $r<r^*$.

Consequently, $z\in U_{X_0(\Omega)}$ and $\Phi(0)+w(r) Lz=0$.

Then we conclude that  for all  $r\in[0,r_{*}[$, if $\|u^*\|_{X(\Omega)}\le r$, there exists a unique $v\in r U_{X_0(\Omega)}$ such that $L(v)-\Phi(v)=0$, $i.e.$  the solution of the non-linear  problem~(\ref{SystwesnV}).
Thanks to the maximal regularity and a priori estimate following from Theorem~\ref{thmeqivcaudampwav},
there exists a constant $C_1=O(\eps^0)>0$, such that
$$\|u^*\|_{X(\Omega)}\le \frac{C_1}{\nu \varepsilon}(\Vert f\Vert_Y+\Vert u_0\Vert_{\mathcal{D}(-\Delta)}+\Vert u_1\Vert_{H^1_0(\Omega)}).$$

Thus, for all  $r\in[0,r_{*}[$ and $\Vert f\Vert_Y+\Vert u_0\Vert_{\mathcal{D}(-\Delta)}+\Vert u_1\Vert_{H^1_0(\Omega)} \leq \frac{\nu \eps}{C_1}r$, the function $u=u^*+v\in X(\Omega)$ is the unique solution of the abstract Cauchy problem for the Westervelt equation and $\Vert u\Vert_{X(\Omega)}\leq 2 r$.
\end{proof}
\subsection{Well posedness with non homogeneous Dirichlet boundary conditions on admissible domains.}\label{secwpWesdirinhom}
As we consider the non homogeneous Dirichlet boundary conditions,   in this section $\Omega$ is an admissible domain in $\mathbb{R}^n$ ($n=2$ or $3$), on which the trace and extension operators are well defined. 

More precisely we consider the following non homogeneous Dirichlet boundary problem for the the Westervelt equation
\begin{equation}\label{systWestnonhom}
\left\lbrace
\begin{array}{l}
\partial_t^2 u-c^2 \Delta u -\nu \varepsilon \Delta \partial_t u=\alpha\varepsilon u\partial_t^2 u+\alpha \varepsilon (\partial_t u)^2+f \hbox{ in }[0,+\infty[\times\Omega,\\
u=g\hbox{ on }\partial\Omega,\\
(u(0),\partial_t u(0))=(u_0,u_1)\hbox{ on }\Omega.
\end{array}
\right.
\end{equation}
We are looking for solution of the form $u=v+\tilde{u}$, where $\tilde{u}$ solves the strongly damped system~(\ref{cauchydampwavdirnonhom}) in a weak sense according to Theorem~\ref{strondampwavedirnonhom}. It lead us to $v$, the solution of the following abstract Cauchy system 
\begin{equation}\label{systWestbishom}
\begin{array}{l}
\partial_t^2 v+c^2 (-\Delta v) +\nu \varepsilon (-\Delta) \partial_t v=\alpha\varepsilon (v+\tilde{u})\partial_t^2 (v+\tilde{u})+\alpha \varepsilon [\partial_t (v+\tilde{u})]^2,\\
(v(0),\partial_t v(0))=(0,0), 
\end{array}
\end{equation}
the well posedness of which is determined in the following theorem using Theorem~\ref{thSuh}.
\begin{theorem}\label{Globwelposvdir}
Let $X(\Omega)$ and $X_1(\Omega)$ be defined by~(\ref{Hspace}) and~(\ref{spacesolexttrace}) respectively.
For $u^*\in X(\Omega)$ and $\overline{g}\in X_1(\Omega)$, there exits $r^*>0$ with $r^*=O(1)$, such that for $r\in [0,r^*[$ and all data satisfying 
$$\Vert u^*\Vert_{X(\Omega)}\leq \frac{r}{2}\; and\; \Vert \overline{g}\Vert_{X_1(\Omega)}\leq \frac{r}{2},$$ there exists the unique weak solution $v\in X(\Omega)$ of the nonlinear problem~(\ref{systWestbishom}) with $\tilde{u}=u^*+\overline{g}$ and $\Vert v\Vert_{X(\Omega)}\leq r$.

Note that $v$ will be a weak solution in the sense, where $\forall \phi\in L^2([0,+\infty[;H^1_0(\Omega))$
 \begin{align*}
\int_0^{+\infty}&(\partial^2_t v,\phi)_{L^2(\Omega)}+c^2(\nabla v,\nabla \phi)_{L^2(\Omega)}+\nu \varepsilon (\nabla \partial_t v,\nabla \phi)_{L^2(\Omega)}ds\nonumber\\
&=\int_0^{+\infty} \alpha \varepsilon ((v+\tilde{u})(v+\tilde{u})_{tt}+[(v+\tilde{u})_t]^2,\phi)_{L^2(\Omega)}ds
 \end{align*}
 with $v(0)=0$ and $\partial_t v(0)=0$.

\end{theorem}
\begin{proof}
As in the proof of Theorem \ref{ThWPWestGlob} we work with the Banach spaces 
$X_0(\Omega)$ 
and $Y=L^2(\R^+;L^2(\Omega))$. Then by Theorem~\ref{thmeqivcaudampwav}, the linear operator 
$$L:X_0(\Omega)\rightarrow Y,\quad  u\in X_0(\Omega)\mapsto\;L(u):=\del_t^2 u+c^2(-\Delta) u+\nu \varepsilon  \del_t (-\Delta) u\in Y$$
 is a bi-continuous isomorphism.
 
 Now we set $u^*\in X(\Omega)$, $\overline{g}\in X_1(\Omega)$ with $\Vert u^*\Vert_{X(\Omega)}\leq \frac{r}{2}$ and $\Vert \overline{g}\Vert_{X_1(\Omega)}\leq \frac{r}{2}$, and consider
 $$\Phi(v)=\alpha\varepsilon (v+u^*+\overline{g})\partial_t^2 (v+u^*+\overline{g})+\alpha \varepsilon [\partial_t (v+u^*+\overline{g})]^2.$$
 The properties of $X(\Omega)$ and $X_1(\Omega)$ allows to show for $w$ and $z$ in $X(\Omega)$ with $\Vert w\Vert_{X(\Omega)} \leq r$ and $\Vert z\Vert_{X(\Omega)} \leq r$ 
 \begin{align*}
\Vert \Phi(w)-\Phi(z)\Vert_Y\leq 
\alpha B \varepsilon r \Vert w-z\Vert_{X(\Omega)}
\end{align*}
with a constant $B>0$ depending only on $\Omega$. The final result follows as in the proof of Theorem~\ref{ThWPWestGlob} with the use of Theorem~\ref{thSuh}.
\end{proof}
Now we can give the result on the well-posedness of the Westervelt system (\ref{systWestnonhom}) on admissible domains:
\begin{theorem}
Let $u_0\in H^2(\Omega)$, $u_1\in H^1(\Omega)$, $g\in F$, space defined in (\ref{regtrspacedir}), and take
\\$f\in L^2([0,+\infty[;L^2(\Omega))$
with the compatibility conditions
$$g(0)=Tr_{\partial\Omega}u_0,\quad \partial_t g(0)=Tr_{\partial\Omega} u_1.$$ There exist $r^*>0$ and $C>0$ such that for $r\in [0,r^*[$
$$\Vert u_0\Vert_{H^2(\Omega)}+\Vert u_1\Vert_{H^1(\Omega)}+\Vert g\Vert_F+\Vert f\Vert_{L^2([0,+\infty[;L^2(\Omega))}\leq C r$$
implies that there exists a unique weak solution $u$ of the Westervelt system (\ref{systWestnonhom}) with non homogeneous Dirichlet boundary conditions is given by $u=v+u^*+\overline{g}$.
 Here $u^*\in X(\Omega)$ and $\overline{g}\in X_1(\Omega)$ (see (\ref{Hspace}) and (\ref{spacesolexttrace}) respectively for the definition of the spaces), such that $\tilde{u}=u^*+\overline{g}$ is the unique weak solution of the strongly damped problem (\ref{cauchydampwavdirnonhom}) and $v\in X(\Omega)$ is the unique solution of system (\ref{systWestbishom}). Moreover, the following estimates hold:
 $$\Vert u^*\Vert_{X(\Omega)}\leq r,\;\;\Vert \overline{g}\Vert_{X_1(\Omega)}\leq \frac{r}{2},\;\;\Vert v\Vert_{X(\Omega)}\leq \frac{r}{2}.$$
\end{theorem}
\begin{proof}
According to Theorem \ref{Globwelposvdir} we have $v\in X(\Omega)$ with $\Vert v\Vert_{X(\Omega)}\leq r$ solution of system (\ref{systWestbishom}) as soon as $\Vert u^*\Vert_{X(\Omega)}\leq \frac{r}{2}$ and $\Vert \overline{g}\Vert_{X_1(\Omega)}\leq \frac{r}{2}$ for $r\in [0,r^*[$ with $r^*>0$.
But according to Theorem \ref{strondampwavedirnonhom} if $\tilde{u}=u^*+\overline{g}$ is the unique weak solution of the strongly damped problem (\ref{cauchydampwavdirnonhom}), then we have the estimates
$$\Vert u^*\Vert_{X(\Omega)}\leq C (\Vert f\Vert_{[0,+\infty[;L^2(\Omega))}+\Vert  u_0\Vert_{H^2(\Omega)}+\Vert u_1\Vert_{H^1(\Omega)}+\Vert g\Vert_F)$$
and
$$\Vert \overline{g}\Vert_{X_1(\Omega)}\leq C \Vert g\Vert_F.$$
So there exists $C>0$ such that 
$$\Vert f\Vert_{L^2([0,+\infty[;L^2(\Omega))}+\Vert  u_0\Vert_{H^2(\Omega)}+\Vert u_1\Vert_{H^1(\Omega)}+\Vert g\Vert_F\leq C r$$
implies
$$\Vert u^*\Vert_{X(\Omega)}\leq \frac{r}{2},\hbox{ and }\Vert \overline{g}\Vert_{X_1(\Omega)}\leq \frac{r}{2}$$
which allows to conclude.
\end{proof}
\section{Asymptotic analysis for the Westervelt equation with homogeneous Dirichlet boundary condition.}\label{secasymptan}
\subsection{Approximated problems in $\mathbb{R}^n$ with $n=2$ or $3$.}\label{subsecasymptan}
Let $\Omega$ be an arbitrary bounded domain in $\mathbb{R}^3$ or in $\mathbb{R}^2$. 
\begin{definition}\label{convomegmR3}
Let $( \Omega_m)_{m \geq 0}$ be a sequence of open sets in $\mathbb{R}^3$. We say that $\Omega_m$ converges to $\Omega$, $\Omega_m\rightarrow \Omega$ if the following two conditions are satisfied
\begin{itemize}
\item For any compact $K\subset \Omega$ there is $m=m(K) \geq 0$ such that
$$ K\subset \Omega_m \hbox{ for all }m\geq m(K).$$
\item The sets $\Omega_m\setminus \Omega $ are bounded and when $m\rightarrow +\infty$
$$\lambda((\Omega\setminus \Omega_m)\cup (\Omega_m\setminus \Omega) )\rightarrow 0.$$
\end{itemize}
If $\Omega\subset\mathbb{R}^2$ we add the condition that for all $m\in \N^*$ domains $\Omega_m$ are  NTA domains characterized by constants $M$ and $r_0$ (see Definition~\ref{defNTA}).
\end{definition}
We consider the problem with homogeneous Dirichlet boundary conditions on an open set $\Omega$ associated to the Westervelt equation
\begin{equation}\label{Kuzeqrobin}
\left\lbrace
\begin{array}{l}
\partial^2_t u-c^2 \Delta u-\nu \varepsilon\Delta\partial_t u=\alpha\varepsilon \partial_t[u \partial_t u]+f \;\;on\;]0,+\infty[\times \Omega,\\
 u=0\;\;on\;]0,+\infty[\times \partial \Omega,\\
u(0)=u_0,\;\;\del_t u(0)=u_1 \;\;on\;\Omega
\end{array}\right.
\end{equation}
with the compatibility condition $$ u_0\vert_{\partial \Omega}=0\;\;\;\;\;\; u_1\vert_{\partial \Omega}=0.$$

Let $\Omega$ be an arbitrary bounded domain in $\mathbb{R}^3$ or  in $\R^2$ for which  there exists a sequence $( \Omega_m)_{m \geq 0}$ of arbitrary domains in $\R^3$ or of NTA domains,  uniformly characterized by geometrical constants $M$ and $r_0$, in $\mathbb{R}^2$ respectively, such that 
$\Omega_m\rightarrow \Omega$ in the sense of Definition~\ref{convomegmR3}. 
In addition we fix an open bounded set $\Omega^*$ such that $\Omega\subset \Omega^*$ and  $\Omega_m\subset\Omega^*$ for all $m\in \mathbb{N^*}$.

Thanks to Corollaries~\ref{PoissonR3} and~\ref{linfPoissonR2} as the domains $\Omega_m$ are uniformly bounded we can estimate the corresponding sequence of the weak solutions $u_m$ of the Poisson equation on $\Omega_m$
\begin{equation}\label{aprioestpoisprobpref}
 \Vert u_m\Vert_{H^1_0(\Omega_m)}\leq C\Vert f\Vert_{L^2(\Omega_m)} \quad \forall m\in \N^*
 \end{equation}
 with a constant $C>0$  independent of$m$.

For the case of the strongly damped wave equation on $\Omega_m$ %
as in Section~\ref{secwpdampwavdirhom} we can %
apply Theorems~\ref{ThExLinW} and~\ref{dampedwaveeqregint1} on $\Omega_m$. Nevertheless it is important to notice that  if we consider the associated estimates~(\ref{aprioriglobdampedwave})--(\ref{aprioriglobdampedwavebis}) on $\Omega_m$, the constants can be taken independent of$m$. It comes from the fact that the constant in the Poincar\'e's inequality depends only on the Lebesgue measure of the domain $\lambda(\Omega_m)$ which is bounded (by the volume of $\Omega^*$) and that there are no any other dependence on $\Omega_m$ which appears in the constants of the estimates~(\ref{aprioriglobdampedwave})--(\ref{aprioriglobdampedwavebis}). 

Moreover, for  the non linear Westervelt equation on $\Omega_m$ we have:

\begin{theorem}\label{ThWpWesprefcell}
Let $(\Omega_m)_{m\in \N^*}$ be the described previously sequence of domains, 
  $\nu>0$ and $\R^+=[0,+\infty[$. Considering the homogeneous Dirichlet boundary problem for the Westervelt equation on $\Omega_m$
\begin{equation}\label{CauchypbWesprefcell}
\left\lbrace
\begin{array}{l}
\partial^2_t u_{m}-c^2\Delta u_{m}-\nu \varepsilon \Delta \partial_t u_{m}=\alpha\varepsilon u_{m} \partial^2_t u_{m}+\alpha\varepsilon (\partial_t u_{m})^2+f\;\;\;in\;\;[0,+\infty[\times \Omega_m \\
u_{m}=0 \;\;\;on\;\;[0,+\infty[\times\partial\Omega_m, \\
u_{m}(0)=u_{0,m}, \hbox{ }\partial_t u_{m}(0)=u_{1,m},
\end{array}\right.
\end{equation} 
assume that $f\in L^2(\R^+;L^2(\Omega_m))$ and that the initial data
  $$u_{0,m}\in H^1_0(\Omega_m) \quad  \hbox{and}\quad u_{1,m}\in H^1_0(\Omega_m)$$
 with $\Delta u_{0,m} \in L^2(\Omega_m)$  in the weak sense of a solution of the Poisson equation. %
 Moreover,  
  let
  $C_1=O(1)$, independent of $m$ thanks to Theorem~\ref{dampedwaveeqregint1},  be the minimal constant such that the 
weak  solution $u_m^*\in X(\Omega_m)$ of the corresponding non homogeneous linear Cauchy problem (\ref{dampedwaveeqdirhom}) on $\Omega_m$ satisfies
  \begin{align*}  
  \Vert u_m^*\Vert_{X(\Omega_m)} \leq \frac{C_1}{\nu \eps}(\Vert f\Vert_{L^2(\R^+;L^2(\Omega_m))} +\Vert \Delta u_{0,m}\Vert_{L^2(\Omega_m)}+\Vert u_{1,m}\Vert_{H^1_0(\Omega_m)}).
  \end{align*}

 Then there exists $r_{*}>0$ independent of $m$ with $r_*=O(1)$ such that for all  $r\in[0,r_{*}[$
  and all 
  data  satisfying  
$$
\Vert f\Vert_{L^2(\R^+;L^2(\Omega_m))} +\Vert \Delta u_{0,m}\Vert_{L^2(\Omega_m)}+\Vert u_{1,m}\Vert_{H^1_0(\Omega_m)}\le \frac{\nu \eps}{C_1}r,
$$
there exists the unique weak solution $u_m\in X(\Omega_m)$ of the problem~(\ref{CauchypbWesprefcell}) for the Westervelt equation in same sense that in Theorem~\ref{ThWPWestGlob} and 
$$ \Vert u_m\Vert_{X(\Omega_m)}
   \leq 2r .$$
\end{theorem}
\begin{proof}
The proof is essentially the same as for Theorem~\ref{ThWPWestGlob} and is thus omitted. The independence on $m$ of $r^*$ comes from the independence on $m$ in the estimates ~(\ref{aprioriglobdampedwave})--(\ref{aprioriglobdampedwavebis}) and~(\ref{EqCREch2}) considered on $\Omega_m$.
\end{proof}
Note that we can apply Theorem~\ref{ThWPWestGlob} on $\Omega$ in $\R^3$ or if it is only a NTA domain in $\R^2$. In next subsection we extend the class of NTA domains to arbitrary domains which it is possible to approximate with a sequence of NTA domains with uniform geometrical constants.
\subsection{Mosco type convergence.}
In the assumptions introduced in the previous subsection we define 
\begin{align}
H(\Omega^*):=H^1([0,+\infty[;H^1_0(\Omega^*))\cap & H^2([0,+\infty[;L^2(\Omega^*)).\label{Moscospace}
\end{align}
Then for $u\in H(\Omega^*)$, $f\in L^2([0,+\infty[;L^2(\Omega^*))$ and $\phi\in L^2([0,+\infty[,H^1_0(\Omega^*))$ we introduce 
\begin{align}
F_m[u,\phi]:=&\int_0^{+\infty}\int_{\Omega_m} \partial_t^2 u  \phi+c^2 \nabla u \nabla \phi+\nu \varepsilon \nabla \partial_t u \nabla \phi\;d\lambda\;dt\nonumber\\
&\int_0^{+\infty}\int_{\Omega_m} -\alpha \varepsilon (u \partial_t^2 u) \phi -\alpha \varepsilon ( \partial_t u)^2 \phi+f\phi\;d\lambda\;dt\label{Fnu}
\end{align}
and also
\begin{align}
F[u,\phi]:=&\int_0^{+\infty}\int_{\Omega} \partial_t^2 u  \phi+c^2 \nabla u \nabla \phi+\nu \varepsilon \nabla \partial_t u \nabla \phi\;d\lambda\;dt\nonumber\\
&\int_0^{+\infty}\int_{\Omega} -\alpha \varepsilon (u \partial_t^2 u) \phi -\alpha \varepsilon ( \partial_t u)^2 \phi+f\phi\;d\lambda\;dt.\label{Fu}
\end{align}
We also define
 for $u\in L^2(\Omega^*)$
\begin{equation}\label{Fmbar}
\overline{F}_m[u,\phi]=\left\lbrace\begin{array}{ll}
F_m[u,\phi], & \hbox{ if } u\in H(\Omega^*),\\
+\infty, & otherwise
\end{array}\right.
\end{equation}
and
\begin{equation}\label{Fbar}
\overline{F}[u,\phi]=\left\lbrace\begin{array}{ll}
F[u,\phi], & \hbox{ if } u\in H(\Omega^*),\\
+\infty, & otherwise.
\end{array}\right.
\end{equation}
\begin{remark}\label{defweaksolutionKuz}
We  notice that $u$ is a weak solution of the Westervelt problem (\ref{Kuzeqrobin}) on $[0,+\infty[ \times \Omega $ in the sense of Theorem~\ref{ThWPWestGlob} if
\begin{itemize}
\item $u\in X(\Omega)$, space defined in (\ref{Hspace}).
\item For every $\phi \in L^2([0,+\infty[;H^1_0(\Omega))$
$F[u,\phi]=0,$
where $F$ defined in (\ref{Fu}).
\item $u(0)=u_0$ and $\del_t u(0)=u_1$ on $\Omega$.
\end{itemize}
\end{remark}
The expression $F[u,\phi]=0$ can be obtained multiplying the Westervelt equation from system (\ref{Kuzeqrobin}) by $\phi \in X(\Omega)$ integrating on $[0,+\infty[\times \Omega$ and doing integration by parts taking into account the boundary conditions.
In the same way with $F_m[u,\phi]$ defined in equation (\ref{Fnu}) we can define the weak solution of problem~(\ref{CauchypbWesprefcell}).

In order to state our main result, we also need to recall the notion of $M-convergence$ of functionals introduced in Ref.~\cite{Mosco}.
\begin{definition}
A sequence of functionals $G^m:H\rightarrow (-\infty,+\infty]$ is said to $M$-converge to a functional $G:H\rightarrow (-\infty,+\infty]$ in a Hilbert space $H$, if
\begin{enumerate}
\item(lim sup condition) For every $u\in H$ there exists $u_m$ converging strongly in $H$ such that
\begin{equation}
\overline{\lim} G^m[u_m]\leq G[u],\;\;\;\text{as}\;m\rightarrow +\infty.
\end{equation}
\item(lim inf condition) For every $v_m$ converging weakly to $u$ in $H$
\begin{equation}
\underline{\lim} G^m[v_m]\geq G[u],\;\;\;\text{as}\;m\rightarrow +\infty.
\end{equation}
\end{enumerate}
\end{definition}
The main result is the following theorem.
\begin{theorem}\label{Mconv}
For $\phi\in L^2([0,+\infty[;H^1_0(\Omega^*))$, the sequence of functionals $u\mapsto \overline{F}^m[u,\phi]$ defined in (\ref{Fmbar}), $M$-converges in $L^2([0,+\infty[;L^2(\Omega^*))$ to the following functional $u\mapsto \overline{F}[u,\phi]$ defined in (\ref{Fbar}) as $m\rightarrow +\infty$. More precisely in this case $\forall \phi\in L^2([0,+\infty[;H^1_0(\Omega^*))$ if $v_m\rightharpoonup u$ in $H(\Omega^*)$, the space defined in~(\ref{Moscospace}), then 
$$ F_m[v_m,\phi]\underset{m\rightarrow +\infty}{\longrightarrow}F[u,\phi].$$
\end{theorem}
\begin{proof}
We consider $\phi\in L^2([0,+\infty[;H^1_0(\Omega^*))$.
Given the definition of $\overline{F}$ and $\overline{F}^m$, we only consider the case where these functions take finite value.

\textit{Proof of "lim sup" condition.}
Without loss of generality, let us take directly a fixed $u\in H(\Omega^*)$ and define $v_m=u$ for all $m$. Hence $(v_m)_{m\in \N^*}$ is strongly converging sequence in $L^2([0,+\infty[;L^2(\Omega^*))$. Thus by the definition of functionals $\overline{F_m}[u,\phi]$ and $\overline{F}[u,\phi]$, they are equal respectively to 
$F_m[u,\phi]$ and $F[u,\phi]$, which are well defined (and hence are finite).

As $\Omega_m\rightarrow \Omega$ in the sense of Definition~\ref{convomegmR3}
 and $u \in H(\Omega^*)$ defined in Eq.~(\ref{Moscospace}),   to pass to the limit in~(\ref{Fnu}) we can directly apply the   
 dominated convergence theorem  for $m\rightarrow +\infty$ 
\begin{multline}
\int_0^{+\infty}\int_{\Omega_m} \partial_t^2 u  \phi+c^2 \nabla u \nabla \phi+\nu \varepsilon \nabla \partial_t u \nabla \phi\;d\lambda\;dt\\
\rightarrow \int_0^{+\infty}\int_{\Omega} \partial_t^2 u  \phi+c^2 \nabla u \nabla \phi+\nu \varepsilon \nabla \partial_t u \nabla \phi\;d\lambda\;dt,\label{limsupconv1}
\end{multline}
\begin{multline}
\int_0^{+\infty}\int_{\Omega_m} -\alpha \varepsilon (u \partial_t^2 u) \phi -\alpha \varepsilon ( \partial_t u)^2 \phi\;d\lambda\;dt\\
\rightarrow \int_0^{+\infty}\int_{\Omega} -\alpha \varepsilon (u \partial_t^2 u) \phi -\alpha \varepsilon ( \partial_t u)^2 \phi\;d\lambda\;dt.\label{limsupconv1bis}
\end{multline}
This comes from the fact that for $u\in H(\Omega^*)$ and $\phi\in L^2([0,+\infty[;H^1_0(\Omega^*))$ by H\"older's inequality we ensure the boundness
\begin{align*}
\int_0^{+\infty}\int_{\Omega^*}  \vert \partial_t^2 u  \phi\vert\;d\lambda\;dt
\leq & \Vert \partial_t^2 u \Vert_{L^2([0,+\infty[;L^2(\Omega^*))} \Vert \phi\Vert_{L^2([0,+\infty[;L^2(\Omega^*))}<+\infty,\\
\int_0^{+\infty}\int_{\Omega^*}  \vert c^2 \nabla u \nabla \phi\vert\;d\lambda\;dt
\leq & c^2 \Vert \nabla u \Vert_{L^2([0,+\infty[;L^2(\Omega^*))} \Vert\nabla \phi\Vert_{L^2([0,+\infty[;L^2(\Omega^*))}<+\infty,\\
\int_0^{+\infty}\int_{\Omega^*}  \vert \nu \varepsilon \nabla \partial_t u \nabla \phi\vert\;d\lambda\;dt
\leq & \nu \Vert \nabla \partial_t u \Vert_{L^2([0,+\infty[;L^2(\Omega^*))} \Vert \nabla\phi\Vert_{L^2([0,+\infty[;L^2(\Omega^*))}<+\infty,
\end{align*}
and, using also successively H\"older's inequality and the Sobolev embeddings, we also control the nonlinear terms
\begin{align*}
\int_0^{+\infty}\int_{\Omega^*} \vert (u & \partial_t^2 u) \phi\vert  \;d\lambda\;dt\\
\leq & \Vert u\Vert_{L^{\infty}([0,+\infty[;L^4(\Omega^*))}\Vert \partial^2_t u\Vert_{L^2([0,+\infty[;L^2(\Omega^*))} \Vert \phi\Vert_{L^2([0,+\infty[;L^4(\Omega^*))}\\
\leq & C \Vert u\Vert_{H^1([0,+\infty[;H^1_0(\Omega^*))}\Vert \partial^2_t u\Vert_{L^2([0,+\infty[;L^2(\Omega^*))} \Vert \phi\Vert_{L^2([0,+\infty[;H^1_0(\Omega^*))}<+\infty
\end{align*}
and 
\begin{align*}
\int_0^{+\infty}\int_{\Omega^*} \vert & ( \partial_t u)^2  \phi\vert\;d\lambda\;dt \\
\leq & \Vert \partial_t u\Vert_{L^{\infty}([0,+\infty[;L^2(\Omega^*))} \Vert \partial_t u\Vert_{L^{2}([0,+\infty[;L^4(\Omega^*))} \Vert \phi\Vert_{L^2([0,+\infty[;L^4(\Omega^*))}\\
\leq & C \Vert \partial_t u\Vert_{H^1([0,+\infty[;L^2(\Omega^*))} \Vert \partial_t u\Vert_{L^{2}([0,+\infty[;H^1_0(\Omega^*))} \Vert \phi\Vert_{L^2([0,+\infty[;H^1_0(\Omega^*))}<+\infty.
\end{align*}
Note that by the dominated convergence theorem we have for $f\in L^2([0,+\infty[;L^2(\Omega^*))$
$$\int_0^{+\infty}\int_{\Omega_m}f\phi \;d\lambda\;dt\underset{m\rightarrow +\infty}{\longrightarrow} \int_0^{+\infty}\int_{\Omega}f\phi \;d\lambda\;dt. $$
With the help of~(\ref{limsupconv1}) and~(\ref{limsupconv1bis}) we  conclude that for all $\phi\in L^2([0,+\infty[,H^1_0(\Omega^*))$
$$F_m[u,\phi]\underset{m\rightarrow +\infty}{\longrightarrow} F[u,\phi],$$
from where follows the "lim sup" condition.

\textit{Proof of the "lim inf" condition.}
Now, let $v_m$ be a sequence such that
$$v_m \rightharpoonup u\;\;\;\text{in}\;H(\Omega^*),$$
with $H(\Omega^*)$ defined in~(\ref{Moscospace}) and
$$H(\Omega^*)\hookrightarrow H^1([0,+\infty[;H^1_0(\Omega^*)).$$
Then we have
\begin{equation}\label{convdistr0}
\partial_t^2 v_m \rightharpoonup \partial_t^2 u\;\;\;\text{in}\; L^2([0,+\infty[;L^2(\Omega^*)),
\end{equation}
\begin{equation}\label{convdistr1}
\partial_t v_m \rightharpoonup \partial_t u,\; \nabla\partial_t v_m \rightharpoonup \nabla\partial_t u \;\;\;\text{in}\; L^2([0,+\infty[;L^2(\Omega^*)),
\end{equation}
and
\begin{equation}\label{convdistr2}
v_m\rightharpoonup u,\hbox{ } \nabla v_m \rightharpoonup \nabla u \;\;\;\text{in}\;L^2([0,+\infty[;L^2(\Omega^*)).
\end{equation}
Moreover, working in $\R^n$ with dimension $n\leq 3$, by Theorem~\ref{thmsobolembadm} it is possible to chose any  $2 \leq p<6$ ensuring the compactness of the embedding $L^2([0,+\infty[;H^1(\Omega^*))\subset\subset L^2([0,+\infty[;L^p(\Omega^*))$. For higher dimension the desired assertion with $p\ge 2$ fails. So for $2 \leq p<6$
\begin{equation}\label{convdistr4}
 v_m\rightarrow u,\hbox{ }\partial_t v_m \rightarrow \partial_t u \hbox{ in } L^2([0,+\infty[;L^p(\Omega^*)).
\end{equation}
Let $\phi\in L^2([0,\infty[,H^1_0(\Omega))$, we want to show that 
$$F_m[v_m,\phi] \underset{m\rightarrow+\infty}{\longrightarrow} F[u,\phi].$$
We start to study the convergence of linear terms 
\begin{multline*}
\Big\vert \int_0^{+\infty}\int_{\Omega_m} \partial_t^2 v_m  \phi \;d\lambda\;ds- \int_0^{+\infty}\int_{\Omega} \partial_t u \partial_t \phi \;d\lambda\;ds\Big\vert \leq\\
\Big\vert \int_0^{+\infty}\int_{\Omega^*} \partial_t^2 v_m \mathbbm{1}_{\Omega_m}  \phi\;d\lambda\;ds-\int_0^{+\infty} \int_{\Omega^*} \partial_t^2 v_m \mathbbm{1}_{\Omega}\phi\;d\lambda\;ds\Big\vert\\
+\Big\vert \int_0^{+\infty}\int_{\Omega^*} \partial_t^2 v_m \mathbbm{1}_{\Omega}  \phi\;d\lambda\;ds-\int_0^{+\infty} \int_{\Omega^*} \partial_t^2 u \mathbbm{1}_{\Omega} \phi\;d\lambda\;ds\Big\vert.
\end{multline*}
The second term on the right hand side tends to zero as $m\rightarrow+\infty$ by (\ref{convdistr0}) as $\mathbbm{1}_{\Omega}\phi\in L^2([0,+\infty[;L^2(\Omega^*))$.
For the first term
$$\Big\vert \int_0^{+\infty}\int_{\Omega^*} \partial_t^2 v_m (\mathbbm{1}_{\Omega_m}-\mathbbm{1}_{\Omega}) \phi\;d\lambda\;ds\Big\vert\leq \Vert (\mathbbm{1}_{\Omega_m}-\mathbbm{1}_{\Omega})\phi\Vert_{L^2([0,+\infty[\times\Omega^*) }\Vert \partial_t^2 v_m\Vert_{L^2([0,+\infty[\times\Omega^*)},$$
but $\Vert \partial_t v_m\Vert_{L^2([0,T]\times\Omega^*)}$ is bounded according to~(\ref{convdistr0}), and hence by the dominated convergence theorem we find
$$\Vert (\mathbbm{1}_{\Omega_m}-\mathbbm{1}_{\Omega})\phi\Vert_{L^2([0,+\infty[\times\Omega^*) }\underset{m\rightarrow+\infty}{\longrightarrow} 0.$$
Then for $m\rightarrow+\infty$
$$\int_0^{+\infty}\int_{\Omega_m} \partial_t^2 v_m  \phi \;d\lambda\;ds\rightarrow \int_0^{+\infty}\int_{\Omega} \partial_t^2 u \partial_ \phi \;d\lambda\;ds.$$
 Using~(\ref{convdistr1}) and~(\ref{convdistr2}) we can deduce in the same way that
\begin{multline}
\int_0^{+\infty}\int_{\Omega_m} \partial_t^2 v_m  \phi+c^2 \nabla v_m \nabla \phi+\nu \varepsilon \nabla\partial_t v_m \nabla \phi\; d\lambda\;dt\\
\underset{m\rightarrow+\infty}{\longrightarrow} \int_0^{+\infty}\int_{\Omega} \partial_t^2 u  \phi+c^2 \nabla u \nabla \phi+\nu \varepsilon \nabla \partial_t u \nabla \phi \;d\lambda\;dt.\label{liminfconv1}
\end{multline}
For the quadratic terms we have
\begin{multline}
\left\vert \int_0^{+\infty}\int_{\Omega_m} (v_m \partial_t^2 v_m)  \phi d\lambda dt-\int_0^{+\infty}\int_{\Omega} (u \partial_t^2 u)  \phi d\lambda dt\right\vert \leq \\
\left\vert \int_0^{+\infty}\int_{\Omega_m} (v_m \partial_t^2 v_m) \phi d\lambda dt - \int_0^{+\infty}\int_{\Omega} (v_m \partial_t^2 v_m) \phi d\lambda dt\right\vert \\
+\left \vert\int_0^{+\infty}\int_{\Omega} ( v_m \partial_t^2 v_m)  \phi d\lambda dt -\int_0^{+\infty}\int_{\Omega} (u \partial_t^2 u) \phi d\lambda dt\right\vert.\label{liminfconvquadstep1}
\end{multline}
To show that the first term on the right hand side tends to $0$ for $m\rightarrow +\infty$, we firstly notice  that by H\"older's inequality
\begin{align*}
\left\vert \int_0^{+\infty}\right. &\int_{\Omega_m} (v_m \partial_t^2 v_m) \phi d\lambda dt -\left. \int_0^{+\infty}\int_{\Omega} (v_m \partial_t^2 v_m) \phi d\lambda dt\right\vert\\
\leq & \Vert (\mathbbm{1}_{\Omega_m}-\mathbbm{1}_{\Omega})\phi\Vert_{L^2([0,+\infty[;L^4(\Omega^*)) } \Vert v_m\Vert_{L^{\infty}([0,+\infty[;L^4(\Omega^*))}\Vert \partial_t^2 v_m \Vert_{L^2([0,+\infty[;L^2(\Omega^*))}.
\end{align*}
Then using the Sobolev embeddings we have the existence of a constant $K>0$ such that
\begin{align*}
\Vert v_m\Vert_{L^{\infty}([0,+\infty[;L^4(\Omega^*))} & \Vert \partial_t^2 v_m \Vert_{L^2([0,+\infty[;L^2(\Omega^*))}\\
\leq &  C \Vert v_m\Vert_{H^1([0,+\infty[;H^1_0(\Omega^*))}\Vert \partial_t^2 v_m \Vert_{L^2([0,+\infty[;L^2(\Omega^*))} \Vert\leq K,
\end{align*}
as $(v_m)_{m\in \N^*}$ is weakly convergent in $H(\Omega^*)$. Moreover, by the dominated convergence theorem combining with the Sobolev embeddings  
$$\phi\in L^2([0,+\infty[;H^1_0(\Omega^*))\hookrightarrow L^2([0,+\infty[;L^4(\Omega^*)),$$ 
and hence
$$\Vert (\mathbbm{1}_{\Omega_m}-\mathbbm{1}_{\Omega})\phi\Vert_{L^2([0,+\infty[;L^4(\Omega^*)) }\underset{m\rightarrow+\infty}{\longrightarrow} 0.$$
Therefore
\begin{equation}\label{liminfconvquadstep2}
\left\vert \int_0^{+\infty}\int_{\Omega_m} (v_m \partial_t^2 v_m) \phi d\lambda dt- \int_0^{+\infty}\int_{\Omega} (v_m \partial_t^2 v_m) \phi d\lambda dt\right\vert \underset{m\rightarrow+\infty}{\longrightarrow} 0.
\end{equation}
Now we consider 
$$\left \vert\int_0^{+\infty}\int_{\Omega} ( v_m \partial_t^2 v_m)  \phi d\lambda dt -\int_0^{+\infty}\int_{\Omega} (u \partial_t^2 u) \phi d\lambda dt\right\vert.$$
Noticing that
\begin{align*}
\Vert v_m \phi - u \phi\Vert_{L^2([0,+\infty[;L^2(\Omega))}^2=& \int_{0}^{+\infty}\Vert (v_m-u)\phi\Vert_{L^2(\Omega)}^2ds,\\
\intertext{thanks to the Young inequality it can be estimated as}
\Vert v_m \phi - u \phi\Vert_{L^2([0,+\infty[;L^2(\Omega))}^2 \leq & \int_{0}^{+\infty} \Vert v_m -u\Vert_{L^3(\Omega)}^2 \Vert \phi \Vert_{L^6(\Omega)}^2ds,
\intertext{and with the help of the Sobolev embeddings it becomes}
\Vert v_m \phi - u \phi\Vert_{L^2([0,+\infty[;L^2(\Omega))}^2  \leq & K \int_{0}^{+\infty} \Vert v_m -u\Vert_{H^1(\Omega)}^2 \Vert \phi \Vert_{H^1(\Omega)}^2ds\\
\leq & K \Vert v_m -u\Vert_{L^{\infty}([0,+\infty[; H^1(\Omega)}^2 \Vert \phi\Vert_{L^2([0,+\infty[;L^2(\Omega))}^2.
\end{align*}
But in the same time
$v_m\rightharpoonup u$ in $H^1([0,+\infty[;H^1(\Omega))\subset\subset L^{\infty}([0,+\infty[;H^1(\Omega))$, what implies the strong convergence $v_m\rightarrow u$ in $L^{\infty}([0,+\infty[;H^1(\Omega))$.
Then 
$$ v_m\phi \rightarrow u\phi \hbox{ in }L^2([0,+\infty[;L^2(\Omega)).$$
Combining this convergence result with~(\ref{convdistr0}) we obtain 
\begin{equation}\label{liminfconvquadstep3}
\left \vert\int_0^{+\infty}\int_{\Omega} ( v_m \partial_t^2 v_m)  \phi d\lambda dt -\int_0^{+\infty}\int_{\Omega} (u \partial_t^2 u) \phi d\lambda dt\right\vert\underset{m\rightarrow+\infty}{\longrightarrow} 0.
\end{equation}
Consequently, Eqs.~(\ref{liminfconvquadstep1}),~(\ref{liminfconvquadstep2}) and~(\ref{liminfconvquadstep3}) allow to conclude that
\begin{equation}\label{liminfconvquad1}
\left\vert \int_0^{+\infty}\int_{\Omega_m} (v_m \partial_t^2 v_m)  \phi d\lambda dt-\int_0^{+\infty}\int_{\Omega} (u \partial_t^2 u)  \phi d\lambda dt\right\vert \underset{m\rightarrow+\infty}{\longrightarrow} 0.
\end{equation}
Next let us consider
\begin{multline}
\left\vert \int_0^{+\infty}\int_{\Omega_m} ( \partial_t v_m)^2  \phi d\lambda dt-\int_0^{+\infty}\int_{\Omega} ( \partial_t u)^2  \phi d\lambda dt\right\vert \leq \\
\left \vert \int_0^{+\infty}\int_{\Omega_m} ( \partial_t v_m)^2  \phi d\lambda dt- \int_0^{+\infty}\int_{\Omega} ( \partial_t v_m)^2  \phi d\lambda dt\right\vert\\
+\left\vert \int_0^{+\infty}\int_{\Omega} ( \partial_t v_m)^2  \phi d\lambda dt -\int_0^{+\infty}\int_{\Omega} ( \partial_t u)^2  \phi d\lambda dt\right\vert\label{liminfconvquad2step1}
\end{multline}
The first term goes to $0$ for $m\to +\infty$ by the same reason as in the proof of~(\ref{liminfconvquadstep2}), moreover we have:
\begin{align*}
\left\vert \int_0^{+\infty}\int_{\Omega} ( \partial_t v_m)^2  \phi d\lambda dt -\int_0^{+\infty}\int_{\Omega} ( \partial_t u)^2  \phi d\lambda dt\right\vert=& \left\vert \int_0^{+\infty}\int_{\Omega} (\partial_t v_m-\partial_t u)(\partial_t v_m+\partial_t u)\phi d\lambda dt\right\vert.
\end{align*}
By the Young inequality
\begin{align*}
\left\vert \int_0^{+\infty}\int_{\Omega}( ( \partial_t v_m)^2-( \partial_t u)^2 )  \phi d\lambda dt \right\vert\leq & \int_0^{+\infty} \Vert \partial_t v_m-\partial_t u\Vert_{L^3(\Omega)}\Vert \partial_t v_m+\partial_t u\Vert_{L^2(\Omega)}\Vert \phi\Vert_{L^6(\Omega)}dt
\end{align*}
and by the Sobolev embeddings and the Cauchy-Schwarz inequality
\begin{align*}
\left\vert \int_0^{+\infty}\int_{\Omega}( ( \partial_t v_m)^2-( \partial_t u)^2 )  \phi d\lambda dt \right\vert\leq  \Vert \partial_t v_m+&\partial_t u\Vert_{L^{\infty}([0,+\infty[;L^2(\Omega))}\\
 & \Vert\partial_t v_m-\partial_t u\Vert_{L^{2}([0,+\infty[;L^3(\Omega))}\Vert \phi\Vert_{L^2([0,+\infty[;H^1(\Omega))}.
\end{align*}
Thanks to~(\ref{convdistr4}), $\Vert\partial_t v_m-\partial_t u\Vert_{L^{2}([0,+\infty[;L^3(\Omega))}\underset{m\rightarrow+\infty}{\longrightarrow} 0$. As in addition
$$\partial_t v_m \rightharpoonup \partial_t u \hbox{ in } H^1([0,+\infty[;L^2(\Omega))\hookrightarrow L^{\infty}([0,+\infty[;L^2(\Omega)),$$ the norm
$\Vert \partial_t v_m+\partial_t u\Vert_{L^{\infty}([0,+\infty[;L^2(\Omega))}$ is bounded, and thus 
$$\left\vert \int_0^{+\infty}\int_{\Omega}( ( \partial_t v_m)^2-( \partial_t u)^2 )  \phi d\lambda dt \right\vert\underset{m\rightarrow+\infty}{\longrightarrow} 0.$$
Coming back to~(\ref{liminfconvquad2step1}), we finally obtain
\begin{equation}\label{liminfconvquad2}
\left\vert \int_0^{+\infty}\int_{\Omega_m} ( \partial_t v_m)^2  \phi d\lambda dt-\int_0^{+\infty}\int_{\Omega} ( \partial_t u)^2  \phi d\lambda dt\right\vert \underset{m\rightarrow+\infty}{\longrightarrow} 0.
\end{equation}

So, from~(\ref{liminfconv1}),~(\ref{liminfconvquad1}) and~(\ref{liminfconvquad2})  we deduce that for all $\phi\in L^2([0,+\infty[;H^1_0(\Omega^*))$ the functionals $F_m[v_m,\phi]\rightarrow F[u,\phi]$
for $m\rightarrow+\infty$, which finishes the proof.
\end{proof}

Finally we generalize the well-posedness result in $\R^2$ and obtain  an approximation result of the solution of the Westervelt equation on $\Omega$ by the solutions of the Westervelt equation on $\Omega_m$.

\begin{theorem}\label{thmconvR3}
Let $\Omega$ an open bounded domain approximated by a sequence of open domains $(\Omega_m)_{m\in \N^*}$ such that $\Omega_m\rightarrow \Omega$ in the sense of Definition~\ref{convomegmR3}, 
$i.e.$ $\Omega_m$ are arbitrary in $\R^3$ and are NTA-domains uniformly characterized by constants $M$ and $r_0$ in $\R^2$. Let in addition
$\Omega^*$ be an open domain, such that for all $m\in \N^*$ $\Omega_m\subset\Omega^*$ with $\Omega\subset \Omega^*$, and $f\in L^2([0,+\infty[;L^2(\Omega^*))$ whose  restrictions on $\Omega$ and $\Omega_m$ are denoted again by $f$.
Let $u_0\in H^1_0(\Omega)$, $u_1\in H^1_0(\Omega)$, $\Delta u_0\in L^2(\Omega) $ in the sense  of the weak Poisson problem~(\ref{EqWeekDirPois}). Assume the existence of two sequences $u_{0,m}\in H^1_0(\Omega_m)$ and $u_{1,m}\in H^1_0(\Omega_m)$ with $ \Delta u_{0,m}\in L^2(\Omega_m)$ (also in the sense of~(\ref{EqWeekDirPois})) such that  their extensions on $\mathbb{R}^n$ 
by $0$ satisfy
\begin{align*}
E_{\mathbb{R}^n}\Delta u_{0,m}\underset{m\rightarrow+\infty}{\longrightarrow} E_{\mathbb{R}^n} \Delta u_0 \hbox{ in }L^2(\mathbb{R}^n), \\
E_{\mathbb{R}^n} u_{1,m} \underset{m\rightarrow+\infty}{\longrightarrow} E_{\mathbb{R}^n} u_1\hbox{ in }H^1_0(\mathbb{R}^n).
\end{align*}
Let us denote by $u_m\in X(\Omega_m)$ the  weak solutions  of the problem~(\ref{CauchypbWesprefcell}) on $\Omega_m$ associated to the initial conditions $u_{0,m}$ and $u_{1,m}$ in the sense of Theorem~\ref{ThWpWesprefcell}, which are also the weak  solutions in the sense of Remark~\ref{defweaksolutionKuz}.

Then there exists an unique weak solution $u\in H(\Omega)$  of the Westervelt problem~(\ref{Kuzeqrobin}) on $\Omega$ in the sense of the weak formulation~(\ref{weakformWes}). In addition, %
for  $r_{*}>0$  with $r_*=O(1)$ and $C_1>0$,  both existing by Theorem~\ref{ThWpWesprefcell} for the sequence $(u_m)_{m\in \N^*}$, there exists a constant $C>0$, depending on the volume of $\Omega$, such that for all  $r\in[0,r_{*}[$ 
$$\Vert f\Vert_{L^2(\R^+;L^2(\Omega))} +\Vert \Delta u_{0}\Vert_{L^2(\Omega)}+\Vert u_{1}\Vert_{H^1_0(\Omega)}\le \frac{\nu \eps}{C_1}r \; \Rightarrow \;  \Vert u\Vert_{H(\Omega)}\leq 2Cr. $$
Moreover, there is a subsequence of
the extensions of $u_m$ by $0$ converging weakly to the extension by $0$ of~$u$:
$$(E_{\mathbb{R}^n}u_{m_k})\vert_{\Omega^*}\rightharpoonup (E_{\mathbb{R}^n}u)\vert_{\Omega^*} \hbox{ in } H(\Omega^*).$$
If $\Omega\subset \R^3$ or a NTA domain in $\R^2$, then the solution $u\in X(\Omega)$ and satisfies Theorem~\ref{ThWPWestGlob}. 
\end{theorem}
\begin{proof}
By definition of $\Omega^*$ we have $H(\Omega_m)\hookrightarrow H(\Omega^*)$ and $H(\Omega)\hookrightarrow H(\Omega^*)$ (see Eq.~(\ref{Moscospace})).
By the definition of $u_m$ 
in Theorem~\ref{ThWpWesprefcell} 
we have as a direct consequence that $u_m\in X(\Omega_m)$ 
is a weak solution in the sense of Remark~\ref{defweaksolutionKuz}. Therefore, for all $\phi\in L^2([0,+\infty[;H^1_0(\Omega_m))$ %
$F_m[u_m,\phi]=0. $ %
Extending by $0$ we obtain
$$\Vert (E_{\mathbb{R}^n}u_m)\vert_{\Omega^*}\Vert_{H(\Omega_m)}\leq \Vert u_m\Vert_{H(\Omega_m)}.$$

But, thanks to the independence on $m$ of $r^*$ in Theorem~\ref{ThWpWesprefcell}, if $r<r^*$ and
$$
\Vert f\Vert_{L^2(\R^+;L^2(\Omega_m))} +\Vert \Delta u_{0,m}\Vert_{L^2(\Omega_m)}+\Vert u_{1,m}\Vert_{H^1_0(\Omega_m)}\le \frac{\nu \eps}{C_1}r
$$
with $C_1>0$ also independent of$m$, then it holds
$$C \Vert u_m\Vert_{H(\Omega_m)} \leq \Vert u_m\Vert_{X(\Omega_m)}\leq 2r$$
with  a constant $C>0$ depending only on $\lambda(\Omega_m) $.

Therefore, as the sequence $((E_{\mathbb{R}^n}u_m)\vert_{\Omega^*})_{m\in \N^*}$ is bounded, there exits $u^*$ in $ H(\Omega^*)$ and a subsequence, still denoted by $(E_{\mathbb{R}^n}u_m)\vert_{\Omega^*}$, such that
$$(E_{\mathbb{R}^n}u_m)\vert_{\Omega^*}\rightharpoonup u^*\quad \hbox{in}\quad H(\Omega^*).$$
Set $\phi\in L^2([0,+\infty[;\mathcal{D}(\Omega))$,
then $\phi\in L^2([0,+\infty[;H^1_0(\Omega))$, and from a certain rank $M$ for all $m\geq M$ $\operatorname{supp}(\phi)\subset(\Omega_m)$, which implies that $\phi\in L^2([0,+\infty[;H^1_0(\Omega_m))$ for $m\geq M$ as $\Omega_m\rightarrow\Omega$ in the sense of Definition~\ref{convomegmR3}.
Therefore, by Theorem~\ref{Mconv}  for $m\geq M$ we have
$$0=F_m[(E_{\mathbb{R}^n}u_m)\vert_{\Omega^*},\phi]\rightarrow F[u^*,\phi].$$
Consequently, for all $\phi\in L^2([0,+\infty[;\mathcal{D}(\Omega))$ 
$$F[u^*,\phi]=0$$
  which by the density argument holds also for all $\phi\in L^2([0,+\infty[;H^1_0(\Omega))$. By definition of $u_m$ we also have $u^*(0)=u_0$, $\Delta u^*(0)=\Delta u_0$ in $L^2(\Omega)$ and $\partial_t u^*(0)=u_1$ in $H^1_0(\Omega)$. 
Moreover, 
$$(E_{\mathbb{R}^n}u_m)\vert_{\Omega^*}\underset{m\rightarrow +\infty}{\rightharpoonup} u^*\hbox{ in } H(\Omega^*).$$

But $(E_{\mathbb{R}^n}u_m)\vert_{\Omega^*\setminus \Omega_m}=0$ and $\Omega_m \underset{m\rightarrow +\infty}{\rightarrow}\Omega$, so
$u^*\vert_{\Omega^*\setminus \Omega}=0.$ 
Consequently we obtain
$$u^*\vert_{\Omega}\in H(\Omega) $$
and thus it is the weak solution of the Westervelt problem~(\ref{Kuzeqrobin}) on $\Omega$ satisfying~(\ref{weakformWes}). It is unique by the unicity of the weak limit.
In addition, by the assumptions of strong convergence of the initial data and their uniform boudness, we obtain by passing to the limit the same boudness for $u_0$, $u_1$ and $f$ on $\Omega$. Finally, as $E_{\mathbb{R}^n} u^*\vert_{\Omega^*}$ is the weak limit of $E_{\mathbb{R}^n} u_m\vert_{\Omega^*}$ in $H(\Omega^*)$, we directly have using~\eqref{EqCREch2}, holding for arbitrary domains in $\R^3$ and for NTA domains in $\R^2$ with uniform on $m$ constants, that there exists a constant $C>0$ (depending on the volume of $\Omega$) such that for all $r\in [0,r^*[$
$$\|E_{\mathbb{R}^n} u^*\vert_{\Omega^*}\|_{H(\Omega^*)}\le  C\underline{\lim}\Vert E_{\mathbb{R}^n} u_m\vert_{\Omega^*}\Vert_{X(\Omega_m)}\leq 2Cr.$$
If $\Omega\subset \R^3$ or it is a NTA domain in $\R^2$ then by Theorem~\ref{ThWPWestGlob} there exists a unique weak solution $u\in X(\Omega)$ satisfying ~(\ref{weakformWes}) (see also Remark~\ref{defweaksolutionKuz}).
Thus, by the unicity, $u=u^*\vert_{\Omega}\in X(\Omega)$. If $\Omega\subset \R^2$ which is not a NTA domain a priori we cannot ensure that $u^*\in X(\Omega)$.
\end{proof}
\appendix
\section{Proof of Theorem ~\ref{preqthmNys} }\label{append1}

We note that by an approximation argument we just have to prove the theorem for $f\geq 0$ and $f\in \mathcal{D}(\Omega)$. Let $R_0$ be an initial cube of center $x_0\in \partial\Omega$ with an edge measure $l(R_0)=\delta$ and $R_1$ be a cube associated to $R_0$ of same center and orientation than $R_0$ with $l(R_1)=l(R_0)/(5C_0(M,n))$ with $C_0(M,n)>>1$ as in the Section~9 of Ref.~\cite{Nystromthese}. We take $R_2$ as a cube of same center and of the orientation as $R_1$ with side length $l(R_2)=l(R_1)/K_0$, where $K_0=K_0(M,n)$
is such that if $v(x)$ is an harmonic function on $R_1\cap\Omega$ vanishing continuously on $\partial\Omega\cap R_1$, then
\begin{equation}\label{eqproofNys1}
\int_{R_2\cap\Omega}\left\vert \frac{v(x)}{d(x,\partial\Omega)}\right\vert^q dx\leq K(M,n,q) \lambda(R_2)\left\vert \frac{v(y_0)}{d(y_0,\partial\Omega)}\right\vert^q
\end{equation}
for some $y_0\in R_2\cap \Omega$, $d(y_0,\partial\Omega)\geq K_1(M,n) l(R_2)$. This can always be arranged by Theorem 6.1 in Ref.~\cite{Nystromthese}. We construct $(Q_k)_{1\leq k\leq N}$ a covering of $\partial\Omega$ with cubes of same orientation and edge side than $R_2$ centered at $x_k\in \partial\Omega$ and such that $Q_k$ is related to a cube $Q_k^*$ in the same way that $R_2$ is related to $R_1$. Note that 
$$\sum_{k=1}^N \lambda(Q_k)\leq 2 \lambda(\Omega),$$
which implies $N\leq K_3(\delta,M,n,diam(\Omega)).$

Let $\Omega_0:=\Omega\setminus \cup_{i=1}^N Q_k$. We may assume that $d(\Omega_0,\partial\Omega)\geq K_1(M,n) l(R_2)$. Then
\begin{equation}\label{eqproofNys2}
\int_{\Omega}\left\vert\frac{Gf(x)}{d(x,\partial\Omega)}\right\vert^q dx \leq \int_{\Omega_0} \left\vert\frac{Gf(x)}{d(x,\partial\Omega)}\right\vert^q dx+ \sum_{k=1}^N \int_{Q_k\cap\Omega} \left\vert\frac{Gf(x)}{d(x,\partial\Omega)}\right\vert^q dx.
\end{equation}
We denote  by $G(x,y)$ the Green function associated to $\Omega$.  
According to Lemma~2.1 in Ref.~\cite{Widman} if $\lambda(\Omega)\leq 1$ then 
$$0 \leq G(x,y)\leq \frac{C(n)}{\vert x-y\vert^{n-1}}, $$
by dilatation we can extend this result for all $\Omega$ bounded simply connected
\begin{equation}\label{estgreenfunc}
0 \leq G(x,y)\leq \frac{C(n,\lambda(\Omega))}{\vert x-y\vert^{n-1}}. 
\end{equation}
Then as $f\geq 0$,
\begin{align*}
\vert Gf(x)\vert=& \int_{\Omega}G(x,y)f(y)\;dy\\
\leq & \int_{\Omega} C(n,\lambda(\Omega)) \frac{f(y)}{\vert x-y\vert^{n-1}}\;dy\\
\leq & C(n,\lambda(\Omega)) \gamma(1) I_1f(x)
\end{align*}
where we have taken for $\alpha=1$ the Riesz potential defined for $0<\alpha<n$
\begin{equation}\label{defRieszpot}
I_{\alpha}f(x)=\frac{1}{\gamma(\alpha)}\int_{\mathbb{R}^n} \frac{f(y)}{\vert x-y\vert^{n-\alpha}}\;dy.
\end{equation}
We  also use the fact according to Ref.~\cite{Stein} that for $0<\alpha<n$, $1<p\leq q <+\infty$, $\frac{1}{q}=\frac{1}{p}-\frac{\alpha}{n}$ then
\begin{equation}\label{apRieszpot}
\Vert I_{\alpha}f\Vert_{L^q}\leq A_{p,q}\Vert f\Vert_{L^p}.
\end{equation}
We obtain as $d(\Omega_0,\partial\Omega)\geq K_1(M,n) l(R_2)$
\begin{align*}
\int_{\Omega_0} \left\vert\frac{Gf(x)}{d(x,\partial\Omega)}\right\vert^q dx \leq & \frac{1}{K_1(M,n) l(R_2)}\int_{\Omega_0} \vert Gf(x) \vert^q \;dx\\
\leq & \frac{ (C(n,\lambda(\Omega)) \gamma(1))^q}{K_1(M,n) l(R_2)} \int_{\Omega_0} I_1f(x)^q\; dx\\
\leq & C(M,n,q,\delta, diam(\Omega))\Vert I_1 f\Vert_{L^q}^q\\
\leq & C(M,n,q,\delta, diam(\Omega)) \Vert f\Vert_{L^p}^q
\end{align*}
with $\frac{1}{q}=\frac{1}{p}-\frac{1}{n}$.

Now let us consider the second term in Eq.~(\ref{eqproofNys2}). We fix $k$ and take $\varphi\in \mathcal{D}(Q_k^*)$ with
\\$0\leq \varphi\leq 1 $ such that $\varphi(x)=1$ on $(1-\varepsilon) Q_k^*$ for some small $\varepsilon$. We put 
\\$g_k(x)=(1-\varphi(x))f(x)$ and $h_k(x)= \varphi(x) f(x)$. Then
\begin{equation}\label{eqproofNys3}
\int_{Q_k\cap\Omega} \left\vert\frac{Gf(x)}{d(x,\partial\Omega)}\right\vert^q dx \leq C_q \left( \int_{Q_k\cap\Omega} \left\vert\frac{G h_k(x)}{d(x,\partial\Omega)}\right\vert^q dx +\int_{Q_k\cap\Omega} \left\vert\frac{G g_k(x)}{d(x,\partial\Omega)}\right\vert^q dx\right).
\end{equation}
We call the integrals on the right hand side $I_1$ and $I_2$. Using Lemma~10.1 in Refs.~\cite{Nystromthese} or~\cite{Nystrom} as $\operatorname{supp} \;h_k\subset Q_k^*$, we have 
$$C=C(M,n,q,r_0,\delta,\operatorname{dimloc}(\partial\Omega),diam(\Omega))$$
such that
$$I_1\leq C \Vert h_k\Vert_{L^p}^q \leq C \Vert f\Vert_{L^p}^q.$$
Using~(\ref{eqproofNys1}) on $G g_k$ considering it is  harmonic on $(1-\varepsilon) Q_k^*$ we have $y_k\in \Omega\cap Q_k$ such that, $d(y_k,\partial\Omega)\geq K_1(M,n) l(R_2)$ and 
$$I_2\leq K(M,n,q) \lambda(R_2)\left\vert \frac{G g_k(y_k)}{d(y_k,\partial\Omega)}\right\vert^q\leq C(M,n,q,\delta) \vert G g_k(y_k)\vert^q.$$
But as $g_k\geq 0$
\begin{align*}
\vert G g_k(y_k)\vert= & \int_{\Omega} G(y_k,y) g_k(y)\;dy \\
\leq & C(n,diam(\Omega)) \int_{\Omega} \frac{g_k(y)}{\vert y_k-y\vert^{n-1}}\;dy
\end{align*}
according to~(\ref{estgreenfunc}). But $supp\;g_k\subset Q_k^*\setminus(1-\varepsilon)Q_k^*$ and if $y \in Q_k^*\setminus(1-\varepsilon)Q_k^*$ then 
$d(y_k,y)\geq \frac{1}{2} (1-\varepsilon)l(Q_k^*)-l(Q_k)$ a non negative constant depending only on $\delta,M,n$ given the relation between $Q_k$ and $Q_k^*$ and their definition.
As a result 
$$\int_{\Omega} \frac{g_k(y)}{\vert y_k-y\vert^{n-1}}\;dy\leq C(M,n,\delta) \int_{\Omega} g_k(y)\;dy\leq C(M,n,\delta) \Vert g_k\Vert_{L^1}\leq C(M,n,\delta) \Vert f\Vert_{L^1} $$
and 
$$I_2\leq C(m,n,q,\delta,diam(\Omega))\Vert f\Vert_{L^1}^q.$$
Considering again~(\ref{eqproofNys3}) we have
$$
\int_{Q_k\cap\Omega} \left\vert\frac{Gf(x)}{d(x,\partial\Omega)}\right\vert^q dx \leq C(M,n,q,r_0,\delta,\operatorname{dimloc}(\partial\Omega),diam(\Omega)) \Vert f\Vert_{L^p}^q,$$
which allows to conclude.

\def\refname{References}
\bibliographystyle{unsrt}
\bibliography{ref}

\begin{thebibliography}{10}

\bibitem{Westervelt}
P.~J. Westervelt.
\newblock Parametric acoustic array.
\newblock {\em The Journal of the Acoustical Society of America},
  35(4):535--537, 1963.

\bibitem{DEKKERS-2020}
{A}. {D}ekkers, {A}. {R}ozanova {P}ierrat, and {A}. {T}eplyaev.
\newblock {M}ixed boundary valued problem for linear and nonlinear wave
  equations in domains with fractal boundaries.
\newblock {\em {S}ubmitted. {P}reprint hal-02514311}, 2020.

\bibitem{DEKKERS-2020-1}
{A}. {D}ekkers, {A}. {R}ozanova {P}ierrat, and {V}. {K}hodygo.
\newblock {M}odels of nonlinear acoustics viewed as approximations of the
  {K}uznetsov equation.
\newblock {\em {D}iscrete and {C}ontinuous {D}ynamical {S}ystems-{A}}, 40,
  2020.

\bibitem{DEKKERS-2019}
{A}. {D}ekkers.
\newblock {\em {\NoAutoSpaceBeforeFDP} {{}M}athematical analysis of the
  {K}uznetsov equation : {C}auchy problem, approximation questions and problems
  with fractals boundaries{\AutoSpaceBeforeFDP}}.
\newblock PhD thesis, 2019.

\bibitem{AANONSEN-1984}
{S}.~{I}. {A}anonsen, {T}. {B}arkve, {J}.~{N}. {T}jötta, and {S}. {T}jötta.
\newblock {D}istortion and harmonic generation in the nearfield of a finite
  amplitude sound beam.
\newblock {\em {T}he {J}ournal of the {A}coustical {S}ociety of {A}merica},
  75(3):749--768, {M}ar 1984.

\bibitem{ROZANOVA-PIERRAT-2015}
{A}. {R}ozanova {P}ierrat.
\newblock {A}pproximation of a compressible {N}avier-{S}tokes system by
  non-linear acoustical models.
\newblock In {\em 2015 {D}ays on {D}iffraction ({D}{D})}. {I}{E}{E}{E}, {M}ay
  2015.

\bibitem{Shibata}
Y.~Shibata.
\newblock On the rate of decay of solutions to linear viscoelastic equation.
\newblock {\em Math. Methods Appl. Sci.}, 23(3):203--226, 2000.

\bibitem{EVANS-1994}
{L}.~{C}. {E}vans.
\newblock {\em {P}artial {D}ifferential {E}quations}.
\newblock {G}raduate {S}tudies in {M}athematics, 1994.

\bibitem{Kalt3}
B.~Kaltenbacher and I.~Lasiecka.
\newblock Global existence and exponential decay rates for the {W}estervelt
  equation.
\newblock {\em Discrete Contin. Dyn. Syst. Ser. S}, 2(3):503--523, 2009.

\bibitem{Kalt2}
B.~Kaltenbacher and I.~Lasiecka.
\newblock Well-posedness of the {W}estervelt and the {K}uznetsov equation with
  nonhomogeneous {N}eumann boundary conditions.
\newblock {\em Discrete Contin. Dyn. Syst. Ser. A}, (Dynamical systems,
  differential equations and applications. 8th AIMS Conference. Suppl. Vol.
  II):763--773, 2011.

\bibitem{Kalt1}
B.~Kaltenbacher and I.~Lasiecka.
\newblock An analysis of nonhomogeneous {K}uznetsov's equation: local and
  global well-posedness; exponential decay.
\newblock {\em Math. Nachr.}, 285(2-3):295--321, 2012.

\bibitem{Kaltwer}
B.~Kaltenbacher, I.~Lasiecka, and S.~Veljovi\'c.
\newblock Well-posedness and exponential decay for the {W}estervelt equation
  with inhomogeneous {D}irichlet boundary data.
\newblock In {\em Parabolic problems}, volume~80 of {\em Progr. Nonlinear
  Differential Equations Appl.}, pages 357--387. Birkh\"auser/Springer Basel
  AG, Basel, 2011.

\bibitem{Meyer}
S.~Meyer and M.~Wilke.
\newblock Global well-posedness and exponential stability for {K}uznetsov's
  equation in {$L_p$}-spaces.
\newblock {\em Evol. Equ. Control Theory}, 2(2):365--378, 2013.

\bibitem{Grisvard}
P.~Grisvard.
\newblock {\em Elliptic problems in nonsmooth domains}, volume~69 of {\em
  Classics in Applied Mathematics}.
\newblock Society for Industrial and Applied Mathematics (SIAM), Philadelphia,
  PA, 2011.
\newblock Reprint of the 1985 original [ MR0775683], With a foreword by Susanne
  C. Brenner.

\bibitem{EVANS-2010}
{L}.~{C}. {E}vans.
\newblock {\em {P}artial {D}ifferential {E}quations}.
\newblock {A}merican {M}ath {S}ociety, 2010.

\bibitem{ArendtElst}
W.~Arendt and A.~F.~M. ter Elst.
\newblock The {D}irichlet-to-{N}eumann operator on rough domains.
\newblock {\em J. Differential Equations}, 251(8):2100--2124, 2011.

\bibitem{EDMUNDS-1987}
{D}. {E}dmunds and {W}. {E}vans.
\newblock {\em {S}pectral theory and differential operators}.
\newblock {O}xford {M}ath. {M}onogr., {O}xford {U}niversity {P}ress, {O}xford,
  1987.

\bibitem{Nystrom}
K.~Nystr\"om.
\newblock Integrability of {G}reen potentials in fractal domains.
\newblock {\em Ark. Mat.}, 34(2):335--381, 1996.

\bibitem{Jerison}
D.~S. Jerison and C.~E. Kenig.
\newblock Boundary behavior of harmonic functions in nontangentially accessible
  domains.
\newblock {\em Adv. in Math.}, 46(1):80--147, 1982.

\bibitem{Xie}
W.~Xie.
\newblock A sharp pointwise bound for functions with {$L^2$}-{L}aplacians and
  zero boundary values of arbitrary three-dimensional domains.
\newblock {\em Indiana Univ. Math. J.}, 40(4):1185--1192, 1991.

\bibitem{Sukhinin}
M.~F. Sukhinin.
\newblock On the solvability of the nonlinear stationary transport equation.
\newblock {\em Teoret. Mat. Fiz.}, 103(1):23--31, 1995.

\bibitem{ROZANOVA-2004}
{A}.~{V}. {R}ozanova.
\newblock {C}ontrollability in a {N}onlinear {P}arabolic {P}roblem with
  {I}ntegral {O}verdetermination.
\newblock {\em {D}ifferential {E}quations}, 40(6):853--872, {J}un 2004.

\bibitem{ROZANOVA-2004-1}
{A}.~{V}. {R}ozanova.
\newblock {C}ontrollability for a {N}onlinear {A}bstract {E}volution
  {E}quation.
\newblock {\em {M}athematical {N}otes}, 76(3/4):511--524, {S}ep 2004.

\bibitem{ROSANOVA-2005}
{A}.~{V}. {R}osanova.
\newblock {L}etter to the {E}ditor.
\newblock {\em {M}athematical {N}otes}, 78(5-6):745--745, {N}ov 2005.

\bibitem{HAJLASZ-2008-1}
{P}. {H}ajłasz, {P}. {K}oskela, and {H}. {T}uominen.
\newblock {M}easure density and extendability of {S}obolev functions.
\newblock {\em {R}evista {M}atem{\'a}tica {I}beroamericana}, pages 645--669,
  2008.

\bibitem{ARFI-2017}
{K}. {A}rfi and {A}. {R}ozanova {P}ierrat.
\newblock {D}irichlet-to-{N}eumann or {P}oincar{\'e}-{S}teklov operator on
  fractals described by d-sets.
\newblock {\em {D}iscrete and {C}ontinuous {D}ynamical {S}ystems - {S}},
  12:1--26, 2019.

\bibitem{ROZANOVA-PIERRAT-2020}
{A}. {R}ozanova {P}ierrat.
\newblock {G}eneralization of {R}ellich-{K}ondrachov theorem and trace
  compacteness in the framework of irregular and fractal boundaries.
\newblock {\em {P}reprint hal-02489325}, 2020.

\bibitem{Raviart}
P.-A. Raviart and J.-M. Thomas.
\newblock {\em Introduction \`a l'analyse num\'{e}rique des \'{e}quations aux
  d\'{e}riv\'{e}es partielles}.
\newblock Collection Math\'{e}matiques Appliqu\'{e}es pour la Ma\^{i}trise.
  [Collection of Applied Mathematics for the Master's Degree]. Masson, Paris,
  1983.

\bibitem{Marschall}
J.~Marschall.
\newblock The trace of sobolev-slobodeckij spaces on lipschitz domains.
\newblock {\em manuscripta mathematica}, 58(1):47--65, Mar 1987.

\bibitem{JONSSON-1997}
{A}. {J}onsson and {H}. {W}allin.
\newblock {B}oundary value problems and brownian motion on fractals.
\newblock {\em {C}haos, {S}olitons \&{} {F}ractals}, 8(2):191--205, {F}eb 1997.

\bibitem{JONSSON-1994}
{A}. {J}onsson.
\newblock {B}esov spaces on closed subsets of $\mathbb{{R}}\sp n$.
\newblock {\em {T}ransactions of the {A}merican {M}athematical {S}ociety},
  341(1):355--370, {J}an 1994.

\bibitem{Lancia1}
M.~R. Lancia.
\newblock A transmission problem with a fractal interface.
\newblock {\em Z. Anal. Anwendungen}, 21(1):113--133, 2002.

\bibitem{BARDOS-2016}
{C}. {B}ardos, {D}. {G}rebenkov, and {A}. {R}ozanova {P}ierrat.
\newblock {S}hort-time heat diffusion in compact domains with discontinuous
  transmission boundary conditions.
\newblock {\em {M}ath. {M}odels {M}ethods {A}ppl. {S}ci.}, 26(01):59--110,
  {J}an 2016.

\bibitem{Mosco}
U.~Mosco.
\newblock Convergence of convex sets and of solutions of variational
  inequalities.
\newblock {\em Advances in Math.}, 3:510--585, 1969.

\bibitem{FEIREISL-2002-1}
{E}. {F}eireisl.
\newblock {S}hape {O}ptimization in {V}iscous {C}ompressible {F}luids.
\newblock {\em {A}pplied {M}athematics and {O}ptimization}, 47(1):59--78, {D}ec
  2002.

\bibitem{MAGOULES-2020}
{F}. {M}agoul{\`e}s, {T}. {P}.~{K}. {N}guyen, {P}. {O}mn{\`e}s, and {A}.
  {R}ozanova {P}ierrat.
\newblock {O}ptimal absorption of acoustical waves by a boundary.
\newblock {\em {S}ubmitted}, 2020.

\bibitem{HENROT-2005}
{A}. {H}enrot and {M}. {P}ierre.
\newblock {\em {V}ariation et optimization de formes. {U}ne analyse
  g{\'e}om{\'e}trique}.
\newblock {S}pringer, 2005.

\bibitem{CALDERON-1961}
{A}.-{P}. {C}alderon.
\newblock {L}ebesgue spaces of differentiable functions and distributions.
\newblock {\em {P}roc. {S}ymp. {P}ure {M}ath.}, 4:33--49, 1961.

\bibitem{STEIN-1970}
{E}.~{M}. {S}tein.
\newblock {\em {S}ingular integrals and differentiability properties of
  functions}.
\newblock {P}rinceton {U}niversity {P}ress, 1970.

\bibitem{JONES-1981}
{P}.~{W}. {J}ones.
\newblock {Q}uasi onformal mappings and extendability of functions in {S}obolev
  spaces.
\newblock {\em {A}cta {M}athematica}, 147(1):71--88, {D}ec 1981.

\bibitem{JONSSON-1984}
{A}. {J}onsson and {H}. {W}allin.
\newblock {\em {F}unction spaces on subsets of $\mathbb{{R}}^n$}.
\newblock {M}ath. {R}eports 2, {P}art 1, {H}arwood {A}cad. {P}ubl. {L}ondon,
  1984.

\bibitem{JONSSON-1995}
{A}. {J}onsson and {H}. {W}allin.
\newblock {T}he dual of {B}esov spaces on fractals.
\newblock {\em {S}tudia {M}athematica}, 112(3):285--300, 1995.

\bibitem{WALLIN-1991}
{H}. {W}allin.
\newblock {T}he trace to the boundary of {S}obolev spaces on a snowflake.
\newblock {\em {M}anuscripta {M}ath}, 73(1):117--125, {D}ec 1991.

\bibitem{TRIEBEL-1997}
{H}. {T}riebel.
\newblock {\em {F}ractals and {S}pectra. {R}elated to {F}ourier {A}nalysis and
  {F}unction {S}paces}.
\newblock {B}irkh{\"a}user, 1997.

\bibitem{Nystromal}
Jonas Azzam, Steve Hofmann, Jos\'{e}~Mar\'{\i}a Martell, Kaj Nystr\"{o}m, and
  Tatiana Toro.
\newblock A new characterization of chord-arc domains.
\newblock {\em J. Eur. Math. Soc. (JEMS)}, 19(4):967--981, 2017.

\bibitem{Gehring}
F.~W. Gehring.
\newblock {\em Characteristic properties of quasidisks}, volume~84 of {\em
  S\'eminaire de Math\'ematiques Sup\'erieures [Seminar on Higher
  Mathematics]}.
\newblock Presses de l'Universit\'e de Montr\'eal, Montreal, Que., 1982.

\bibitem{Vaisala}
J.~V\"ais\"al\"a.
\newblock {\em Lectures on {$n$}-dimensional quasiconformal mappings}.
\newblock Lecture Notes in Mathematics, Vol. 229. Springer-Verlag, Berlin-New
  York, 1971.

\bibitem{Ahlfors}
L.~V. Ahlfors.
\newblock Quasiconformal reflections.
\newblock {\em Acta Math.}, 109:291--301, 1963.

\bibitem{Jones2}
P.~W. Jones.
\newblock Extension theorems for {BMO}.
\newblock {\em Indiana Univ. Math. J.}, 29(1):41--66, 1980.

\bibitem{HAJLASZ-2008}
{P}. {H}ajłasz, {P}. {K}oskela, and {H}. {T}uominen.
\newblock {S}obolev embeddings, extensions and measure density condition.
\newblock {\em {J}ournal of {F}unctional {A}nalysis}, 254(5):1217--1234, {M}ar
  2008.

\bibitem{WALLIN-1982}
{H}. {W}allin.
\newblock {\em {M}arkov's inequality on subsets of $\mathbb{{R}}\sp n$}.
\newblock 6, {D}epartment of {M}ath., {U}niv. of {U}mea, 1982.

\bibitem{JONSSON-1984-1}
{A}. {J}onsson, {P}. {S}j{\"o}gren, and {H}. {W}allin.
\newblock {H}ardy and {L}ipschitz spaces on subsets of $\mathbb{{R}}^n$.
\newblock {\em {S}tudia {M}ath.}, 80:141--166, 1984.

\bibitem{JONSSON-2009}
{A}. {J}onsson.
\newblock {B}esov spaces on closed sets by means of atomic decomposition.
\newblock {\em {C}omplex {V}ariables and {E}lliptic {E}quations},
  54(6):585--611, {J}un 2009.

\bibitem{LIONS-1972}
{J}. {L}ions and {E}. {M}agenes.
\newblock {\em {N}on-{H}omogeneous {B}oundary {V}alue {P}roblems and
  {A}pplications}, volume~1.
\newblock {B}erlin: {S}pringer-{V}erlag, 1972.

\bibitem{MARSCHALL-1987}
{J}. {M}arschall.
\newblock {T}he trace of {S}obolev-{S}lobodeckij spaces on {L}ipschitz domains.
\newblock {\em {M}anuscripta {M}ath}, 58(1-2):47--65, {M}ar 1987.

\bibitem{JonssonBes}
Alf Jonsson.
\newblock Besov spaces on closed sets by means of atomic decomposition.
\newblock {\em Complex Var. Elliptic Equ.}, 54(6):585--611, 2009.

\bibitem{Nystromthese}
K.~Nystr\"om.
\newblock Smoothness properties of solutions to dirichlet problems in domains
  with a fractal boundary.
\newblock {\em Doctoral Thesis, University of Ume\"a, Ume\"a}, 1994.

\bibitem{Dahlberg}
B.~E.~J. Dahlberg.
\newblock {$L^{q}$}-estimates for {G}reen potentials in {L}ipschitz domains.
\newblock {\em Math. Scand.}, 44(1):149--170, 1979.

\bibitem{Roubicek}
Tom\'{a}\v{s} Roub\'{\i}\v{c}ek.
\newblock {\em Nonlinear partial differential equations with applications},
  volume 153 of {\em International Series of Numerical Mathematics}.
\newblock Birkh\"{a}user/Springer Basel AG, Basel, second edition, 2013.

\bibitem{DenkHiebPruss}
R.~Denk, M.~Hieber, and J.~Pr\"{u}ss.
\newblock Optimal {$L^p$}-{$L^q$}-estimates for parabolic boundary value
  problems with inhomogeneous data.
\newblock {\em Math. Z.}, 257(1):193--224, 2007.

\bibitem{Widman}
Kjell-Ove Widman.
\newblock Inequalities for the {G}reen function and boundary continuity of the
  gradient of solutions of elliptic differential equations.
\newblock {\em Math. Scand.}, 21:17--37 (1968), 1967.

\bibitem{Stein}
Elias~M. Stein.
\newblock {\em Singular integrals and differentiability properties of
  functions}.
\newblock Princeton Mathematical Series, No. 30. Princeton University Press,
  Princeton, N.J., 1970.

\end{thebibliography}

\end{document}